\documentclass[11pt]{article}
\usepackage{amsfonts}
\usepackage{amssymb,amsmath,amsthm, amsfonts}
\usepackage[colorlinks,linkcolor=blue,anchorcolor=blue,citecolor=blue]{hyperref}
\newcommand{\esssup}{\mathop{\mathrm{esssup}}}
\newcommand{\essinf}{\mathop{\mathrm{essinf}}}

 \allowdisplaybreaks

\newtheorem{theorem}{Theorem}[section]
\newtheorem{lemma}[theorem]{Lemma}
\newtheorem{definition}[theorem]{Definition}
\newtheorem{proposition}[theorem]{Proposition}
\newtheorem{corollary}[theorem]{Corollary}
\newtheorem{remark}[theorem]{Remark}

\numberwithin{equation}{section}

\begin{document}

\title{Nash equilibrium payoffs for stochastic  differential games
with jumps and  coupled nonlinear cost functionals }

\author{Qian Lin
\\
{\small  Center for Mathematical Economics, Bielefeld University,
Postfach 100131, 33501 Bielefeld,   Germany}\\{\small {\color{blue}
linqian1824@163.com}}}

 \maketitle

{\bf Abstract:}\hskip2mm In this paper we investigate Nash
equilibrium payoffs for   two-player nonzero-sum stochastic
differential games whose cost functionals are defined  by a system
of coupled backward stochastic differential equations. We obtain an
existence theorem and a characterization theorem for Nash
equilibrium payoffs. For this end the problem is described
equivalently by a stochastic differential game with jumps. But,
however, unlike the work by Buckdahn, Hu and Li \cite{BHL2010}, here
the important tool of a dynamic programming principle for stopping
times has to be developed. Moreover, we prove that the lower and
upper value functions
 are the viscosity solutions of the associated coupled systems of PDEs of  Isaacs type,
respectively.  Our results generalize those by Buckdahn,
Cardaliaguet and Rainer \cite{BCR2004} and by Lin \cite{L2011}.

\vskip2mm


\section{Introduction}
 The objective of this paper is to study Nash equilibrium
payoffs for  two-player nonzero-sum stochastic differential games
(SDGs, for short) with jumps and  coupled nonlinear cost
functionals. Since the pioneering paper of Fleming and Souganidis
\cite{FS1989}, SDGs have been studied by many authors. For instance,
recently, Buckdahn and Li \cite{BL2006} investigated zero-sum
two-player SDGs with nonlinear cost functionals  using a backward
stochastic differential equation (BSDE, for short) approach.  Unlike
Fleming and Souganidis \cite{FS1989}  they allow the controls to
depend on the past and prove with  a Girsanov transformation
argument that the priori random value functions are deterministic.
Buckdahn, Hu and Li \cite{BHL2010} extended the approach developed
in \cite{BL2006} to SDGs with jumps, while Biswas \cite{B2010}
investigated two-player zero-sum  SDGs with jump diffusion in the
framework  of Fleming and Souganidis \cite{FS1989}. The reader
interested in other approaches can be also referred to Hamad\`{e}ne
\cite{H1998}  and the references therein.

In nonzero-sum SDGs,  Hamad\`{e}ne, Lepeltier and Peng
\cite{HLP1997}  obtained the existence of a Nash equilibrium point
for nonzero sum SDGs with the help of BSDEs. Bessoussan and Frehse
\cite{BF2000} obtained   Nash equilibrium payoffs for SDGs by using
parabolic partial differential equations. But both methods rely
heavily on the assumption of the non degeneracy diffusion of the
coefficient and it is independent of controls. Buckdahn,
Cardaliaguet and Rainer \cite{BCR2004} got rid of the strong
assumptions on the diffusion coefficient. Lin \cite{L2011}
generalizes  the result in  \cite{BCR2004} by investigating  Nash
equilibrium payoffs for nonzero-sum SDGs with nonlinear cost
functionals.    Lasry and Lions \cite{LL2007} studied mean field
games, i.e., stochastic control of many agent systems where agents
are coupled via their costs.
 Motivated by the above results, we investigate  Nash equilibrium payoffs for SDGs
with coupled nonlinear cost functionals, i.e., the both players do
not only influence mutually their cost functionals in the choice of
their control processes, but also their gain processes.

In \cite{L2011}, the cost functionals of the both  players are
defined by a system of decoupled BSDEs, the both players  influence
mutually their cost functionals only by the choice of their control
processes.  An open problem was that how to study SDGs whose cost
functionals are defined by two coupled BSDEs, i.e, SDGs with coupled
nonlinear cost functionals. This is the objective of the paper.

 Let us be more precise now:  The dynamics of our two-player nonzero-sum SDG is given
 by the process  $ N^{t,i}$  and  the following doubly controlled
stochastic system:
  \begin{equation*}\label{}
  \left \{
  \begin{array}{llll}
  dX^{t,x ;u, v}_s & = & b(s,X^{t,x; u,v}_s, u_s, v_s) ds +
          \sigma(s,X^{t,x; u,v}_s, u_s, v_s) dB_s, \hskip1cm  s\in [t,T],\\
   X^{t,x ;u, v}_t  & = & x,
   \end{array}
   \right.
 \end{equation*}
where   $\{B_t\}_{t\geq 0}$ is a $d$-dimensional standard Wiener
process,  $\{N_t\}_{t\geq 0}$ is a Poisson process independent of
$\{B_t\}_{t\geq 0}$, and   $\mathbb{F}$ is the filtration generated
by $B$ and  $ N$. For $0\leq s\leq t\leq T, i=1,2,$   we let  $
N_{s}^{t,i}=m(i+N_{s}-N_{t})$, where $m(j)=1,$ if $j$ is odd, and
$m(j)=2,$ if $j$ is even.   The control  $u=\{u\}_{s\in[t,T]}$
(resp., $v=\{v\}_{s\in[t,T]}$) is supposed to be
$\mathbb{F}$-predictable and takes its values in a compact metric
space $U$ (resp., $V$). The set of these controls is denoted by
$\mathcal {U}_{t,T}$ (resp., $\mathcal {V}_{t,T}$). We shall give
its assumptions on $b$ and $\sigma$  in the next section.

We define our nonlinear cost functionals by introducing a system of
two coupled BSDEs:
    \begin{eqnarray}\label{s1}
   \left \{\begin{array}{rcl}
   -d\ ^{1}\widetilde{Y}_s & = & \widetilde{f}_{1}(s,X^{t,x; u, v}_s, \ ^{1}\widetilde{Y}_s,
   \ ^{2}\widetilde{Y}_s+\ ^{2}\widetilde{H}_s,\
   ^{1}\widetilde{Z}_s,u_s, v_s) ds-\ ^{1}\widetilde{Z}_s dB_s-
   \ ^{1}\widetilde{H}_s d N_{s},\\
      -d\ ^{2}\widetilde{Y}_s & = & \widetilde{f}_{2}(s,X^{t,x; u, v}_s, \ ^{1}\widetilde{Y}_s
      +\ ^{1}\widetilde{H}_s, \ ^{2}\widetilde{Y}_s,\
      ^{2}\widetilde{Z}_s,u_s, v_s) ds
   -\ ^{2}\widetilde{Z}_s dB_s-\ ^{2}\widetilde{H}_s d N_{s},\\
       ^{1}\widetilde{Y}_T  & = & \Phi_{1} (X^{t,x; u, v}_T),\
       ^{2}\widetilde{Y}_T  =  \Phi_{2} (X^{t,x; u,
       v}_T),\  s\in[t,T].
   \end{array}\right.
   \end{eqnarray}
 The assumptions on $\Phi_{i}$ and $ \widetilde{f}_{i},$
 $i=1,2,$ will be given in the next section. The  cost
functional for the $i^{th}$ player, $i=1,2,$ is defined by
\begin{equation*}\label{}
J_{i}(t,x;u,v):= \ ^{i}\widetilde{Y}^{t,x;u,v}_t,\qquad (t,x)\in
[0,T]\times\mathbb{R}^n,
\end{equation*}
where $(\ ^{i}\widetilde{Y}^{t,x;u,v},\
^{i}\widetilde{Z}^{t,x;u,v},\ ^{i}\widetilde{H}^{t,x;u,v}), i=1,2,$
is the unique solution of (\ref{s1}). Note   the special form of
$\widetilde{f}_{1}$ and $\widetilde{f}_{2}$, which is related with
our approach. The general case of $ \widetilde{f}_{i}$ not depending
on  $^{i}\widetilde{H}$ is still open.

In our framework, in opposite to zero-sum SDGs,  nonzero-sum SDGs
are of the type of "NAD strategy against NAD strategy": an NAD
strategy is  a measurable, nonanticipative  mapping $\alpha:\mathcal
{V}_{t,T}\rightarrow \mathcal {U}_{t,T}$ for the $1^{th}$ player
(resp., $\beta:\mathcal {U}_{t,T}\rightarrow \mathcal {V}_{t,T}$ for
the $2^{nd}$ player) and has a  delay (The definition will be
introduced in next section). The set of all such NAD strategies for
the $1^{th}$ player is denoted by $\mathcal {A}_{t,T}$ (resp., for
the $2^{nd}$ player   $\mathcal {B}_{t,T}$).

For $(\alpha,\beta)\in\mathcal {A}_{t,T}\times \mathcal {B}_{t,T}$,
 there exists a unique couple of
controls $(u,v)\in\mathcal {U}_{t,T}\times\mathcal {V}_{t,T}$ such
that  $(\alpha(v),\beta(u))=(u,v)$. This allows  to define
$J_{i}(t,x;\alpha,\beta):=J_{i}(t,x;u,v),$   as well as  the value
functions of the two-player zero-sum SDG associated with
$J_{i},i=1,2,$ the lower value function
\begin{eqnarray*}\label{}
W_{i}(t,x):=\esssup_{\alpha\in\mathcal {A}_{t,T}}
\essinf_{\beta\in\mathcal {B}_{t,T}} J_{i}(t,x;\alpha,\beta),
\end{eqnarray*}
and  the upper value function
\begin{eqnarray*}\label{}
U_{i}(t,x):=\essinf_{\beta\in\mathcal {B}_{t,T}}
\esssup_{\alpha\in\mathcal {A}_{t,T}} J_{i}(t,x;\alpha,\beta),
i=1,2.
\end{eqnarray*}
We note that, since the  BSDEs (\ref{s1}) are coupled, the values of
the two-player zero-sum SDGs are also coupled.

 In our approach we need a
probabilistic interpretation of coupled systems of
Hamilton-Jacobi-Bellman-Isaacs equations: A first result of our
paper is that   the value functions $U=(U_{1},U_{2})$ and
$W=(W_{1},W_{2})$ are  viscosity solutions of the following coupled
Isaacs equations:
 \begin{eqnarray*}
\left\{
\begin{array}{rcl}
\dfrac{\partial }{\partial t} U_{i}(t,x) +  H_{i}^{+}(t, x,
U_{1}(t,x), U_{2}(t,x),DU_{i}(t,x), D^2U_{i}(t,x))&=&0,
 \quad (t,x)\in [0,T)\times {\mathbb{R}}^n,\\
 U_{i}(T,x)&=&\Phi_{i} (x),  \ i=1,2,
 \end{array}
\right.
\end{eqnarray*}
and
\begin{eqnarray*}
\left\{
\begin{array}{rcl}
\dfrac{\partial }{\partial t} W_{i}(t,x) +  H_{i}^{-}(t, x,
W_{1}(t,x),W_{2}(t,x),DW_{i}(t,x), D^2W_{i}(t,x))&=&0,
\quad (t,x)\in [0,T)\times {\mathbb{R}}^n,\\
 W_{i}(T,x)&=&\Phi_{i}(x), \ i=1,2,
 \end{array}
\right.
\end{eqnarray*}
respectively,  where, for $ (t, x, y_{1}, y_{2}, p, A, u,v)\in [0,
T]\times{\mathbb{R}}^n\times\mathbb{R}\times \mathbb{R}\times
\mathbb{R}^{d}\times \mathbb{S}^{d}\times U\times V$,
\begin{eqnarray*}
H_{i}(t, x, y_{1}, y_{2}, p, A, u,v)&=&
\dfrac{1}{2}tr(\sigma\sigma^{T}(t, x, u, v) A)+ p b(t, x, u, v)+
\widetilde{f}_{i}(t, x, y_{1}, y_{2},  p \sigma(t, x, u, v), u, v),
\end{eqnarray*}
and
\begin{eqnarray*}
H_{i}^-(t, x, y_{1}, y_{2}, p, A)&=&\sup_{u \in U}\inf_{v \in
V}H_{i}(t, x, y_{1}, y_{2}, p, A, u,v),\\
 H_{i}^+(t, x, y_{1}, y_{2}, p, A)&=& \inf_{v \in V}\sup_{u \in U}H_{i}(t, x,
y_{1}, y_{2}, p, A, u,v).
\end{eqnarray*}
A crucial step in the proof of these results is to obtain  dynamic
programming principles for stopping times: i.e.,
 for any stopping time $\tau$ with
$0\leq t<\tau \leq T,\ x\in {\mathbb{R}}^n, i=1,2,$
\begin{eqnarray*}
W_{i}(t,x) &=&\esssup_{\alpha \in {\mathcal{A}}_{t,
\tau}}\essinf_{\beta \in {\mathcal{B}}_{t, \tau}}\
^{i}G^{t,x;\alpha,\beta}_{t,\tau} [W_{N^{t,i}_{\tau}}(\tau,
X^{t,x;\alpha,\beta}_{\tau})],\\
 U_{i}(t,x)&=&\essinf_{\beta
\in {\mathcal{B}}_{t, \tau}}\esssup_{\alpha \in {\mathcal{A}}_{t,
\tau}}\ ^{i}G^{t,x;\alpha,\beta}_{t,\tau} [U_{N^{t,i}_{\tau}}(\tau,
X^{t,x;\alpha,\beta}_{\tau})],
\end{eqnarray*}
where $^{i}G^{t,x;\alpha,\beta}_{t,\tau} [\cdot]$ is a backward
stochastic semigroup (the precise definition as well as those of
${\mathcal{A}}_{t, \tau}$ and  ${\mathcal{B}}_{t, \tau}$ will be
given later).

The most important part of the paper is dedicated to the  Nash
equilibrium payoffs of our games. A couple
$(e_{1},e_{2})\in\mathbb{R}^{2}$ is called a Nash equilibrium payoff
at the point $(t,x)$,  if for any $\varepsilon>0$, there exists
$(\alpha_{\varepsilon},\beta_{\varepsilon})\in \mathcal
{A}_{t,T}\times \mathcal {B}_{t,T}$ such that, for all
$(\alpha,\beta)\in \mathcal {A}_{t,T}\times \mathcal {B}_{t,T},$
\begin{eqnarray*}
J_{1}(t,x;\alpha_{\varepsilon},\beta_{\varepsilon})\geq
J_{1}(t,x;\alpha,\beta_{\varepsilon})-\varepsilon,\
J_{2}(t,x;\alpha_{\varepsilon},\beta_{\varepsilon})\geq
J_{2}(t,x;\alpha_{\varepsilon},\beta)-\varepsilon,\ \mathbb{P}-a.s.,
\end{eqnarray*}
and
\begin{eqnarray*}
|\mathbb{E}[J_{j}(t,x;\alpha_{\varepsilon},\beta_{\varepsilon})]-e_{j}|\leq
\varepsilon, \ j=1,2.
\end{eqnarray*}

Our model has some practical backgrounds in financial markets.  For
example, let us consider the following problem in a financial
market. There are two companies (players) in a financial market.  A
company 1 has invested money in bonds (paying dividends) of company
2, and company 2 in bonds of company 1, where the dividends are
payed in proportion with the gain of the corresponding company.
Therefore, the dynamics gain of company 1 has as one element of the
running gain the dividends payed by company 2, and company 2 has as
one element of the running gain the dividends payed by company 1.
Both companies try to maximize their payoff which can be different.
Since the financial market is not so quick in reacting to the moves
of both companies, both companies have to use  strategies with
delays. The above described problem  is a nonzero-sum stochastic
differential game.

 The main results of our paper concern  the  existence  and a
characterization of  Nash equilibrium payoffs for our games: We
first obtain the characterization of Nash equilibrium payoffs (see
Theorem \ref{Jt6}), and then get the existence of  a Nash
equilibrium payoff  (see Theorem \ref{Jt2}).

Let us explain what is new and  which difficulties are related with.
In comparison with \cite{BCR2004} and \cite{L2011}, the first
difficulty was to get a dynamic programming principle for a system
of two coupled BSDEs. To overcome this difficulty,   we associate
with this  system an auxiliary one  which cost functionals coincide
with ours.
   This leads to the new problem: we need  a  dynamic
programming principle  for this system not only for deterministic
but also for stopping times. The method used  in Buckdahn and Hu
\cite{BH2010} to get for control problems the dynamic programming
principle for stopping times is not applicable anymore, because in
the framework of SDGs the monotonicity argument used in
\cite{BH2010} doesn't work anymore. To overcome  this new
difficulty, we develop an argument to obtain the time continuity of
the value functions, which in return is used to obtain the dynamic
programming principle for stopping times from the the dynamic
programming principle for deterministic times.   Another technical
difficulty comes from the fact that we study here nonzero-sum SDGs
and not zero-sum SDGs. In order to give both players symmetric
tools, they
 have to use "strategies with delay against strategies with delay" and not only
"strategies against controls"  as in \cite{BL2006}. Finally,
comparing to our previous work \cite{L2011}, the presence of jump
terms adds a supplementary complexity.

Our paper is organized as follows. In Section \ref{NS1} we introduce
some notations and recall some basics of  BSDEs with jumps,  which
will be needed in what follows. Section \ref{NS2} introduces the
setting of SDGs and studies the dynamic programming principle for
stopping times.  Section \ref{NS3} gives a probabilistic
interpretation of coupled systems of Isaacs equations. In Section
\ref{NS4}  we investigate Nash equilibrium payoffs for nonzero-sum
SDGs. An existence theorem and a characterization theorem of  Nash
equilibrium payoffs are established.  Finally, we postpone   the
proof of the Theorems \ref{t5} and \ref{t1} to Section \ref{NS5}.

\section{Preliminaries}\label{NS1}

The objective of this section  is to give some preliminaries, which
will be useful in what follows. Let the underlying probability space
$(\Omega, {\cal{F}}, \mathbb{P})$ be the completed product of the
Wiener space $(\Omega_1, {\cal{F}}_1, \mathbb{P}_1)$\ and the
Poisson space $(\Omega_2, {\cal{F}}_2, \mathbb{P}_2).$ As concerns
the Wiener space $(\Omega_1, {\cal{F}}_1, \mathbb{P}_1)$:
$\Omega_1=C_0({\mathbb{R}};{\mathbb{R}}^d)$\ is the set of
continuous functions from ${\mathbb{R}}$\ to ${\mathbb{R}}^d$\ with
value zero at 0,  endowed with the topology generated by the uniform
convergence on compacts; $ {\cal{F}}_1 $\ is the  Borel
$\sigma$-algebra over $\Omega_1$, completed by
   the Wiener measure $\mathbb{P}_1$  under which the $d$-dimensional coordinate
processes $B_s(\omega)=\omega_s,\ s\in {\mathbb{R}}_+,\ \omega\in
\Omega_1,$\ and $B_{-s}(\omega)=\omega(-s),\ s\in {\mathbb{R}}_+,\
\omega\in \Omega_1,$\ are two independent $d$-dimensional Brownian
motions. We denote by $\{{\mathcal{F}}^{B}_s,\ s\geq 0\}$  the
natural filtration generated by $B$\ and augmented by all
$\mathbb{P}_1$-null sets, i.e.,
$${\mathcal{F}}^{B}_s=\sigma\Big\{B_r, r\in (-\infty, s]\Big\}\vee {\mathcal{N}}_{\mathbb{P}_1}, s\geq 0. $$

Let us now introduce the Poisson space $(\Omega_2, {\cal{F}}_2,
\mathbb{P}_2) $ as follows:

$$\Omega_{2}=\Big\{\omega_{2}=\sum\limits_{j\geq0}\delta_{t_{j}}, \{t_{j}\}_{j\geq 0}\subset \mathbb{R} \Big\},$$

 $$\mathcal {F}'=\sigma\Big\{N_{A}:
N_{A}(\omega_{2})=\omega_{2}(A), A\in\mathcal {B}(\mathbb{R})
\Big\},$$

\noindent and the Probability measure $\mathbb{P}_{2}$ can de
defined over $(\Omega_2, \cal{F}')$  such that $\{N_{t}\}_{t\geq 0}$
and $\{N_{-t}\}_{t\geq 0}$ are two independent Poisson processes
with intensity $\lambda.$  Let us denote  ${\cal{F}}_{2}$ by the
completion of $\cal{F}'$ with respect to the probability
$\mathbb{P}_{2}$ and
$$\dot{\cal F}_t^{N}=\sigma\Big\{N_{(-\infty,s]}:\,
-\infty<s\leq t \Big\},\ t\geq 0,$$ \noindent and  ${\cal
F}_t^{N}=\big(\bigcap\limits_{s>t}\dot{\cal F}^{N}_s\big)\vee{\cal
N}_{\mathbb{P}_2},\, t\geq 0$, augmented by the $\mathbb{P}_2$-null
sets. Moreover, we put $$\Omega=\Omega_1\times\Omega_2,\ {\cal
F}={\cal F}_1\otimes{\cal F}_2,\ \mathbb{P}=\mathbb{P}_1\otimes
\mathbb{P}_2,$$ where ${\cal F}$ is completed with respect to
$\mathbb{P}_{2}$, and the filtration ${\mathbb{F}}=\{{\cal
F}_t\}_{t\geq 0}$\ is generated by
 $${\cal F}_t:={\cal F}_t^{B,N}={\cal F}_t^B\otimes{\cal F}_t^N, \ \ t\geq 0,\ \
 \mbox{augmented by all $\mathbb{P}$-null sets}.$$

 Let $T>0$\ be an arbitrarily fixed time horizon. We denote by $\widetilde{N}_{t}=N_{t}-\lambda t,$ for all $t\geq 0.$
  For any
 $n\geq 1,$ we denote by $|z|$ the Euclidean norm of $z\in{\mathbb{R}}^{n}$.
 We  introduce the following spaces of stochastic processes.
\begin{eqnarray*}
&&\bullet\ L^2 (\Omega, \mathcal {F}_{T}, \mathbb{P};
\mathbb{R}^{n}) = \bigg\{ \xi\ |\  \xi:
\Omega\rightarrow\mathbb{R}^{n}\  \mbox {is an}  \ \mathcal {F}_{T}
\mbox {-measurable random variable such
that}\\ && \qquad \qquad\qquad\qquad\qquad \qquad\mathbb{E}[|\xi|^2]<+\infty \bigg\},\\
&&\bullet\ S^2 (0,T; \mathbb{R})=\bigg\{ \varphi\ |\
\varphi:\Omega\times[0, T]\rightarrow\mathbb{R} \ \mbox {is an}\
{\mathbb{F}}\mbox{-adapted c\`{a}dl\`{a}g
    process such that}\\ && \qquad \qquad\qquad\qquad\qquad \qquad \mathbb{E}[\sup\limits_{0\leq t\leq
T}|\varphi_t|^2]<+\infty \bigg\},\\
&&\bullet\ \mathcal {H}^2 (0,T; \mathbb{R}^{d}) =\bigg\{ \varphi\ |\
\varphi:\Omega\times[0, T]\rightarrow\mathbb{R}^{d}\ \mbox {is an} \
\ {\mathbb{F}}\mbox{-predictable  process such that} \\ && \qquad
\qquad\qquad\qquad\qquad \qquad
\mathbb{E}\int_0^T|\varphi_t|^2dt<+\infty \bigg\}.
\end{eqnarray*}
Let us  consider the following BSDE with data $(f,\xi)$:
\begin{equation}\label{e4}
y_t=\xi+\int_t^T f(s,y_s,z_s,k_{s})ds-\int_t^T z_sdB_s-\int_t^T
k_{s}d\widetilde{N}_{s},\quad 0\leq t\leq T.
\end{equation}
Here $f: \Omega \times [0,T]\times \mathbb{R} \times
\mathbb{R}^{d}\times\mathbb{R}\rightarrow \mathbb{R}$ is
$\mathbb{F}-$predictable and satisfies the following assumptions:

  $(H1)$ (Lipschitz condition): There exists a positive constant $C$  such
  that,  for all  $ (t,y_{i},z_{i},k_{i})\in [0,T]\times  \mathbb{R} \times \mathbb{R}^{d}
  \times\mathbb{R},$  $i=1, 2$,
  $$|f(t,y_{1},z_{1},k_{1})-f(t,y_{2},z_{2},k_{2})|\leq C(|y_{1}-y_{2}|+|z_{1}-z_{2}|+|k_{1}-k_{2}|).$$

  $(H2)$ $f (\cdot,0, 0, 0)\in \mathcal {H}^2 (0,T; \mathbb{R})$.\\

   $(H3)$  There exists a  constant $K>-1$  such
  that,  for all  $ (t,y,z,k_{1},k_{2})\in [0,T]\times  \mathbb{R} \times \mathbb{R}^{d}
  \times\mathbb{R}^{2},$
  $$f(t,y,z,k_{1})-f(t,y,z,k_{2})\geq K(k_{1}-k_{2}).$$

We note that a Poisson process is a special case of a Poisson random
measure. For this, we can  take as the compensator
$\nu(ds,de)=\lambda ds\delta_1(de),$  where
 \begin{eqnarray*}
\delta_1(x)=\left\{
\begin{array}{rcl}
1, \qquad x=1,\\
 0, \qquad   x\neq 1.
 \end{array}
\right.
\end{eqnarray*}We have the following existence  and uniqueness theorem of BSDE
(\ref{e4}). For its proof we refer the reader to  Tang and Li
\cite{TL1994}.
\begin{lemma}\label{lemma5}
Let the assumptions $(H1)$ and $(H2)$ hold. Then, for all $\xi\in
L^2 (\Omega, \mathcal {F}_{T}, \mathbb{P};\mathbb{R})$,  BSDE
(\ref{e4}) has a unique solution $(y,z,k)\in S^2 (0,T;
\mathbb{R})\times\mathcal{H}^2 (0,T; \mathbb{R}^{d})\times
\mathcal{H}^2 (0,T; \mathbb{R}).$
\end{lemma}
We have  the following  comparison theorem  for solutions of BSDEs
(\ref{e4}), which is proved with the help of standard arguments (see
Royer \cite{R2006}).
\begin{lemma}\label{l8}
Let us  denote by
 $(y^{1},z^{1},k^{1})$ and $(y^{2},z^{2},k^{2})$  the solutions of BSDEs
with data $(f^{1},\xi^{1})$ and $(f^{2},\xi^{2})$,
 respectively. Moreover, if  $\xi^{1}, \xi^{2}\in L^2 (\Omega, \mathcal {F}_{T},
\mathbb{P};\mathbb{R})$, and $f^{1}$ and $f^{2}$ satisfy the
assumptions $(H1)$, $(H2)$ and $(H3)$, and the following holds

  (i) $\xi^{1}\geq \xi^{2}$, $\mathbb{P}-a.s.,$

  (ii)  $f^{1}(t,y_{t}^{2}, z_{t}^{2}, k_{t}^{2}) \geq f^{2}(t, y_{t}^{2}, z_{t}^{2}, k_{t}^{2})$,
  $dtd\mathbb{P}-a.e,$\\
then we have $y_{t}^{1} \geq y_{t}^{2}$, $ a.s.$, for all $t \in
[0,T]$.
\end{lemma}

For some $f: \Omega\times[0, T]\times{\mathbb{R}}
\times{\mathbb{R}}^{d}\times{\mathbb{R}}\rightarrow {\mathbb{R}}$\
satisfying $(H1)$ and $(H2)$,  we let, for $i=1, 2$,
\begin{eqnarray*}
f_i(s, y_s^i, z_s^i, k_s^i)=f(s, y_s^i, z_s^i, k_s^i)+\varphi_i(s),
\end{eqnarray*}
 where $\varphi_i\in
{\cal{H}}^{2}(0,T;{\mathbb{R}}).$ If  $\xi_1$ and $\xi_2$ are in
$L^{2}(\Omega, {\cal{F}}_{T}, \mathbb{P};\mathbb{R})$, then we have
the following lemma.

\begin{lemma}\label{l1}
 Let us denote by  $(y^1, z^1, k^1)$ and $(y^2, z^2, k^2)$ the solutions  of BSDE (\ref{e4}) with the data
 $(\xi_1, f_1)$\ and $(\xi_2, f_2)$, respectively. Then
 the following holds:  for all  $ t\in[0,T]$,
 \begin{eqnarray*}
  &&|y^1_t-y^2_t|^2+\frac{1}{2}\mathbb{E}[\int^T_te^{\beta(s-t)}(|
  y^1_s-y^2_s|^2+ |z^1_s-z^2_s|^2)ds|{\cal{F}}_t]+\frac{\lambda}{2}\mathbb{E}[\int^T_te^{\beta(s-t)}|
  k^1_s-k^2_s|^2ds|{\cal{F}}_t]  \\
  &&\qquad\leq \mathbb{E}[e^{\beta(T-t)}|\xi_1-\xi_2|^2|{\cal{F}}_t]+ \mathbb{E}[\int^T_te^{\beta(s-t)}
           |\varphi_1(s)-\varphi_2(s)|^2ds|{\cal{F}}_t],\ \mathbb{P}-a.s.
 \end{eqnarray*}
Here $\beta\geq 2+2C+4C^2$, where $C$ is the Lipschitz constant in
$(H1)$.
\end{lemma}
 For the  proof, the readers can be referred to Barles, Buckdahn
and Pardoux~\cite{BBP1997}.

\section{Stochastic differential games
with jumps}\label{NS2} In this section, we first introduce
nonzero-sum SDGs, and then we define the value functions and show
that they have a deterministic version. Finally, we state the
dynamic programming principle for stopping times, which is crucial
for the next section.

 Let  ${\mathcal{U}}$ (resp., ${\mathcal{V}}$) be the set
of admissible control processes for the first (resp., second)
player, i.e., the set of all $U$ (resp., $V$)-valued
${\mathbb{F}}$-predictable processes. We suppose that the control
state spaces $U$ and $V$ are  compact metric spaces.

For given admissible controls $u(\cdot)\in {\mathcal{U}}$ and
$v(\cdot)\in {\mathcal{V}}$, we consider  the following stochastic
differential equation (SDE): for $t\in [0,T]$ and  $\zeta \in L^2
(\Omega ,{\mathcal{F}}_t, \mathbb{P};{\mathbb{R}}^n)$,
  \begin{equation}\label{e3}
  \left \{
  \begin{array}{llll}
  dX^{t,\zeta ;u, v}_s & = & b(s,X^{t,\zeta; u,v}_s, u_s, v_s) ds +
          \sigma(s,X^{t,\zeta; u,v}_s, u_s, v_s) dB_s, \hskip1cm  s\in [t,T],\\
   X^{t,\zeta ;u, v}_t  & = & \zeta,
   \end{array}
   \right.
 \end{equation}
where
  $$
  \begin{array}{llll}
  &   b:[0,T]\times {\mathbb{R}}^n\times U\times V \rightarrow {\mathbb{R}}^n \
  ,\ \ \ \ \ \   \sigma: [0,T]\times {\mathbb{R}}^n\times U\times V\rightarrow {\mathbb{R}}^{n\times d},
  \end{array}
  $$
  satisfy the following assumptions:
  $$
  \begin{array}{ll}
 \rm{(i)}& \mbox{For every fixed}\ x\in {\mathbb{R}}^n,\ b(.,x,
 .,.), \  \mbox{and}\ \sigma(.,x,
 .,.)\ \mbox{are  continuous in}\ (t,u,v).\\
 \rm{(ii)}&\mbox{There exists a constant }C>0\ \mbox{such that, for all}\ t\in [0,T],\
  x, x'\in {\mathbb{R}}^n,\ u \in U,\ v \in V, \\
   &\hskip1cm |b(t,x,u,v)-b(t,x',u ,v)|+ |\sigma(t,x,u,v)-\sigma(t,x',u, v)|\leq C|x-x'|.
  \end{array}
  \eqno{(H4)}
  $$
From $(H4)$ we know that  there exists some $C>0$\ such that, for
all $0 \leq t \leq T,\ u\in U,\ v \in V,\  x\in {\mathbb{R}}^n $,
\begin{eqnarray*}
  \begin{array}{rcl}
  |b(t,x,u,v)| +|\sigma (t,x,u,v)| \leq C(1+|x| ).
  \end{array}
\end{eqnarray*}
It is well known  that under $(H4)$, for any $u(\cdot)\in
{\mathcal{U}}$ and $v(\cdot)\in {\mathcal{V}}$, SDE (\ref{e3}) has a
unique strong solution. Furthermore, we have the following estimates
for the solution of SDE (\ref{e3}) (e.g., see \cite{BHL2010}).
\begin{proposition}
Let the assumption $(H4)$ hold. Then, for  $p\geq2$, there exists a
positive constant $C=C_{p}$\ such that, for $t \in [0,T]$,
$u(\cdot)\in {\mathcal{U}}, v(\cdot)\in {\mathcal{V}}$\ and $ \zeta,
\zeta'\in L^2 (\Omega ,{\mathcal{F}}_t,\mathbb{P};{\mathbb{R}}^n),$
\begin{eqnarray}\label{e22}
&& \mathbb{E}[\sup \limits_{s\in [t,T]}|X^{t,\zeta; u, v}_s
-X^{t,\zeta';u, v}_s|^p|{{\mathcal{F}}_t}]
 \leq  C|\zeta -\zeta'|^p, \quad \mathbb{P}-a.s., \nonumber\\
&& \mathbb{E}[ \sup \limits_{s\in [t,T]} |X^{t,\zeta
;u,v}_s-\zeta|^p|{{\mathcal{F}}_t}] \leq
                        C(1+|\zeta|^p)|T-t|^{\frac{p}{2}}, \
                        \mathbb{P}-a.s.
\end{eqnarray}
\end{proposition}
For given $ \Phi_{i}: {\mathbb{R}}^n \rightarrow {\mathbb{R}},\
\widetilde{f}_{i}:[0,T]\times {\mathbb{R}}^n \times {\mathbb{R}}^{2}
\times {\mathbb{R}}^d \times U \times V \rightarrow {\mathbb{R}},\
i=1,2, $ we make the following assumptions:
$$
\begin{array}{ll}
\rm{(i)}& \mbox{For every fixed}\ (x, y, z)\in {\mathbb{R}}^n \times
{\mathbb{R}}^{2} \times {\mathbb{R}}^d\times {\mathbb{R}},\
\widetilde{f}_{i}(., x, y, z, .,.)\
\mbox{is continuous in}\ \\
&(t,u,v)\ \mbox{and there exists a constant}\ C>0 \ \mbox{such that,
for all}\ t\in [0,T],\ x, x'\in {\mathbb{R}}^n,\\\
& y, y'\in {\mathbb{R}}^{2},\ z, z'\in {\mathbb{R}}^d, u \in U \
\mbox{and}\ v \in V,\\
 &\hskip1cm
 \begin{array}{l}
|\widetilde{f}_{i}(t,x,y,z,u,v)-\widetilde{f}_{i}(t,x',y',z',u,v)|
\leq C(|x-x'|+|y-y'| +|z-z'|).
\end{array}\\
\rm{(ii)} & \mbox{For all} \ (y_{1},y_{2}), (y'_{1},y'_{2})\in
{\mathbb{R}}^{2}, \mbox{and}\ (t,x,z,u,v)\in [0, T]\times
{\mathbb{R}}^n\times  {\mathbb{R}}^{d}\times U\times V,\\ &
\mbox{there exists a constant}\ K>-1\ \mbox{such that}\
\\ &
 \widetilde{f}_{1}(t,x,y_{1},y_{2},z,u,v)-\widetilde{f}_{1}(t,x,y_{1},y'_{2},z,u,v)\geq
K(y_{2}-y'_{2}),
 \\&
 \widetilde{f}_{2}(t,x,y_{1},y_{2},z,u,v)-\widetilde{f}_{2}(t,x,y'_{1},y_{2},z,u,v)\geq
 K(y_{1}-y'_{1}).\\
\rm{(iii)}&\mbox{There exists a constant}\ C>0 \ \mbox{such that,
for all}\ x, x'\in {\mathbb{R}}^n, i=1,2,\\
 &\mbox{  }\hskip3cm |\Phi_{i} (x) -\Phi_{i} (x')|\leq C|x-x'|.
 \end{array}
 \eqno {(H5)}
 $$
The following system of two coupled  BSDEs will define the cost
functionals of the game associated with (\ref{e3}).
    \begin{eqnarray}\label{e1}
   \left \{\begin{array}{rcl}
   -d\ ^{1}\widetilde{Y}_s & = & \widetilde{f}_{1}(s,X^{t,x; u, v}_s, \ ^{1}\widetilde{Y}_s,
   \ ^{2}\widetilde{Y}_s+\ ^{2}\widetilde{H}_s,
  \ ^{1}\widetilde{Z}_s,u_s, v_s) ds\\ && \qquad \qquad \qquad
  -\lambda \ ^{1}\widetilde{H}_sds-\ ^{1}\widetilde{Z}_s dB_s
  -\ ^{1}\widetilde{H}_s d\widetilde{N}_{s},\\
      -d\ ^{2}\widetilde{Y}_s & = & \widetilde{f}_{2}(s,X^{t,x; u, v}_s, \ ^{1}\widetilde{Y}_s
      +\ ^{1}\widetilde{H}_s, \ ^{2}\widetilde{Y}_s,
     \ ^{2}\widetilde{Z}_s,u_s, v_s) ds \\ && \qquad \qquad \qquad-\lambda \ ^{2}\widetilde{H}_sds
   -\ ^{2}\widetilde{Z}_s dB_s-\ ^{2}\widetilde{H}_s d\widetilde{N}_{s},\\
       ^{1}\widetilde{Y}_T  & = & \Phi_{1} (X^{t,x; u, v}_T),
       \    ^{2}\widetilde{Y}_T   =  \Phi_{2} (X^{t,x; u, v}_T), \
       s\in[t,T],
   \end{array}\right.
   \end{eqnarray}
where $X^{t,x; u, v}$\ is the solution of equation (\ref{e3}) with
$\zeta=x\in\mathbb{R}^{n}$.   Under the assumption $(H5)$, from Tang
and Li \cite{TL1994} we know that equation (\ref{e1}) has a unique
solution. For given  control processes $u(\cdot)\in\mathcal {U}$ and
$v(\cdot)\in\mathcal {V}$, we introduce now the associated cost
functional for the $i^{th}$ player, $i=1,2,$
\begin{equation*}\label{}
J_{i}(t,x;u,v):= \ ^{i}\widetilde{Y}^{t,x;u,v}_t,\qquad (t,x)\in
[0,T]\times\mathbb{R}^n,
\end{equation*}
where $(\ ^{i}\widetilde{Y}^{t,x;u,v},\
^{i}\widetilde{Z}^{t,x;u,v},\ ^{i}\widetilde{H}^{t,x;u,v})$ is the
solution of (\ref{e1}).

In the above definition we generalize the framework studied by Lin
\cite{L2011}. Indeed, in \cite{L2011} we studied cost functionals
defined by a decoupled system of BSDEs, while now the both BSDEs are
coupled: the both players do not only influence mutually their cost
functionals in the choice of their control processes, but also their
gain processes $\ ^{i}\widetilde{Y}^{t,x;u,v}, i=1,2.$ With an
argument introduced in Pardoux, Pradeilles and  Rao \cite{PPR1997}
we can transfer the coupled system of BSDEs into a decoupled system
of BSDEs. For $0\leq s\leq t\leq T, i=1,2,$ we denote by
$N((0,s]):=N_{s}$ and $N((t,s]):=N_{s}-N_{t}$. Let us  define a
Markov process $N_{s}^{t,i}$ as follows: $
N_{s}^{t,i}=m(i+N((t,s]))$, where $m(j)=1,$ if $j$ is odd, and
$m(j)=2,$ if $j$ is even.

For  $(t,x,y,h,z,u,v)\in [0, T]\times {\mathbb{R}}^n\times
{\mathbb{R}} \times{\mathbb{R}}\times{\mathbb{R}}^{d}\times U\times
V$, we define
\begin{eqnarray*}\label{}
f_{1}(t,x,y,h,z,u,v):=\widetilde{f}_{1}(t,x,y,y+h,z,u,v),\\
f_{2}(t,x,y,h,z,u,v):=\widetilde{f}_{2}(t,x,y+h,y,z,u,v),
\end{eqnarray*}
and we consider the following controlled decoupled BSDEs with jumps:
for $i=1,2,$
  \begin{eqnarray}\label{e2}
   \left \{\begin{array}{rcl}
   -d\ ^{i}Y^{t,x; u, v}_s & = & f_{N_{s}^{t,i}}(s,X^{t,x; u, v}_s, \ ^{i}Y^{t,x; u, v}_s,
   \ ^{i}H^{t,x; u, v}_s,\  ^{i}Z^{t,x; u, v}_s,u_s, v_s) ds, \\
   &&-\lambda\ ^{i}H^{t,x; u, v}_s ds-\ ^{i}Z^{t,x; u, v}_s dB_s-\ ^{i}H^{t,x; u, v}_s d\widetilde{N}_{s}, s\in [t,T],\\
        ^{i}Y^{t,x; u, v}_T  & = & \Phi_{N_{T}^{t,i}} (X^{t,x; u,
        v}_T).
   \end{array}\right.
  \end{eqnarray}

 Since
$\widetilde{f}_{i}, i=1,2,$ are Lipschitz in $(x,y,z)$, uniformly
with respect to $(t,u,v)$,   it is easy to check that also the
coefficients $f_{i}, i=1,2,$   have this property. From Lemma
\ref{lemma5} we know that the above BSDE has a unique solution. In
what follows we choose the intensity $\lambda>0$  such that
$K-\lambda>-1$. It follows from Lemma \ref{l8} that the comparison
theorem for the BSDE (\ref{e2}) holds. Moreover, we also have the
following propositions.
\begin{proposition}\label{p10}
 Let the assumption $(H5)$ hold. Then we have, for $s\in[t,T], i=1,2$,  $$\ ^{i}Y^{t,x; u, v}_s=\
^{N_{s}^{t,i}}\widetilde{Y}^{t,x; u, v}_s,$$
 $$\ ^{i}H^{t,x; u, v}_s= \ ^{m(N_{s-}^{t,i}+1)}\widetilde{H}^{t,x; u, v}_s
 + \ ^{m(N_{s-}^{t,i}+1)} \widetilde{Y}^{t,x; u, v}_{s-}-
 \ ^{N_{s-}^{t,i}} \widetilde{Y}^{t,x; u, v}_{s-}.$$
 In particular, we have  $\ ^{i}Y^{t,x; u, v}_t=\ ^{i}\widetilde{Y}^{t,x; u,
 v}_t$, i.e.,  $J_{i}(t,x;u,v)= \
 ^{i}Y^{t,x;u,v}_t$.
\end{proposition}

\begin{proof}
We consider the solution $(^{i}\widetilde{Y}^{t,x; u, v},
   \ ^{i}\widetilde{Z}^{t,x; u, v}, \ ^{i}\widetilde{H}^{t,x; u, v}), i=1,2,$  of equation
   (\ref{e1}) and, suppressing for simplicity the superscript $(t,x, u,
   v)$, we put
 $$\ ^{i}\hat{Y}_s=\
^{N_{s}^{t,i}}\widetilde{Y}_s, \ ^{i}\hat{Z}_s=\
^{N_{s}^{t,i}}\widetilde{Z}_s, \ ^{i}\hat{H}_s= \
^{m(N_{s-}^{t,i}+1)}\widetilde{H}_s
 + \ ^{m(N_{s-}^{t,i}+1)} \widetilde{Y}_{s-}-
 \ ^{N_{s-}^{t,i}} \widetilde{Y}_{s-},\ s\in[t,T].$$
Then, obviously, $(^{i}\hat{Y},
   \ ^{i}\hat{Z}, \ ^{i}\hat{H})\in S^2 (0,T; \mathbb{R})\times\mathcal{H}^2 (0,T;
\mathbb{R}^{d})\times  \mathcal{H}^2 (0,T; \mathbb{R})$. Moreover,
setting $\tau_{0}=t$ and
\begin{eqnarray*}
\tau_{l}=\inf\Big\{s\geq t, N((t,s])=l\Big\}\wedge T, l\geq 1,
\end{eqnarray*}
we have,  on the stochastic interval $]]\tau_{0},\tau_{1}[[,$
\begin{eqnarray*}\label{}
   -d\ ^{1}\hat{Y}_s & = & -d\ ^{1}\widetilde{Y}_s =
    \widetilde{f}_{1}(s,X_s, \ ^{1}\widetilde{Y}_s, \ ^{2}\widetilde{Y}_s+\ ^{2}\widetilde{H}_s,
  \ ^{1}\widetilde{Z}_s,u_s, v_s) ds-\ ^{1}\widetilde{Z}_s dB_s\\
  & = & \widetilde{f}_{1}(s,X_s, \ ^{1}\widetilde{Y}_s,
   (\ ^{2}\widetilde{Y}_s-\ ^{1}\widetilde{Y}_s+\ ^{2}\widetilde{H}_s)+\ ^{1}\widetilde{Y}_s,
  \ ^{1}\widetilde{Z}_s,u_s, v_s) ds-\ ^{1}\widetilde{Z}_s dB_s\\
  & = & f_{1}(s,X_s, \ ^{1}\widetilde{Y}_s,
   \ ^{2}\widetilde{Y}_s-\ ^{1}\widetilde{Y}_s+\ ^{2}\widetilde{H}_s,
  \ ^{1}\widetilde{Z}_s,u_s, v_s) ds-\ ^{1}\widetilde{Z}_s dB_s\\
   & = & f_{1}(s,X_s, \ ^{1}\hat{Y}_s,\ ^{1}\hat{H}_s,
  \ ^{1}\hat{Z}_s,u_s, v_s) ds-\ ^{1}\hat{Z}_s dB_s.
\end{eqnarray*}
On the other hand, analysing  the jump hight of $^{1}\hat{Y}$ at
$\tau_{1}$, we get
\begin{eqnarray*}
\Delta\ ^{1}\hat{Y}_{\tau_{1}}&=&\ ^{1}\hat{Y}_{\tau_{1}}-\
^{1}\hat{Y}_{\tau_{1}-}=(\ ^{2}\widetilde{Y}_{\tau_{1}}-\
^{1}\widetilde{Y}_{\tau_{1}-})\Delta N_{\tau_{1}}\\
&=&(\ ^{2}\widetilde{H}_{\tau_{1}}+\ ^{2}\widetilde{Y}_{\tau_{1}-}-\
^{1}\widetilde{Y}_{\tau_{1}-})\Delta N_{\tau_{1}}=\
^{1}\hat{H}_{\tau_{1}}\Delta N_{\tau_{1}}.
\end{eqnarray*}
Consequently, $(^{1}\hat{Y},
   \ ^{1}\hat{Z}, \ ^{1}\hat{H})$
solves  (\ref{e2}) over the interval $[[\tau_{0},\tau_{1}]],$ with
$^{1}\hat{Y}_{\tau_{1}}=\ ^{N_{s}^{t,1}}\widetilde{Y}_{\tau_{1}}$.
By iterating this argument and arguing in a similar way for $i=2$,
we complete the proof. Indeed, from the uniqueness of BSDE
(\ref{e2}) it follows that $(^{i}Y^{t,x; u, v},
   \ ^{i}Z^{t,x; u, v}, \ ^{i}H^{t,x; u, v})=(^{i}\hat{Y},
   \ ^{i}\hat{Z}, \ ^{i}\hat{H}), i=1,2.$
\end{proof}
From  standard  BSDEs estimates  we have the following:
\begin{proposition}\label{p3}
 There exists some constant $C>0$\ such that, for all $0 \leq t
\leq T,\ x, x' \in \mathbb{R}^n,\ u(\cdot) \in {\mathcal{U}}\
\mbox{and}\ v(\cdot) \in {\mathcal{V}},$
 \begin{eqnarray*}
 |\ ^{i}Y^{t,x; u, v}_t -\ ^{i}Y^{t,x'; u, v}_t| \leq C|x -x'|,\quad
  |\ ^{i}Y^{t,x; u, v}_t| \leq C (1+|x|), \ \mathbb{P}-a.s.
\end{eqnarray*}
\end{proposition}
Let us now introduce  subspaces of admissible controls and give the
definition of NAD strategies.  For later applications  this has to
be done for games over stochastic intervals. Let $\sigma,\tau$ be
two stopping times such that $t\leq \sigma \leq \tau \leq T.$
\begin{definition}
 The space $\mathcal {U}_{\sigma,\tau}$ (resp., $\mathcal {V}_{\sigma,\tau}$) of admissible
controls for the $1^{th}$ player (resp., the $2^{nd}$ player) over
the given stochastic time interval $[[\sigma,\tau]]$ is defined as
the space of all processes $\{u_{r}, \sigma\leq r \leq \tau\}$
(resp.,
 $\{v_{r}, \sigma\leq r \leq \tau\}$), such that, for $u_{0}\in U$,
the process  $\{u_r 1_{[\sigma,\tau]}+ u_0 1_{[\sigma,\tau]^{c}}\}$
(resp., for $v_{0}\in V$, the process  $\{v_r 1_{[\sigma,\tau]}+ v_0
1_{[\sigma,\tau]^{c}}\}$) are $\mathbb{F}$-predictable and take its
values in $U$ (resp., $V$).
\end{definition}

\begin{definition}\label{d1}
 A nonanticipative strategy with delay (NAD strategy)
for the $1^{th}$ player over the given stochastic time interval
$[[\sigma,\tau]]$ is a measurable mapping $\alpha:\mathcal
{V}_{\sigma,\tau}\rightarrow \mathcal {U}_{\sigma,\tau}$ such that
the following properties hold:

1)  $\alpha$ is a nonanticipative strategy, i.e., for every
$\mathbb{F}$-stopping time  $\tau'$ on $\Omega$ with $\sigma\leq
\tau'\leq\tau,$ and for all
 $v_{1},v_{2} \in\mathcal {V}_{\sigma,\tau}$ with $v_{1}=v_{2}$   on $[[\sigma,\tau']]$,
it holds $\alpha(v_{1})=\alpha(v_{2})$ on $[[\sigma,\tau']]$. (The
identification of $v_{1}=v_{2}$ and $\alpha(v_{1})=\alpha(v_{2})$ is
in the $dsd\mathbb{P}$ almost everywhere sense.)

2) $\alpha$ is a strategy with delay, i.e., for all $v\in\mathcal
{V}_{\sigma,\tau}$,
 there exists an increasing sequence of stopping times
$\{S_{n}(v)\}_{n\geq 1}$ with

i) $\sigma=S_{0}(v)\leq S_{1}(v) \leq\cdots \leq S_{n}(v)\leq \cdots
\leq \tau$, \  ii) $\bigcup_{n\geq
1}\{S_{n}(v)=\tau\}=\Omega$, $\mathbb{P}$-a.s.,\\
such that, for all $\Gamma\in\mathcal {F}_{\sigma}$ and for all
$n\geq 1 $ and $v,v' \in \mathcal {V}_{\sigma,\tau}$, it holds:  if
$v=v'$
 on $[[\sigma,S_{n-1}(v)]]\bigcap (\Gamma\times[t,T])$, then
iii) $S_{l}(v)=S_{l}(v')$, on $\Gamma$, $\mathbb{P}$-a.s.,  $1\leq l
\leq n$,

iv) $\alpha(v)=\alpha(v')$, on $[[\sigma,S_{n}(v)]]\bigcap
(\Gamma\times[t,T])$.

We denote the collection of   all such NAD strategies for the
$1^{th}$ player by $\mathcal {A}_{\sigma,\tau}$. We can define all
NAD strategies $\beta:\mathcal {U}_{\sigma,\tau}\rightarrow \mathcal
{V}_{\sigma,\tau}$ for the $2^{nd}$ player    symmetrically and
denote the set of them by $\mathcal {B}_{\sigma,\tau}$.
\end{definition}

We have the following lemma, which turns out to be  useful in what
follows. Since the proof of this lemma is similar to that in Lin
\cite{L2011}, we omit it here.
\begin{lemma}\label{l5}
For  $(\alpha,\beta)\in\mathcal {A}_{\sigma,\tau}\times \mathcal
{B}_{\sigma,\tau}$,
 there exists a unique couple of admissible control processes
$(u,v)\in \mathcal {U}_{\sigma,\tau} \times \mathcal
{V}_{\sigma,\tau}$ such that $\alpha(v)=u, \beta(u)=v.$
\end{lemma}
For  $(\alpha,\beta)\in\mathcal {A}_{t,T}\times \mathcal {B}_{t,T}$,
it follows from Lemma \ref{l5} that  there exists a unique couple of
controls $(u,v)\in\mathcal {U}_{t,T}\times\mathcal {V}_{t,T}$ such
that  $(\alpha(v),\beta(u))=(u,v)$. In this sense, we define
$$(X^{t,x; \alpha,\beta},\ ^{i}Y^{t,x; \alpha,\beta},   \ ^{i}Z^{t,x; \alpha,\beta}, \ ^{i}H^{t,x; \alpha,\beta})
:=(X^{t,x; u, v},\ ^{i}Y^{t,x; u, v},   \ ^{i}Z^{t,x; u, v}, \
^{i}H^{t,x; u, v}),$$ and $J_{i}(t,x;\alpha,\beta):=J_{i}(t,x;u,v).$
This definition allows, in particular, to define the  value
functions $W_{i}$ and $U_{i}$ of the game: For all $(t,x)\in
[0,T]\times\mathbb{R}^n$, we put
\begin{eqnarray*}\label{}
\quad W_{i}(t,x):=\esssup_{\alpha\in\mathcal {A}_{t,T}}
\essinf_{\beta\in\mathcal {B}_{t,T}} J_{i}(t,x;\alpha,\beta),\quad
U_{i}(t,x):=\essinf_{\beta\in\mathcal {B}_{t,T}}
\esssup_{\alpha\in\mathcal {A}_{t,T}} J_{i}(t,x;\alpha,\beta).
\end{eqnarray*}
The functions $W_{i}$ and $U_{i}, i=1,2,$ are called the lower and
upper value functions, respectively. Observe that, according to the
definition of esssup and essinf over a uniformly essentially bounded
family of $\mathcal {F}_{t}$-measurable random variables, both
$W_{i}(t,x)$ and $U_{i}(t,x)$ are a priori elements of
$L^{\infty}(\Omega,\mathcal {F}_{t},\mathbb{P})$. However, we will
prove that  they are deterministic.
\begin{proposition}\label{p1}
Let the assumptions $(H4)$ and $(H5)$ hold. Then, for any $(t, x)\in
[0, T]\times {\mathbb{R}}^n$, we have
$W_{i}(t,x)=\mathbb{E}[W_{i}(t,x)]$, and
$U_{i}(t,x)=\mathbb{E}[U_{i}(t,x)]$, $\mathbb{P}-a.s.,$ $i=1,2.$
\end{proposition}
This proposition allows to identify $W_{i}(t,x)$, $U_{i}(t,x)$ with
the deterministic functions $\mathbb{E}[W_{i}(t,x)]$,
$\mathbb{E}[U_{i}(t,x)]$, respectively,  $i=1,2.$   The proposition
follows then by the following two lemmas.  The following lemma
indicates that $W_{i}$ is invariant by a sufficiently large class of
transformations on $\Omega$.  \vskip1mm
\begin{lemma}\label{Jll1}
For $(t,x)\in [0, T]\times {\mathbb{R}}^n$,  $\tau:
\Omega\rightarrow\Omega$ is  an invertible ${\cal{F}}-{\cal{F}}$
measurable transformation such that
\begin{eqnarray*}\label{}
&& i)\ \ \tau\  \mbox{and}\  \tau^{-1}: \Omega\rightarrow\Omega\
\mbox{are}\  {\cal{F}}_t-{\cal{F}}_t\
 \mbox{measurable};\\
&& ii) \ \ (B_s-B_t)\circ\tau=B_s-B_t,\ s\in[t, T], \ \ N((t, s])\circ\tau=N((t, s]),\ s\in[t, T];\\
&& iii)\ \ \mbox{the law}\ \mathbb{P}\circ [\tau]^{-1}\ \mbox{of}\
\tau\ \mbox{is equivalent to the underlying probability measure}\
\mathbb{P},
\end{eqnarray*}
 then  $W_{i}(t,x)\circ\tau=W_{i}(t,x),$
$\mathbb{P}$-a.s.
\end{lemma}

Even if the formulation of of the lemma is the same as in
\cite{BHL2010}, the proof is more difficult here. Indeed, in
\cite{BHL2010} games of the type "strategy again control" were
studied, while we investigate here games of the type "NAD strategy
against NAD strategy".
 \begin{proof}: We give the proof in four steps:
 \vskip0.1cm
\noindent ${\bf Step\ 1}$: For any $u\in {\mathcal{U}}_{t,T}, \ v\in
{\mathcal{V}}_{t,T}, \ J_{i}(t, x; u,v)\circ\tau= J_{i}(t, x;
u(\tau),v(\tau)),\ \mathbb{P}\mbox{-a.s.}$ \vskip0.1cm The proof
follows closely  from the arguments in \cite{BHL2010} or
\cite{BL2006} and is therefore  omitted.

 \vskip0.1cm
\noindent ${\bf Step\ 2}$:\ \  For $\beta\in {\cal{B}}_{t,T},$\ let
$\widehat{\beta}(u):=\beta(u(\tau^{-1}))(\tau),\ u\in
{\mathcal{U}}_{t,T},$ and for $\alpha\in {\cal{A}}_{t,T},$\ let
$\widehat{\alpha}(v):=\alpha(v(\tau^{-1}))(\tau),\ v\in
{\mathcal{V}}_{t,T}.$  Then, $\widehat{\beta}\in {\cal{B}}_{t,T}$
and $\widehat{\alpha}\in {\cal{A}}_{t,T}.$

\vskip0.1cm We only give the proof for $\widehat{\beta}$, since we
can use a similar argument for  $\widehat{\alpha}$.  From the
definition of $\widehat{\beta}$ we know that  $\widehat{\beta}$\
maps ${\mathcal{U}}_{t,T}$\ into ${\mathcal{V}}_{t,T}$.

(1) {\it $\widehat{\beta}$ is a nonanticipative strategy.} Indeed,
let $\sigma: \Omega\rightarrow [t, T]$\ be an
${\mathbb{F}}$-stopping time and $ u_1, u_2 \in {\mathcal{U}}_{t,
T}$\ such that $ u_1\equiv u_2\ \mbox {on}\ [t,\ \sigma].$\ Since
$\tau({\cal{F}}_s):=\{\tau(A),\ A\in {\cal{F}}_s\}={\cal{F}}_s,\
s\in[t, T]$,  the assumptions i) and ii) imply that $
\sigma(\tau^{-1})\ \mbox{is still an}\ {\mathbb{F}}$-stopping time.
Obviously, we have  $ u_1(\tau^{-1})\equiv u_2(\tau^{-1})\ \mbox
{on}\ [[t,\ \sigma(\tau^{-1})]]$. From  $\beta\in {\cal{B}}_{t,T}$
it follows that $\beta(u_1(\tau^{-1}))= \beta(u_2(\tau^{-1}))\ $ $
\mbox {on}\ [[t,\ \sigma(\tau^{-1})]]$. Consequently,
$$\widehat{\beta}(u_1)=\beta(u_1(\tau^{-1}))(\tau)= \beta(u_2(\tau^{-1}))(\tau)=\widehat{\beta}(u_2)\ \mbox
{on}\ [[t,\ \sigma]].$$

 (2) {\it$\widehat{\beta}$ is a  strategy with
 delay.}   Since $\beta$ is a nonanticipative strategy with delay,  we
have, for all $u\in\mathcal {U}_{t,T}$,
 the existence of an increasing sequence of stopping times
$\{S'_{n}(u)\}_{n\geq 1}$ with

(a) $t=S'_{0}(u)\leq S'_{1}(u) \leq\cdots \leq S'_{n}(u)\leq \cdots
\leq T$,  (b) $\bigcup_{n\geq
1}\{S'_{n}(u)=T\}=\Omega$, $\mathbb{P}$-a.s.,\\
such that, for all $\Gamma\in\mathcal {F}_{t}$ and  for all $n\geq 1
$ and $u,u' \in \mathcal {U}_{t,T}$, it holds:  if $u=u'$
 on $[[t,S'_{n-1}(u)]]\bigcap (\Gamma\times[t,T])$, then
(c) $S'_{l}(u)=S'_{l}(u')$, on $\Gamma$, $\mathbb{P}$-a.s., $1\leq l
\leq n$, (d) $\beta(u)=\beta(u')$, on $[ t,S'_{n}(u)]]\bigcap
(\Gamma\times[t,T])$.

For all $u\in\mathcal {U}_{t,T}$, we put
 $S_{n}(u)=S'_{n}(u(\tau^{-1}))(\tau),  n\geq 1$. It is easy to
 check that $\widehat{\beta}$ is a nonanticipative strategy with
 delay. Moreover, since $\beta(u)=\widehat{\beta}(u(\tau))(\tau^{-1}),\ u\in
{\mathcal{U}}_{t,T},$ we have that $ \{\widehat{\beta}\ |\ \beta \in
{\mathcal{B}}_{t,T} \}={\mathcal{B}}_{t,T}$. Ditto for
$\widehat{\alpha}$. \vskip2mm

\noindent ${\bf Step\ 3}$: The following holds:
$$\Big(\esssup_{\alpha \in
{\mathcal{A}}_{t,T}}\essinf_{\beta \in
{\mathcal{B}}_{t,T}}J_{i}(t,x;
\alpha,\beta)\Big)(\tau)=\esssup_{\alpha \in
{\mathcal{A}}_{t,T}}\essinf_{\beta \in
{\mathcal{B}}_{t,T}}\Big(J_{i}(t,x; \alpha,\beta)(\tau)\Big),\
\mathbb{P}\mbox{-a.s.}
$$
Taking into account the properties of $\tau$, this relation can be
proven in the same manner as the corresponding relation in
\cite{BL2006}, also see \cite{BCQ2011}.

 \vskip0.1cm

\noindent ${\bf Step\ 4}$:  We now show that
$W_{i}(t,x)(\tau)=W_{i}(t,x), \ \mathbb{P}\mbox{-a.s.}$

Let $(\alpha,\beta)\in \mathcal {A}_{t,T}\times \mathcal {B}_{t,T}$.
Then there exists a unique couple $(u,v)\in \mathcal {U}_{t,T}
\times \mathcal {V}_{t,T}$ such that $\alpha(v)=u, \beta(u)=v.$
Consequently, due to Step 2,  $\widehat{\alpha}(v(\tau))=u(\tau)$
and $\widehat{\beta}(u(\tau))=v(\tau)$, and, thus, from Step 1,
\begin{eqnarray*}\label{}
J_{i}(t,x; \alpha,\beta)(\tau)=J_{i}(t,x; u,v)(\tau)=J_{i}(t,x;
u(\tau),v(\tau))=J_{i}(t,x; \widehat{\alpha},\widehat{\beta}),\quad
\mathbb{P}-a.s.
\end{eqnarray*}
Thus, since $ \{\widehat{\alpha}\ |\ \alpha \in {\mathcal{A}}_{t,T}
\}={\mathcal{A}}_{t,T}$ and $ \{\widehat{\beta}\ |\ \beta \in
{\mathcal{B}}_{t,T} \}={\mathcal{B}}_{t,T}$, we conclude with Step
3,
\begin{eqnarray*}
     W_{i}(t,x)(\tau) & = & \esssup_{\alpha \in
{\mathcal{A}}_{t,T}}\essinf_{\beta \in
{\mathcal{B}}_{t,T}}(J_{i}(t,x; \alpha,\beta)(\tau))\\
& = & \esssup_{\alpha \in {\mathcal{A}}_{t,T}}\essinf_{\beta \in
{\mathcal{B}}_{t,T}}J_{i}(t,x;  \widehat{\alpha},\widehat{\beta})\\
& = & \esssup_{\alpha \in {\mathcal{A}}_{t,T}}\essinf_{\beta \in
{\mathcal{B}}_{t,T}}J_{i}(t,x;  \alpha,\beta)\\
& = &W_{i}(t,x),\ \mathbb{P}\mbox{-a.s.}
\end{eqnarray*}
We get the wished result.
\end{proof}

For  $ \ell\geq 1$, let us define the transformation $\tau_\ell:
\Omega\rightarrow\Omega$\ such that
\begin{eqnarray*}
 \begin{array}{lll}
 &(\tau_\ell\omega)((t-\ell, r])=\omega((t-2\ell, r-\ell]):=\omega(r-\ell)-\omega(t-2\ell);\\
 &(\tau_\ell\omega)((t-2\ell, r-\ell])=\omega((t-\ell, r]),\ \mbox{for}\ r\in [t-\ell, t];\\
 &(\tau_\ell\omega)((s, r])=\omega((s, r]),\ (s, r]\cap (t-2\ell, t]=\emptyset;\\
 &(\tau_\ell\omega)(0)=0.
 \end{array}
\end{eqnarray*}
The transformation $\tau_\ell$ satisfies the assumptions $i), ii)$
and $iii)$ of Lemma \ref{Jll1}. Therefore, $W(t, x)(\tau_\ell)=W(t,
x),\ \mathbb{P}\mbox{-a.s.},\ \ell \geq 1.$ We have the following
lemma by using arguments similar to that in \cite{BHL2010}.

\begin{lemma}\label{l2} Let $\zeta \in L^\infty(\Omega, {\cal F}_t, \mathbb{P})$\ be
such that, for all $\ell \geq 1$\ natural number,
$\zeta(\tau_\ell)=\zeta$, $\mathbb{P}$-a.e. Then, there exists some
real $C$\ such that $\zeta=C$, $\mathbb{P}$-a.s.
\end{lemma}

By the definition of $W_{i}(t,x)$ and Proposition \ref{p3}  we can
easily get the following properties for our deterministic function
$W_{i}$. The same proposition holds for $U_{i}$.
\begin{proposition}\label{p4} Under the assumptions $(H4)$ and
$(H5)$, there exists a constant $C>0$\ such that, for all $ 0 \leq t
\leq T,\ x, x'\in {\mathbb{R}}^n$,
\begin{eqnarray*}
 |W_{i}(t,x)-W_{i}(t,x')| \leq C|x-x'|,\ \  |W_{i}(t,x)| \leq C(1+|x|).
\end{eqnarray*}
\end{proposition}

 We now recall the notion of stochastic backward semigroups, which
was first introduced by Peng \cite{P1997} to study stochastic
optimal control problems, and translated later by Buckdahn and Li
\cite{BL2006} to  SDGs.
 For any stopping times $\sigma,\tau$, with $t\leq \sigma\leq\tau\leq
 T$, and a random variable $\eta \in L^2 (\Omega,
{\mathcal{F}}_{\tau},\mathbb{P};{\mathbb{R}})$, we define
\begin{eqnarray*}
 ^{i}G^{t, x; u,v}_{\sigma,\tau} [\eta]:= \ ^{i} \widehat{Y}_\sigma^{t,x; u, v},
\end{eqnarray*}
 where $(\ ^{i}\widehat{Y}_s^{t,x;u, v},
\ ^{i}\widehat{Z}_s^{t,x;u, v}, \ ^{i}\widehat{H}^{t,x; u, v}_s
)_{t\leq s \leq \tau}$ is the solution of the following BSDE:
\begin{eqnarray*}\label{}
   \left \{\begin{array}{rcl}
   -d\ ^{i}\widehat{Y}^{t,x; u, v}_s & = & f_{N_{s}^{t,i}}(s,X^{t,x; u, v}_s,
   \ ^{i}\widehat{Y}^{t,x; u, v}_s,
   \ ^{i}\widehat{H}^{t,x; u, v}_s,\  ^{i}\widehat{Z}^{t,x; u, v}_s,u_s, v_s) ds, \quad s\in [t,\tau]\\
   &&-\lambda\ ^{i}\widehat{H}^{t,x; u, v}_sds
   -\ ^{i}\widehat{Z}^{t,x; u, v}_s dB_s-\ ^{i}\widehat{H}^{t,x; u, v}_s d\widetilde{N}_{s},\\
        ^{i}\widehat{Y}^{t,x; u, v}_{\tau}  & = & \eta,
   \end{array}\right.
\end{eqnarray*}
 where $X^{t,x;u, v}$\ is the solution of SDE (\ref{e3}) with $\zeta=x\in
 \mathbb{R}^{n}$.\vskip2mm

We have the following  dynamic programming principle (DPP) over a
stochastic interval for our games.
\begin{theorem}\label{t5}
Let the assumptions $(H4)$ and  $(H5)$ hold. Then the following
dynamic programming principle holds:
 For any stopping time $\tau$ with
$0\leq t<\tau \leq T,\ x\in {\mathbb{R}}^n, i=1,2,$
\begin{eqnarray*}
W_{i}(t,x)&=&\esssup_{\alpha \in {\mathcal{A}}_{t,
\tau}}\essinf_{\beta \in {\mathcal{B}}_{t, \tau}}\
^{i}G^{t,x;\alpha,\beta}_{t,\tau} [W_{N^{t,i}_{\tau}}(\tau,
X^{t,x;\alpha,\beta}_{\tau})],\\
 U_{i}(t,x) &=&\essinf_{\beta
\in {\mathcal{B}}_{t, \tau}}\esssup_{\alpha \in {\mathcal{A}}_{t,
\tau}}\ ^{i}G^{t,x;\alpha,\beta}_{t,\tau} [U_{N^{t,i}_{\tau}}(\tau,
X^{t,x;\alpha,\beta}_{\tau})].
\end{eqnarray*}
\end{theorem}
This DPP for stopping times will play a crucial role in Section
\ref{NS3}. Since the proof is rather long and technical, we postpone
it to Subsection \ref{A1}.

\section{Probabilistic interpretation of associated coupled systems of Isaacs equations }\label{NS3}
The objective of  this section is to  give a probabilistic
interpretation of coupled systems of Isaacs equations. More
precisely, we show that the value functions  $U=(U_{1},U_{2})$ and
$W=(W_{1},W_{2})$, introduced in Section \ref{NS2}, are viscosity
solutions of the following coupled Isaacs equations:
 \begin{eqnarray}\label{e6}
\left\{
\begin{array}{rcl}
\dfrac{\partial }{\partial t} U_{i}(t,x) +  H_{i}^{+}(t, x,
U_{1}(t,x), U_{2}(t,x),DU_{i}(t,x), D^2U_{i}(t,x))&=&0,
 \quad (t,x)\in [0,T)\times {\mathbb{R}}^n,\\
 U_{i}(T,x)&=&\Phi_{i} (x),  \ i=1,2,
 \end{array}
\right.
\end{eqnarray}
and
\begin{eqnarray}\label{e5}
\left\{
\begin{array}{rcl}
\dfrac{\partial }{\partial t} W_{i}(t,x) +  H_{i}^{-}(t, x,
W_{1}(t,x),W_{2}(t,x),DW_{i}(t,x), D^2W_{i}(t,x))&=&0,
\quad (t,x)\in [0,T)\times {\mathbb{R}}^n,\\
 W_{i}(T,x)&=&\Phi_{i}(x), \ i=1,2,
 \end{array}
\right.
\end{eqnarray}
respectively,  where for $ (t, x, y_{1}, y_{2}, p, A, u,v)\in [0,
T]\times{\mathbb{R}}^n\times\mathbb{R}\times \mathbb{R}\times
\mathbb{R}^{d}\times \mathbb{S}^{d}\times U\times V$,
\begin{eqnarray*}
H_{i}(t, x, y_{1}, y_{2}, p, A, u,v)&=&
\dfrac{1}{2}tr(\sigma\sigma^{T}(t, x, u, v) A)+ p b(t, x, u, v)
  + \widetilde{f}_{i}(t, x, y_{1}, y_{2},  p\sigma(t, x, u, v), u, v),
\end{eqnarray*}
\begin{eqnarray*}
H_{i}^-(t, x, y_{1}, y_{2}, p, A)&=& \sup_{u \in U}\inf_{v \in
V}H_{i}(t, x, y_{1}, y_{2}, p, A, u,v),\\
 H_{i}^+(t, x, y_{1}, y_{2},
p, A)&=&\inf_{v \in V}\sup_{u \in U}H_{i}(t, x, y_{1}, y_{2}, p, A,
u,v).
\end{eqnarray*}
Let  us recall the definition of a viscosity solution of the system
(\ref{e6}). We denote by $C^3_{l, b}([0,T]\times {\mathbb{R}}^n)$
the set of real-valued functions which are continuously
differentiable up to   third order and whose derivatives of order
from 1 to 3 are bounded.
 \begin{definition}\label{d3}
 A continuous function $V=(V_{1},V_{2})\in C([0,T]\times \mathbb{R}^n;\mathbb{R}^2)$ is called \\
  {\rm(i)} a viscosity subsolution of the system (\ref{e6}) if $
V_{i}(T,x) \leq \Phi_{i} (x),\ \mbox{for all}\ i=1,2,  x \in
{\mathbb{R}}^n,$  and if for all test  functions $\varphi \in
C^3_{l, b}([0,T]\times
  {\mathbb{R}}^n), i=1,2,$ and $(t,x) \in [0,T) \times {\mathbb{R}}^n$ such that $V_{i}-\varphi $\ attains
  a local maximum at $(t, x)$,
\begin{eqnarray}\label{v1}
\dfrac{\partial }{\partial t} \varphi(t,x) +  H_{i}^{+}(t, x,
V_{1}(t,x),V_{2}(t,x),D\varphi(t,x), D^2\varphi(t,x))\geq 0,
\end{eqnarray}
\noindent{\rm(ii)} a viscosity supersolution of the system
(\ref{e6}) if $ V_{i}(T,x) \geq \Phi_{i} (x),\ \mbox{for all}\
i=1,2, x \in {\mathbb{R}}^n,$
 and if for all functions $\varphi \in
C^3_{l, b}([0,T]\times
  {\mathbb{R}}^n), i=1,2,$ and $(t,x) \in [0,T) \times {\mathbb{R}}^n$ such that $V_{i}-\varphi $\ attains
  a local minimum at $(t, x)$,
\begin{eqnarray}\label{v2}
\dfrac{\partial }{\partial t} \varphi(t,x) +  H_{i}^{+}(t, x,
V_{1}(t,x),V_{2}(t,x),D\varphi(t,x), D^2\varphi(t,x))\leq 0,
\end{eqnarray}
 {\rm(iii)} a viscosity solution of the system (\ref{e6}) if it is both a viscosity subsolution
  and a supersolution of the system (\ref{e6}). In the same way we can define  a viscosity solution of the system
(\ref{e5}).
\end{definition}
We now give the main result in this section. We shall give the proof
in Section \ref{A2}  since it is rather lengthy.
\begin{theorem}\label{t1}
Let the assumptions $(H4)$ and  $(H5)$ hold. Then
  $U=(U_{1},U_{2})$ (resp., $W=(W_{1},W_{2})$) is a viscosity solution of the system
(\ref{e6}) (resp., (\ref{e5})).
\end{theorem}
To state a uniqueness  theorem for the  viscosity solution of the
system (\ref{e6}), let us first define the  following space:
 \begin{eqnarray*}
 \Theta:&=&\Big\{ \varphi\in C([0, T]\times {\mathbb{R}}^n): \mbox{there exists a constant}\ \ A>0\ \mbox{such
 that}\\ && \qquad \lim_{|x|\rightarrow \infty}|\varphi(t, x)|\exp\{-A[\log((|x|^2+1)^{\frac{1}{2}})]^2\}=0,\
 \mbox{uniformly in}\ t\in [0, T]\Big\}.
\end{eqnarray*}
 \begin{theorem}
 Let the assumptions $(H4)$ and $(H5)$ hold. Then there exists at most one viscosity solution  $u\in
 \Theta$ (resp.,  $v\in \Theta$)
 of the system  (\ref{e6}) (resp., (\ref{e5})).
\end{theorem}
The proof of the Theorem can be adapted from the  arguments in
Barles, Buckdahn and Pardoux \cite{BBP1997} combined with those of
Barles and Imbert \cite{BI2008} to our framework. We omit it here.
\begin{remark}
Since  $U=(U_{1},U_{2})$ (resp., $W=(W_{1},W_{2})$) is a viscosity
solution of linear growth, $U=(U_{1},U_{2})$ (resp.,
$W=(W_{1},W_{2})$) is the unique viscosity solution in $\Theta$ of
the system (\ref{e6}) (resp., (\ref{e5})).
\end{remark}
An immediate consequence of this remark is that, under {\bf Isaacs
condition:}\vskip2mm

For all $ (t, x, y_{1}, y_{2}, p, A, u,v)\in [0,
T]\times{\mathbb{R}}^n\times\mathbb{R}\times \mathbb{R}\times
\mathbb{R}^{d}\times \mathbb{S}^{d}\times U\times V,$  $j=1,2,$ we
have
\begin{eqnarray}\label{Isaacs}
\begin{aligned}
&\sup\limits_{u\in U}\inf\limits_{v\in V}
\Big\{\dfrac{1}{2}tr(\sigma\sigma^{T}(t,
x, u, v) A)+ p b(t, x, u, v)+ f_{j}(t, x, y_{1}, y_{2},  p\sigma(t, x, u, v), u, v) \Big\}\\
=&\inf\limits_{v\in V}\sup\limits_{u\in U}
\Big\{\dfrac{1}{2}tr(\sigma\sigma^{T}(t, x, u, v) A)+ p b(t, x, u,
v)+ f_{j}(t, x, y_{1}, y_{2},  p\sigma(t, x, u, v), u, v) \Big\},
 \end{aligned}
\end{eqnarray}
 the upper value and lower value functions coincide:
\begin{corollary}\label{c1}
Let Isaacs condition  (\ref{Isaacs}) hold. Then we have,  for all $
(t, x)\in [0, T]\times{\mathbb{R}}^n$,
\begin{eqnarray*}
(U_{1}(t,x),U_{2}(t,x))=(W_{1}(t,x),W_{2}(t,x)).
\end{eqnarray*}
\end{corollary}

However, for the next section, we need, in addition to the value
functions we have already introduced, also the following ones.
 For all $(t,x)\in
[0,T]\times\mathbb{R}^n, i=1,2$, we define
\begin{eqnarray*}\label{}
\quad \overline{W}_{i}(t,x):=\esssup_{\beta\in\mathcal {B}_{t,T}}
\essinf_{\alpha\in\mathcal {A}_{t,T}} J_{i}(t,x;\alpha,\beta),\quad
\overline{U}_{i}(t,x):=\essinf_{\alpha\in\mathcal {A}_{t,T}}
\esssup_{\beta\in\mathcal {B}_{t,T}} J_{i}(t,x;\alpha,\beta).
\end{eqnarray*}
{\bf Isaacs condition:}  For all $ (t, x, y_{1}, y_{2}, p, A,
u,v)\in [0, T]\times{\mathbb{R}}^n\times\mathbb{R}\times
\mathbb{R}\times \mathbb{R}^{d}\times \mathbb{S}^{d}\times U\times
V,$  $j=1,2,$
\begin{eqnarray}\label{Isaacs1}
\begin{aligned}
&\inf\limits_{u\in U}\sup\limits_{v\in V}
\Big\{\dfrac{1}{2}tr(\sigma\sigma^{T}(t,
x, u, v) A)+ p b(t, x, u, v)+ f_{j}(t, x, y_{1}, y_{2},  p\sigma(t, x, u, v), u, v) \Big\}\\
=&\sup\limits_{v\in V}\inf\limits_{u\in U}
\Big\{\dfrac{1}{2}tr(\sigma\sigma^{T}(t, x, u, v) A)+ p b(t, x, u,
v)+ f_{j}(t, x, y_{1}, y_{2},  p\sigma(t, x, u, v), u, v) \Big\},
 \end{aligned}
\end{eqnarray}

 Using a similar argument in Sections \ref{NS2} and \ref{NS3} we
 have the following proposition.
\begin{proposition}\label{cp1}
Let Isaacs condition (\ref{Isaacs1}) hold. Then we have,  for all $
(t, x)\in [0, T]\times{\mathbb{R}}^n$,
\begin{eqnarray*}
(\overline{U}_{1}(t,x),\overline{U}_{2}(t,x))=(\overline{W}_{1}(t,x),\overline{W}_{2}(t,x)).
\end{eqnarray*}
\end{proposition}

\begin{remark}\label{rr1}
By virtue of a similar argument in Sections \ref{NS2} and \ref{NS3}
we can get, for $i=1,2,$  $\overline{U}_{i}$ and $\overline{W}_{i}$
 have the similar properties to $U_{i}$ and $W_{i}$, respectively.
 We omit them here. But we will use them in Section \ref{NS4}.
\end{remark}

\section{Nash equilibrium payoffs for nonzero-sum stochastic differential games with coupled cost functionals }\label{NS4}

The objective of this section is to investigate Nash equilibrium
payoffs for nonzero-sum SDGs. An existence theorem and a
characterization theorem of Nash equilibrium payoffs are obtained.

In order to study Nash equilibrium payoffs, we redefine the
following  notations which differ from those of the preceding
sections. Let us define: For all $(t,x)\in [0,T]\times\mathbb{R}^n$,
\begin{eqnarray*}\label{}
\quad W_{1}(t,x):=\esssup_{\alpha\in\mathcal {A}_{t,T}}
\essinf_{\beta\in\mathcal {B}_{t,T}} J_{1}(t,x;\alpha,\beta),\quad
W_{2'}(t,x):=\esssup_{\alpha\in\mathcal {A}_{t,T}}
\essinf_{\beta\in\mathcal {B}_{t,T}} J_{2}(t,x;\alpha,\beta),
\end{eqnarray*}
and
\begin{eqnarray*}\label{}
\quad W_{1'}(t,x):=\esssup_{\beta\in\mathcal {B}_{t,T}}
\essinf_{\alpha\in\mathcal {A}_{t,T}} J_{1}(t,x;\alpha,\beta),\quad
 W_{2}(t,x):=\esssup_{\beta\in\mathcal {B}_{t,T}}
\essinf_{\alpha\in\mathcal {A}_{t,T}} J_{2}(t,x;\alpha,\beta).
\end{eqnarray*}
 In all what
follows we assume that the Isaacs conditions (\ref{Isaacs}) and
(\ref{Isaacs1}) hold, and that all the coefficients are bounded.
This latter assumption is not necessary, but has the objective to
simplify the arguments. Under the Isaacs conditions (\ref{Isaacs})
and (\ref{Isaacs1}), from Corollary \ref{c1} and Proposition
\ref{cp1}, we have, for $(t,x)\in[0,T]\times\mathbb{R}^{n}$,
\begin{eqnarray}\label{Jequa}
W_{1}(t,x)&=& \esssup_{\alpha\in\mathcal {A}_{t,T}}
\essinf_{\beta\in\mathcal {B}_{t,T}}
J_{1}(t,x;\alpha,\beta)=\essinf_{\beta\in\mathcal
{B}_{t,T}}\esssup_{\alpha\in\mathcal {A}_{t,T}}
J_{1}(t,x;\alpha,\beta), \nonumber \\
W_{2}(t,x)&=&\essinf_{\alpha\in\mathcal
{A}_{t,T}}\esssup_{\beta\in\mathcal {B}_{t,T}}
J_{2}(t,x;\alpha,\beta)=\esssup_{\beta\in\mathcal
{B}_{t,T}}\essinf_{\alpha\in\mathcal {A}_{t,T}}
J_{2}(t,x;\alpha,\beta),  \\
 W_{1'}(t,x)&=&\esssup_{\beta\in\mathcal {B}_{t,T}}
\essinf_{\alpha\in\mathcal {A}_{t,T}} J_{1}(t,x;\alpha,\beta)=
\essinf_{\alpha\in\mathcal {A}_{t,T}}\esssup_{\beta\in\mathcal
{B}_{t,T}} J_{1}(t,x;\alpha,\beta),\nonumber \\
  W_{2'}(t,x)&=&\esssup_{\alpha\in\mathcal {A}_{t,T}}
\essinf_{\beta\in\mathcal {B}_{t,T}} J_{2}(t,x;\alpha,\beta)=
\essinf_{\beta\in\mathcal {B}_{t,T}}\esssup_{\alpha\in\mathcal
{A}_{t,T}} J_{2}(t,x;\alpha,\beta).\nonumber
\end{eqnarray}
\begin{remark}
Not only  Hamiltonians of the form $\overline{H}_{1}(t, x, y_{1},
y_{2}, p, A, u)+\overline{H}_{2}(t, x, y_{1}, y_{2}, p, A, v),$ for
$(t, x, y_{1}, y_{2}, p, A, u,v)\in [0,
T]\times{\mathbb{R}}^n\times\mathbb{R}\times \mathbb{R}\times
\mathbb{R}^{d}\times \mathbb{S}^{d}\times U\times V$,
   are covered by Isaacs conditions (\ref{Isaacs}) and
(\ref{Isaacs1}), as the following example shows:

$$\sup_{v\in[-1,1]}\inf_{u\in[0,1]} vu^{2}=\inf_{u\in[0,1]}\sup_{v\in[-1,1]} vu^{2}=0,$$
and
$$\inf_{v\in[-1,1]}\sup_{u\in[0,1]} vu^{2}=\sup_{u\in[0,1]}\inf_{v\in[-1,1]} vu^{2}=0.$$
\end{remark}
We also define the following function: for $j=1,2, l=1,2,$
\begin{equation*}\label{}
n_{j}(l)=\left\{
\begin{array}{rl}
 j, \qquad  l=j,\\
l',  \qquad l\neq j.
\end{array}
\right.
\end{equation*}

We now give the definition of Nash equilibrium payoffs for
nonzero-sum SDGs, which is similar to that in \cite{BCR2004} and
\cite{L2011}.
\begin{definition}\label{Jd2}
For $(t,x)\in[0,T]\times\mathbb{R}^{n}$,  a couple
$(e_{1},e_{2})\in\mathbb{R}^{2}$ is called a Nash equilibrium payoff
at the point $(t,x)$,  if for any $\varepsilon>0$, there exists a
couple $(\alpha_{\varepsilon},\beta_{\varepsilon})\in \mathcal
{A}_{t,T}\times \mathcal {B}_{t,T}$ such that, for all
$(\alpha,\beta)\in \mathcal {A}_{t,T}\times \mathcal {B}_{t,T},$
\begin{eqnarray}\label{Jeq4}
J_{1}(t,x;\alpha_{\varepsilon},\beta_{\varepsilon})\geq
J_{1}(t,x;\alpha,\beta_{\varepsilon})-\varepsilon,\
J_{2}(t,x;\alpha_{\varepsilon},\beta_{\varepsilon})\geq
J_{2}(t,x;\alpha_{\varepsilon},\beta)-\varepsilon,\ \mathbb{P}-a.s.,
\end{eqnarray}
and
\begin{eqnarray*}
|\mathbb{E}[J_{j}(t,x;\alpha_{\varepsilon},\beta_{\varepsilon})]-e_{j}|\leq
\varepsilon, \ j=1,2.
\end{eqnarray*}
\end{definition}

\begin{remark}
We notice that unlike in \cite{BCR2004} our cost functionals
$J_{j}(t,x;\alpha,\beta), j=1,2,$ are not necessarily deterministic.
Indeed, while \cite{BCR2004} is based on the approach by Fleming and
Souganidis \cite{FS1989} in which the admissible cost functionals
for a game over the fixed time interval $[t,T]$ are independent of
$\mathcal {F}_{t}$, the present paper is based on the approaches
developed in \cite{BL2006}, \cite{BCQ2011} and in  \cite{BHL2010}.
\end{remark}

The following equivalent condition of (\ref{Jeq4}) follows easily
 from  Lemma \ref{l5}.
\begin{lemma}\label{Jle3}
For any $\varepsilon>0$, let
$(\alpha_{\varepsilon},\beta_{\varepsilon})\in \mathcal
{A}_{t,T}\times \mathcal {B}_{t,T}$. Then (\ref{Jeq4}) holds if and
only if,  for all $(u,v)\in \mathcal {U}_{t,T}\times \mathcal
{V}_{t,T}$,
\begin{eqnarray}\label{Jeq5}
\begin{aligned}
J_{1}(t,x;\alpha_{\varepsilon},\beta_{\varepsilon})\geq
J_{1}(t,x;u,\beta_{\varepsilon}(u))-\varepsilon, \quad
\mathbb{P}-a.s.,\\
J_{2}(t,x;\alpha_{\varepsilon},\beta_{\varepsilon})\geq
J_{2}(t,x;\alpha_{\varepsilon}(v),v)-\varepsilon, \quad
\mathbb{P}-a.s.
\end{aligned}
\end{eqnarray}
\end{lemma}
As in \cite{L2011}, before giving the characterization theorem of
Nash equilibrium payoffs, we first state  two important lemmas.
\begin{lemma}\label{Jle1}
For $(t,x)\in [0,T]\times \mathbb{R}^{n}$, we fix arbitrarily $u\in
\mathcal {U}_{t,T}$. Then,

(i) for all stopping time $\tau\in[t,T]$ and $\varepsilon>0$, there
exists an NAD strategy $\alpha\in \mathcal {A}_{t,T}$ such that, for
all
 $v\in \mathcal {V}_{t,T}$,
\begin{eqnarray*}\label{}
\alpha(v)&=&u, \text{on}\ [[t,\tau]],\\
^{2}Y^{t,x;\alpha(v),v}_{\tau}&\leq&
W_{n_{2}(N^{t,2}_{\tau})}(\tau,X^{t,x;\alpha(v),v}_{\tau})+\varepsilon,\
\mathbb{P}-a.s.
\end{eqnarray*}

(ii) for all stopping time $\tau\in[t,T]$ and $\varepsilon>0$, there
exists an NAD strategy  $\alpha\in \mathcal {A}_{t,T}$ such that,
for all
 $v\in \mathcal {V}_{t,T}$,
\begin{eqnarray*}\label{}
\alpha(v)&=&u, \text{on}\ [[t,\tau]],\\
^{1}Y^{t,x;\alpha(v),v}_{\tau}&\geq&
W_{n_{1}(N^{t,1}_{\tau})}(\tau,X^{t,x;\alpha(v),v}_{\tau})-\varepsilon,
\ \mathbb{P}-a.s.
\end{eqnarray*}
\end{lemma}

\begin{proof}
Let us only give the proof of (i); that of  (ii) can be carried out
with a similar argument. We notice that $\mathcal {V}_{\tau,T}$ can
be regarded as a subset of $\mathcal {B}_{\tau,T}$. Indeed, putting
$\beta^{v'}(u')=v'$,  $u'\in \mathcal {U}_{\tau,T}$, we associate
all $v'\in \mathcal {V}_{\tau,T}$ with some $\beta^{v'}\in \mathcal
{B}_{\tau,T}$. Therefore, for any $y\in \mathbb{R}^{n}$, similar to
Proposition \ref{p9} we have
\begin{eqnarray*}\label{}
W_{n_{2}(N^{t,2}_{\tau})}(\tau,y)&=&\essinf_{\alpha\in\mathcal
{A}_{\tau,T}}\esssup_{\beta\in\mathcal {B}_{\tau,T}}
J_{N^{t,2}_{\tau}}(\tau,y;\alpha,\beta)\\&\geq&\essinf_{\alpha\in\mathcal
{A}_{\tau,T}}\esssup_{v\in\mathcal {V}_{\tau,T}}
J_{N^{t,2}_{\tau}}(\tau,y;\alpha(v),v), \ \mathbb{P}-a.s.
\end{eqnarray*}
Then, for any $\varepsilon_{0}>0$, by standard arguments we have the
existence of $\alpha_{y}\in \mathcal {A}_{\tau,T}$ such that
\begin{eqnarray}\label{Jeqn1}
W_{n_{2}(N^{t,2}_{\tau})}(\tau,y)&\geq& \esssup_{v\in\mathcal
{V}_{\tau,T}}J_{N^{t,2}_{\tau}}(\tau,y;
\alpha_{y}(v),v)-\varepsilon_{0}, \ \mathbb{P}-a.s.
\end{eqnarray}
We let $\{O_{i}\}_{i\geq1}\subset\mathcal {B}(\mathbb{R}^{n})$  be a
partition of $\mathbb{R}^{n}$ such that
$\sum\limits_{i\geq1}O_{i}=\mathbb{R}^{n}, O_{i}\neq\emptyset,$ and
$\text{diam}(O_{i})\leq\varepsilon_{0},i\geq1$. For $y_{i}\in O_{i},
i\geq 1$ and  $v\in\mathcal {V}_{t,T}$, let us set
\begin{equation}\label{Jequation}
\alpha(v)_{s}=\left\{
\begin{array}{rl}
 u_s, \qquad \qquad\qquad\qquad s\in [[t, \tau]]
,\\
\sum\limits_{i\geq1}1_{\{X^{t,x;u,v}_{\tau}\in
O_{i}\}}\alpha_{y_{i}}(v|_{[\tau,T]})_{s},\qquad s\in ]]\tau,T]].
\end{array}
\right.
\end{equation}
Then  $\alpha:\mathcal {V}_{t,T}\rightarrow \mathcal {U}_{t,T}$ is
an NAD strategy. This can be checked in a straight-forward way and
is omitted in order to shorten the proof.

By virtue of  Proposition \ref{p4}, Remark \ref{rr1}, (\ref{Jeqn1})
and (\ref{Jequation}), we deduce that, for $v\in\mathcal {V}_{t,T}$,
\begin{eqnarray*}\label{}
&&W_{n_{2}(N^{t,2}_{\tau})}(\tau,X^{t,x;\alpha(v),v}_{\tau})\\
&\geq&\sum\limits_{i\geq1}1_{\{X^{t,x;\alpha(v),v}_{\tau}\in
O_{i}\}}W_{n_{2}(N^{t,2}_{\tau})}(\tau,y_{i})-C\varepsilon_{0}\\
&\geq& \sum\limits_{i\geq1}1_{\{X^{t,x;\alpha(v),v}_{\tau}\in
O_{i}\}} J_{N^{t,2}_{\tau}}(\tau,y_{i};
\alpha_{y_{i}}(v|_{[\tau,T]}),v)-C\varepsilon_{0}\\
&=& \sum\limits_{i\geq1}1_{\{X^{t,x;\alpha(v),v}_{\tau}\in O_{i}\}}
J_{N^{t,2}_{\tau}}(\tau,y_{i};\alpha(v),v)-C\varepsilon_{0}.
\end{eqnarray*}
Consequently, due to Proposition \ref{p3} we conclude
\begin{eqnarray*}\label{}
W_{n_{2}(N^{t,2}_{\tau})}(\tau,X^{t,x;\alpha(v),v}_{\tau}) &\geq&
\sum\limits_{i\geq1}1_{\{X^{t,x;\alpha(v),v}_{\tau}\in O_{i}\}}
J_{N^{t,2}_{\tau}}(\tau,X^{t,x;\alpha(v),v}_{\tau};
\alpha(v),v)-C\varepsilon_{0}\\
&=&  J_{N^{t,2}_{\tau}}(\tau,X^{t,x;\alpha(v),v}_{\tau};
\alpha(v),v)-C\varepsilon_{0},
\end{eqnarray*}
where $C$ is a constant which can be different from line to line and
is independent of $v\in\mathcal {V}_{t,T}$.

 Recalling that $\varepsilon_{0}>0$ hasn't been specified yet, let us choose
$\varepsilon_{0}=C^{-1}\varepsilon$. We observe that
\begin{eqnarray*}\label{}
 J_{N^{t,2}_{\tau}}(\tau,X^{t,x;\alpha(v),v}_{\tau};
\alpha(v),v)=\
^{N^{t,2}_{\tau}}Y^{\tau,X^{t,x;\alpha(v),v}_{\tau},\alpha(v),v}_{\tau}=\
^{2}Y^{t,x,\alpha(v),v}_{\tau}.
\end{eqnarray*}
 Then we deduce that
$$W_{n_{2}(N^{t,2}_{\tau})}(\tau,X^{t,x;\alpha(v),v}_{\tau})\geq
 \ ^{2}Y^{t,x,\alpha(v),v}_{\tau}-\varepsilon,\  v\in\mathcal
 {V}_{t,T}.$$
This conclude the proof.
\end{proof}

The following lemma follows from standard estimates for SDEs.
\begin{lemma}\label{Jle2}
There exists a positive constant  $C$ such that, for all
$(u,v),(u',v')\in \mathcal {U}_{t,T}\times \mathcal {V}_{t,T}$, and
for all $\mathbb{F}$-stopping times $\sigma:\Omega\rightarrow[t,T]$
with $X^{t,x;u,v}_{\sigma}=X^{t,x;u',v'}_{\sigma}$,
$\mathbb{P}-a.s.,$ we have the following estimate:  for all real
$r\in[t,T],$
\begin{eqnarray*}\label{}
\mathbb{E}[\sup\limits_{0\leq s\leq r}|X^{t,x;u,v}_{(\sigma+s)\wedge
T}-X^{t,x;u',v'}_{(\sigma+s)\wedge T}|^{2} \Big|\mathcal
{F}_{t}]\leq C r, \ \mathbb{P}-a.s.
\end{eqnarray*}
\end{lemma}

Let us  now give the following characterization theorem of Nash
equilibrium payoffs for two-player nonzero-sum SDGs.
\begin{theorem}\label{Jt6}
For  $(t,x)\in [0,T]\times \mathbb{R}^{n}$, a couple
$(e_{1},e_{2})\in\mathbb{R}^{2}$ is a Nash equilibrium payoff at
point $(t,x)$ if and only if, for all $\varepsilon>0$, there exists
 a couple $(u^{\varepsilon},v^{\varepsilon})\in \mathcal {U}_{t,T}\times\mathcal
 {V}_{t,T}$  such that for all  $\delta\in [0,T-t]$  and $j=1,2,$
\begin{eqnarray}\label{Jeq6}
 \mathbb{P}\Big(\
^{N^{t,j}_{t+\delta}}\widetilde{Y}^{t,x;u^{\varepsilon},v^{\varepsilon}}_{t+\delta}\geq
W_{n_{j}(N_{t+\delta}^{t,j})}(t+\delta,X^{t,x;u^{\varepsilon},v^{\varepsilon}}_{t+\delta})-\varepsilon\
|\ \mathcal {F}_{t}\Big)\geq 1-\varepsilon,\ \mathbb{P}-a.s.
\end{eqnarray} and
\begin{eqnarray}\label{Jeq7}
|\mathbb{E}[J_{j}(t,x;u^{\varepsilon},v^{\varepsilon})]-
e_{j}|\leq\varepsilon.
\end{eqnarray}
\end{theorem}
\begin{remark}
The  characterization theorem of Nash equilibrium payoffs in
\cite{L2011} generalizes the results in \cite{BCR2004}  from
classical cost functionals without running cost to nonlinear cost
functionals with running cost defined by   decoupled BSDEs. The
above theorem on its part extends the  result in \cite{L2011}: on
one hand, we generalize \cite{L2011} from SDGs without jumps to
those with jumps. On the other hand, our cost functionals are
defined by a system of two coupled BSDEs.
\end{remark}
\noindent{\bf Proof of Theorem \ref{Jt6}: \underline {Necessity}} of
(\ref{Jeq6}) and (\ref{Jeq7}).
\begin{proof}
Let us suppose that $(e_{1},e_{2})\in\mathbb{R}^{2}$ is a Nash
equilibrium payoff at the point $(t,x)$. Then, by Definition
\ref{Jd2}  we have, for sufficiently small $\varepsilon>0$, the
existence of $(\alpha_{\varepsilon},\beta_{\varepsilon})\in \mathcal
{A}_{t,T}\times \mathcal {B}_{t,T}$ such that, for all
$(\alpha,\beta)\in \mathcal {A}_{t,T}\times \mathcal {B}_{t,T}$
\begin{eqnarray}\label{Je26}
J_{1}(t,x;\alpha_{\varepsilon},\beta_{\varepsilon})\geq
J_{1}(t,x;\alpha,\beta_{\varepsilon})-\varepsilon^{4},\
J_{2}(t,x;\alpha_{\varepsilon},\beta_{\varepsilon})\geq
J_{2}(t,x;\alpha_{\varepsilon},\beta)-\varepsilon^{4},
\mathbb{P}-a.s.,
\end{eqnarray}
and
\begin{eqnarray}\label{Jeqnarray1}
|\mathbb{E}[J_{j}(t,x;\alpha_{\varepsilon},\beta_{\varepsilon})]-e_{j}|\leq
\varepsilon^{4}, \ j=1,2.
\end{eqnarray}
Since $(\alpha_{\varepsilon},\beta_{\varepsilon})\in \mathcal
{A}_{t,T}\times \mathcal {B}_{t,T}$, it follows from Lemma \ref{l5}
that there exists a unique couple
$(u^{\varepsilon},v^{\varepsilon})\in\mathcal {U}_{t,T}\times
\mathcal {V}_{t,T}$ such that
$\alpha_{\varepsilon}(v^{\varepsilon})=u^{\varepsilon},
\beta_{\varepsilon}(u^{\varepsilon})=v^{\varepsilon}.$  We notice
that from Proposition \ref{p10},   (\ref{Jeq6}) is equivalent to
\begin{eqnarray*}\label{}
 \mathbb{P}\Big(\
^{j}Y^{t,x;u^{\varepsilon},v^{\varepsilon}}_{t+\delta}\geq
W_{n_{j}(N_{t+\delta}^{t,j})}(t+\delta,X^{t,x;u^{\varepsilon},v^{\varepsilon}}_{t+\delta})-\varepsilon\
|\ \mathcal {F}_{t}\Big)\geq 1-\varepsilon,\ \mathbb{P}-a.s.
\end{eqnarray*}
We make the proof by contradiction and  suppose that (\ref{Jeq6})
doesn't hold true. Then, for all $\varepsilon'>0$, there exists some
$\varepsilon\in(0,\varepsilon')$ and some  $\delta\in[0,T-t]$ such
that, for some $j\in\{1,2\}$, say for $j=1$,
\begin{eqnarray}\label{}
\mathbb{P}\Big(\mathbb{P}(\
^{1}Y^{t,x;u^{\varepsilon},v^{\varepsilon}}_{t+\delta}<
W_{n_{1}(N_{t+\delta}^{t,1})}(t+\delta,X^{t,x;u^{\varepsilon},v^{\varepsilon}}_{t+\delta})-\varepsilon\
|\ \mathcal {F}_{t})> \varepsilon\Big)>0.
\end{eqnarray}
Putting
\begin{eqnarray}\label{Je23}
A=\Big\{\ ^{1}Y^{t,x;u^{\varepsilon},v^{\varepsilon}}_{t+\delta}<
W_{n_{1}(N_{t+\delta}^{t,1})}(t+\delta,X^{t,x;u^{\varepsilon},v^{\varepsilon}}_{t+\delta})-\varepsilon\
\Big\}\in\mathcal {F}_{t+\delta},
\end{eqnarray}
and  applying  Lemma \ref{Jle1} to $u^{\varepsilon}$ and $t+\delta$,
we have the existence of  an NAD strategy $\widetilde{\alpha} \in
\mathcal {A}_{t,T}$ such that, for all $v\in \mathcal {V}_{t,T}$,
\begin{eqnarray}\label{Jeq18}
\widetilde{\alpha}(v)&=&u^{\varepsilon}, \text{on}\ [t,t+\delta],\nonumber\\
^{1}Y^{t,x;\widetilde{\alpha}(v),v}_{t+\delta}&\geq&
W_{n_{1}(N_{t+\delta}^{t,1})}(t+\delta,X^{t,x;\widetilde{\alpha}(v),v}_{t+\delta})-\frac{\varepsilon}{2},\
\mathbb{P}-a.s.
\end{eqnarray}
From  Lemma \ref{l5} we have the existence of  a unique couple
$(u,v)\in \mathcal {U}_{t,T}\times \mathcal {V}_{t,T}$ such that
$\widetilde{\alpha}(v)=u,\  \beta_{\varepsilon}(u)=v.$  From
(\ref{Jeq18}) we see that  $u=u^{\varepsilon}$ on $[t,t+\delta]$,
and we put
\begin{equation*}\label{}
\widetilde{u}=\left\{
\begin{array}{ll}
 u^{\varepsilon}, &\ \text{on}\quad ([t,t+\delta[\times\Omega)\cup([t+\delta,T]\times A^{c}),\\
u, & \ \text{on}\quad  [t+\delta,T]\times A.
\end{array}
\right.
\end{equation*}
Consequently, from the nonanticipativity of  $
\beta_{\varepsilon}\in\mathcal {B}_{t,T}$ we see that
 $\beta_{\varepsilon}(\widetilde{u})=\beta_{\varepsilon}(u^{\varepsilon})=v^{\varepsilon}$ on
$[t,t+\delta]$, and for all $s\in[t+\delta,T]$,
\begin{equation*}\label{}
\beta_{\varepsilon}(\widetilde{u})_{s}=\left\{
\begin{array}{ll}
 \beta_{\varepsilon}(u)_{s}=v_{s}, & \text{on}\ A,\\
\beta_{\varepsilon}(u^{\varepsilon})_{s}=v^{\varepsilon}_{s}, &
\text{on}\ A^{c}.
\end{array}
\right.
\end{equation*}
Hence, standard arguments for SDEs and BSDEs yield
\begin{equation*}\label{}
X^{t,x;\widetilde{u},\beta_{\varepsilon}(\widetilde{u})}=X^{t,x;u^{\varepsilon},v^{\varepsilon}},\
\text{on}\ [t,t+\delta],
\end{equation*}
\begin{equation*}\label{}
X^{t,x;\widetilde{u},\beta_{\varepsilon}(\widetilde{u})}=\left\{
\begin{array}{ll}
 X^{t,x;\widetilde{\alpha}(v),v}, & \text{on}\ [t+\delta,T]\times A,\\
X^{t,x;u^{\varepsilon},v^{\varepsilon}}, & \text{on}\
[t+\delta,T]\times A^{c},
\end{array}
\right.
\end{equation*}
as well as
\begin{equation*}\label{}
^{1}Y^{t,x;\widetilde{u},\beta_{\varepsilon}(\widetilde{u})}_{t+\delta}=\left\{
\begin{array}{ll}
 ^{1}Y^{t,x;\widetilde{\alpha}(v),v}_{t+\delta}, & \text{on}\  A,\\
^{1}Y^{t,x;u^{\varepsilon},v^{\varepsilon}}_{t+\delta}, & \text{on}\
 A^{c}.
\end{array}
\right.
\end{equation*}
Thus,
\begin{eqnarray*}\label{}
 J_{1}(t,x;\widetilde{u},\beta_{\varepsilon}(\widetilde{u}))&=&\ ^
 {1}Y^{t,x;\widetilde{u},\beta_{\varepsilon}(\widetilde{u})}_{t}=
 \ ^{1}G^{t,x;\widetilde{u},\beta_{\varepsilon}(\widetilde{u})}_{t,t+\delta}
 [^{1}Y^{t,x;\widetilde{u},\beta_{\varepsilon}(\widetilde{u})}_{t+\delta}]\\
&=&^{1}G^{t,x;\widetilde{u},\beta_{\varepsilon}(\widetilde{u})}_{t,t+\delta}
 [^{1}Y^{t,x;\widetilde{u},\beta_{\varepsilon}(\widetilde{u})}_{t+\delta}1_{A}
 +\ ^{1}Y^{t,x;\widetilde{u},\beta_{\varepsilon}(\widetilde{u})}_{t+\delta}1_{A^{c}}]\\
&=&^{1}G^{t,x;\widetilde{u},\beta_{\varepsilon}(\widetilde{u})}_{t,t+\delta}
 [ ^{1}Y^{t,x;\widetilde{\alpha}(v),v}_{t+\delta}1_{A}
 +\
 ^{1}Y^{t,x;u^{\varepsilon},v^{\varepsilon}}_{t+\delta}1_{A^{c}}].
\end{eqnarray*}
By virtue of Lemma \ref{l8} and   (\ref{Jeq18}) we    conclude
\begin{eqnarray*}\label{}
 J_{1}(t,x;\widetilde{u},\beta_{\varepsilon}(\widetilde{u}))
 &\geq&^{1}G^{t,x;\widetilde{u},\beta_{\varepsilon}(\widetilde{u})}_{t,t+\delta}
 [(W_{n_{1}(N_{t+\delta}^{t,1})}(t+\delta,X^{t,x;\widetilde{\alpha}(v),v}_{t+\delta})-\frac{\varepsilon}{2})1_{A}
 +\ ^{1}Y^{t,x;u^{\varepsilon},v^{\varepsilon}}_{t+\delta}1_{A^{c}}]\\
  &=&^{1}G^{t,x;\widetilde{u},\beta_{\varepsilon}(\widetilde{u})}_{t,t+\delta}
 [W_{n_{1}(N_{t+\delta}^{t,1})}(t+\delta,X^{t,x;u^{\varepsilon},v^{\varepsilon}}_{t+\delta})1_{A}
 +\ ^{1}Y^{t,x;u^{\varepsilon},v^{\varepsilon}}_{t+\delta}1_{A^{c}}-\frac{\varepsilon}{2}1_{A}].
\end{eqnarray*}
Therefore,  from (\ref{Je23}) we deduce that
\begin{eqnarray}\label{Jeqn6}
 J_{1}(t,x;\widetilde{u},\beta_{\varepsilon}(\widetilde{u}))
  &\geq&^{1}G^{t,x;\widetilde{u},\beta_{\varepsilon}(\widetilde{u})}_{t,t+\delta}
 [ (^{1}Y^{t,x;u^{\varepsilon},v^{\varepsilon}}_{t+\delta}+\varepsilon)1_{A}
 +\ ^{1}Y^{t,x;u^{\varepsilon},v^{\varepsilon}}_{t+\delta}1_{A^{c}}-\frac{\varepsilon}{2}1_{A}]\nonumber\\
   &=&^{1}G^{t,x;\widetilde{u},\beta_{\varepsilon}(\widetilde{u})}_{t,t+\delta}
 [^{1}Y^{t,x;u^{\varepsilon},v^{\varepsilon}}_{t+\delta}+\frac{\varepsilon}{2}1_{A}]\nonumber\\
 & =&\ ^{1}G^{t,x;u^{\varepsilon},v^{\varepsilon}}_{t,t+\delta}
 [^{1}Y^{t,x;u^{\varepsilon},v^{\varepsilon}}_{t+\delta}+\frac{\varepsilon}{2}1_{A}].
\end{eqnarray}
Putting  $$y_{s}=\
^{1}G^{t,x;u^{\varepsilon},v^{\varepsilon}}_{s,t+\delta}
 [\
 ^{1}Y^{t,x;u^{\varepsilon},v^{\varepsilon}}_{t+\delta}+\dfrac{\varepsilon}{2}1_{A}],\
 s\in [t,t+\delta],$$
let us  consider  the following BSDE:
\begin{eqnarray*}\label{}
y_s &=&
^{1}Y^{t,x;u^{\varepsilon},v^{\varepsilon}}_{t+\delta}+\dfrac{\varepsilon}{2}1_{A}
+\displaystyle \int_s^{t+\delta}
f_{N_{r}^{t,1}}(r,X^{t,x;u^{\varepsilon},v^{\varepsilon}}_r,
y_r,h_r,z_r,u^{\varepsilon}_{r},v^{\varepsilon}_{r})dr\\&&\qquad-\lambda\displaystyle
\int_s^{t+\delta} h_rdr-\displaystyle \int_s^{t+\delta}
z_rdB_r-\displaystyle \int_s^{t+\delta} h_rd\widetilde{N}_r, \qquad
s\in [t,t+\delta],
\end{eqnarray*}
as well as, for $s\in [t,t+\delta],$
\begin{eqnarray*}\label{}
\begin{aligned}
^{1}Y^{t,x;u^{\varepsilon},v^{\varepsilon}}_s &= \
^{1}Y^{t,x;u^{\varepsilon},v^{\varepsilon}}_{t+\delta}
+\displaystyle \int_s^{t+\delta}
f_{N_{r}^{t,1}}(r,X^{t,x;u^{\varepsilon},v^{\varepsilon}}_r,\
^{1}Y^{t,x;u^{\varepsilon},v^{\varepsilon}}_r,\
^{1}H^{t,x;u^{\varepsilon},v^{\varepsilon}}_r,
\ ^{1}Z^{t,x;u^{\varepsilon},v^{\varepsilon}}_r,u^{\varepsilon}_r, v^{\varepsilon}_r)dr\\
&\quad -\lambda\displaystyle \int_s^{t+\delta} \
^{1}H^{t,x;u^{\varepsilon},v^{\varepsilon}}_rdr-\displaystyle
\int_s^{t+\delta} \
^{1}Z^{t,x;u^{\varepsilon},v^{\varepsilon}}_rdB_r-\displaystyle
\int_s^{t+\delta} \
^{1}H^{t,x;u^{\varepsilon},v^{\varepsilon}}_rd\widetilde{N}_r.
\end{aligned}
\end{eqnarray*}
In order to simplify the notations we suppose until the end of this
proof that the dimension of the Brownian motion  $d$ is equal to
$1$, since we can use a similar arguments for $d>1$. Let us set
$$\overline{y}_{s}=y_{s}-\ ^{1}Y^{t,x;u^{\varepsilon},v^{\varepsilon}}_{s},\
 \overline{h}_{s}=h_{s}-\ ^{1}H^{t,x;u^{\varepsilon},v^{\varepsilon}}_{s},
  \overline{z}_{s}=z_{s}-\ ^{1}Z^{t,x;u^{\varepsilon},v^{\varepsilon}}_{s}, \quad s\in
[t,t+\delta].$$ Then we conclude
\begin{eqnarray}\label{Je24}
\overline{y}_s &=& \dfrac{\varepsilon}{2}1_{A}+\displaystyle
\int_s^{t+\delta}
[f_{N_{r}^{t,1}}(r,X^{t,x;u^{\varepsilon},v^{\varepsilon}}_r,
y_r,h_r,z_r,u^{\varepsilon}_{r},v^{\varepsilon}_{r})\nonumber\\
&&\quad-f_{N_{r}^{t,1}}(r,X^{t,x;u^{\varepsilon},v^{\varepsilon}}_r,\
^{1}Y^{t,x;u^{\varepsilon},v^{\varepsilon}}_r,\
^{1}H^{t,x;u^{\varepsilon},v^{\varepsilon}}_r, \
^{1}Z^{t,x;u^{\varepsilon},v^{\varepsilon}}_r,u^{\varepsilon}_r,
v^{\varepsilon}_r)]dr\nonumber\\&&\qquad -\lambda\displaystyle
\int_s^{t+\delta} \overline{h}_rdr-\displaystyle \int_s^{t+\delta}
\overline{z}_rdB_r-\displaystyle \int_s^{t+\delta}
\overline{h}_rd\widetilde{N}_r, \quad s\in [t,t+\delta].
\end{eqnarray}
Let us put, for  $r\in [t,t+\delta]$,
\begin{eqnarray*}
&&a_{r}=1_{\{\overline{y}_{r}\neq
0\}}(\overline{y}_{r})^{-1}\Big(f_{N_{r}^{t,1}}(r,X^{t,x;u^{\varepsilon},v^{\varepsilon}}_r,
y_r,h_r,z_r,u^{\varepsilon}_{r},v^{\varepsilon}_{r})
\\ && \qquad \qquad  \qquad \qquad \qquad \qquad -f_{N_{r}^{t,1}}(r,X^{t,x;u^{\varepsilon},v^{\varepsilon}}_r, \
^{1}Y^{t,x;u^{\varepsilon},v^{\varepsilon}}_r,
h_r,z_r,u^{\varepsilon}_r,v^{\varepsilon}_r)\Big),\\
&&b_{r}=1_{\{\overline{z}_{r}\neq
0\}}(\overline{z}_{r})^{-1}\Big(f_{N_{r}^{t,1}}(r,X^{t,x;u^{\varepsilon},v^{\varepsilon}}_r,
\ ^{1}Y^{t,x;u^{\varepsilon},v^{\varepsilon}}_r,\
^{1}H^{t,x;u^{\varepsilon},v^{\varepsilon}}_r,
z_r,u^{\varepsilon}_r,v^{\varepsilon}_r)\\&&\qquad \qquad \qquad
\qquad \qquad
-f_{N_{r}^{t,1}}(r,X^{t,x;u^{\varepsilon},v^{\varepsilon}}_r, \
^{1}Y^{t,x;u^{\varepsilon},v^{\varepsilon}}_r, \
^{1}H^{t,x;u^{\varepsilon},v^{\varepsilon}}_r,\
^{1}Z^{t,x;u^{\varepsilon},v^{\varepsilon}}_r,u^{\varepsilon}_r,
v^{\varepsilon}_r)\Big),\\
&&c_{r}=1_{\{\overline{h}_{r}\neq
0\}}(\overline{h}_{r})^{-1}\Big(f_{N_{r}^{t,1}}(r,X^{t,x;u^{\varepsilon},v^{\varepsilon}}_r,
\ ^{1}Y^{t,x;u^{\varepsilon},v^{\varepsilon}}_r, h_r,
z_r,u^{\varepsilon}_r,v^{\varepsilon}_r)\\&&\qquad \qquad \qquad
\qquad \qquad \qquad  \qquad
-f_{N_{r}^{t,1}}(r,X^{t,x;u^{\varepsilon},v^{\varepsilon}}_r, \
^{1}Y^{t,x;u^{\varepsilon},v^{\varepsilon}}_r, \
^{1}H^{t,x;u^{\varepsilon},v^{\varepsilon}}_r, z_r,
u^{\varepsilon}_r, v^{\varepsilon}_r)\Big).
\end{eqnarray*}
Then, from  $(H5)$ we deduce that $|a_{r}|\leq C $, $|b_{r}|\leq C,
|c_{r}|\leq C,$ and $\widehat{c}_{r}:=c_{r}-\lambda\geq
K-\lambda>-1,\  r\in[t,t+\delta].$
 Consequently, there exists a constant $\varepsilon_{0}>0$
small enough,  such that $\widehat{c}_{r}\geq -1+\varepsilon_{0},\
r\in[t,t+\delta],$  and BSDE (\ref{Je24}) can be written as follows:
\begin{eqnarray}\label{Je25}
\overline{y}_s  = \dfrac{\varepsilon}{2}1_{A} +\displaystyle
\int_s^{t+\delta}
[a_{r}\overline{y}_r+b_{r}\overline{z}_r+\widehat{c}_{r}\overline{h}_r]dr-\displaystyle
\int_s^{t+\delta} \overline{z}_rdB_r-\displaystyle \int_s^{t+\delta}
\overline{h}_rd\widetilde{N}_r.
\end{eqnarray}
Let us put
\begin{eqnarray*}
M_{t,t+\delta}=\exp\Big(\int_{t}^{t+\delta}b_{r}dB_{r}
-\frac{1}{2}\int_{t}^{t+\delta}|b_{r}|^{2}dr-\lambda\int_{t}^{t+\delta}\widehat{c}_{r}dr\Big)
\prod\limits_{t<r\leq t+\delta}(1+\widehat{c}_{r}\Delta N_{r}).
\end{eqnarray*}
Then from the Girsanov theorem we know that there exists a
probability measure $\mathbb{Q}=M_{t,t+\delta}\cdot\mathbb{P}$
defined on $(\Omega,\mathcal {F})$ such that
\begin{eqnarray*}
\widehat{N}_{s}=\widetilde{N}_{s}-\int_{t}^{s\wedge(t+\delta)}\widehat{c}_{r}dr,\
s\in [t,T],
\end{eqnarray*}
is an $(\mathbb{F},\mathbb{Q})$-martingale, and
\begin{eqnarray*}
\widehat{B}_{s}=B_{s}-\int_{t}^{s\wedge(t+\delta)}b_{r}dr,\ s\in
[t,T],
\end{eqnarray*}
is an $(\mathbb{F},\mathbb{Q})$-Brownian motion,  and both are
independent under $\mathbb{Q}$. Therefore,  (\ref{Je25})  takes the
following form:
\begin{eqnarray*}\label{}
\overline{y}_s &=& \dfrac{\varepsilon}{2}1_{A} +\displaystyle
\int_s^{t+\delta} a_{r}\overline{y}_rdr\nonumber-\displaystyle
\int_s^{t+\delta} \overline{z}_rd\widehat{B}_r-\displaystyle
\int_s^{t+\delta} \overline{h}_rd\widehat{N}_r, \quad s\in
[t,t+\delta].
\end{eqnarray*}
By applying   It\^o's formula  to $\overline{y}_s
e^{\int_{t}^{s}a_{r}dr}$  we obtain
\begin{eqnarray*}
\overline{y}_t=\dfrac{\varepsilon}{2}\mathbb{E}^{\mathbb{Q}}[1_{A}e^{\int_{t}^{t+\delta}a_{r}dr}\
|\mathcal{F}_{t}]=\dfrac{\varepsilon}{2}\mathbb{E}[1_{A}M_{t,t+\delta}e^{\int_{t}^{t+\delta}a_{r}dr}\
|\mathcal{F}_{t}].
\end{eqnarray*}
Thanks to the H\"{o}lder inequality we have
\begin{eqnarray*}
\mathbb{P}(A |\mathcal{F}_{t})^{2}=(\mathbb{E}[1_{A}|\mathcal
  {F}_{t}])^{2}\leq \mathbb{E}[1_{A}M_{t,t+\delta}e^{\int_{t}^{t+\delta}a_{r}dr}\
|\mathcal{F}_{t}]\mathbb{E}[M_{t,t+\delta}^{-1}
e^{-\int_{t}^{t+\delta}a_{r}dr}\ |\mathcal{F}_{t}].
\end{eqnarray*}
On the other hand,
\begin{eqnarray*}
&&\mathbb{E}[M_{t,t+\delta}^{-1} e^{-\int_{t}^{t+\delta}a_{r}dr}\
|\mathcal{F}_{t}]\\
&=&\mathbb{E}[\exp\Big(-\int_{t}^{t+\delta}a_{r}dr
 -\int_{t}^{t+\delta}b_{r}dB_{r}+\frac{1}{2}\int_{t}^{t+\delta}|b_{r}|^{2}dr+\lambda\int_{t}^{t+\delta}\widehat{c}_{r}dr\Big)\
  \\ && \qquad \qquad \qquad \qquad \qquad \qquad\qquad \qquad
  \times
  \prod\limits_{t<r\leq t+\delta}(1+\widehat{c}_{r}\Delta N_{r})^{-1}|\ \mathcal{F}_{t}]\\
  &=&\mathbb{E}[\exp\Big(-\int_{t}^{t+\delta}a_{r}dr
 -\int_{t}^{t+\delta}b_{r}dB_{r}+\frac{1}{2}\int_{t}^{t+\delta}|b_{r}|^{2}dr+\lambda\int_{t}^{t+\delta}\widehat{c}_{r}dr\Big)\
 \\ && \qquad \qquad \qquad \qquad \qquad \qquad\qquad \qquad
 \times\exp\Big(-\int_{t}^{t+\delta}\ln(1+\widehat{c}_{r})dN_{r}\Big)|\ \mathcal{F}_{t}]\\
 &\leq& \exp((1+\lambda)CT)\mathbb{E}[\exp\Big(
 -\int_{t}^{t+\delta}b_{r}dB_{r}-\frac{1}{2}\int_{t}^{t+\delta}|b_{r}|^{2}dr\Big)
 \exp\Big(-\int_{t}^{t+\delta}\ln(\varepsilon_{0})dN_{r}\Big)\
 |\ \mathcal{F}_{t}]\\ &\leq & C_{0}^{-1},
\end{eqnarray*}
for some suitably chosen constant $ C_{0}>0$. Consequently, putting
$$\Delta=\Big\{\mathbb{P}\Big(\
^{1}Y^{t,x;u^{\varepsilon},v^{\varepsilon}}_{t+\delta}<
W_{N^{t,1}_{t+\delta}}(t+\delta,X^{t,x;u^{\varepsilon},v^{\varepsilon}}_{t+\delta})-\varepsilon\
|\ \mathcal {F}_{t}\Big)> \varepsilon\Big\}\Big(=\{\mathbb{P}(A
|\mathcal{F}_{t})>\varepsilon\}\Big),$$   we have
\begin{eqnarray*}\label{}
\overline{y}_t&=&\dfrac{\varepsilon}{2}\mathbb{E}[1_{A}M_{t,t+\delta}e^{\int_{t}^{t+\delta}a_{r}dr}\
|\mathcal{F}_{t}] \geq\dfrac{\varepsilon
C_{0}}{2}(\mathbb{E}[1_{A}|\mathcal
  {F}_{t}])^{2}\\&=&\dfrac{\varepsilon C_{0}}{2}(\mathbb{P}(A|\mathcal
  {F}_{t}))^{2}>\dfrac{\varepsilon^{3}}{2}C_{0}1_{\Delta}.
\end{eqnarray*}
Thus, since $$ \overline{y}_{t}=y_{t}-\
^{1}Y_{t}^{t,x;u^{\varepsilon},v^{\varepsilon}}= \
^{1}G^{t,x;u^{\varepsilon},v^{\varepsilon}}_{t,t+\delta}
 [^{1}Y^{t,x;u^{\varepsilon},v^{\varepsilon}}_{t+\delta}+\frac{\varepsilon}{2}1_{A}]
-\ ^{1}G^{t,x;u^{\varepsilon},v^{\varepsilon}}_{t,t+\delta}
 [^{1}Y^{t,x;u^{\varepsilon},v^{\varepsilon}}_{t+\delta}],$$
we obtain
\begin{eqnarray*}\label{}
^{1}G^{t,x;u^{\varepsilon},v^{\varepsilon}}_{t,t+\delta}
 [^{1}Y^{t,x;u^{\varepsilon},v^{\varepsilon}}_{t+\delta}+\frac{\varepsilon}{2}1_{A}]
> \ ^{1}G^{t,x;u^{\varepsilon},v^{\varepsilon}}_{t,t+\delta}
 [^{1}Y^{t,x;u^{\varepsilon},v^{\varepsilon}}_{t+\delta}]+\dfrac{\varepsilon^{3}}{2}C_{0}1_{\Delta}.
\end{eqnarray*}
Therefore,  (\ref{Jeqn6}) yields $$
J_{1}(t,x;\widetilde{u},\beta_{\varepsilon}(\widetilde{u}))
  >J_{1}(t,x;\alpha_{\varepsilon},\beta_{\varepsilon})+\dfrac{\varepsilon^{3}}{2}C_{0}1_{\Delta}.$$
Let us choose $\varepsilon'$ sufficiently small such that
$\dfrac{\varepsilon'^{3}}{2}C_{0}>\varepsilon'^{4}$. Then we have
$\dfrac{\varepsilon^{3}}{2}C_{0}>\varepsilon^{4}$.  Since
$\mathbb{P}(\Delta)>0$, we have  a contradiction with (\ref{Je26})
for $\alpha(\cdot)=\widetilde{u}$. The proof is complete.
\end{proof}\vskip2mm

\noindent {\bf Proof of Theorem \ref{Jt6}: \underline {Sufficiency}}
of (\ref{Jeq6}) and (\ref{Jeq7}).\vskip 2mm

\begin{proof}
We fix arbitrarily  $\varepsilon>0$. For $\varepsilon_{0}>0$ being
specified later let us assume  that
$(u^{\varepsilon_{0}},v^{\varepsilon_{0}})\in\mathcal
{U}_{t,T}\times \mathcal {V}_{t,T}$ satisfies (\ref{Jeq6}) and
(\ref{Jeq7}), i.e.,  for all $s\in[t,T]$ and $j=1,2,$
\begin{eqnarray}\label{Je38}
\mathbb{P}\Big(\
^{j}Y^{t,x;u^{\varepsilon_{0}},v^{\varepsilon_{0}}}_{s}\geq
W_{n_{j}(N^{t,j}_{s})}(s,X^{t,x;u^{\varepsilon_{0}},v^{\varepsilon_{0}}}_{s})-\varepsilon_{0}\
|\ \mathcal {F}_{t}\Big)\geq 1-\varepsilon_{0}, \ \mathbb{P}-a.s.,
\end{eqnarray}
and
\begin{eqnarray}\label{Je30}
|\mathbb{E}[J_{j}(t,x;u^{\varepsilon_{0}},v^{\varepsilon_{0}})]-
e_{j}|\leq\varepsilon_{0},
\end{eqnarray}
where we use Proposition \ref{p10} for getting the first inequality.

 Let us put  $t_{i}=t+i\frac{T-t}{m}, 0\leq i \leq m,$
and  $\delta=\frac{T-t}{m}$, where  $m$ will be specified at the end
of the proof of  Lemma \ref{Jl7}.  Let us apply Lemma \ref{Jle1} to
$u^{\varepsilon_{0}}$ and $\tau=t_{1},\cdots,t_{m}$, successively.
Then, for $\varepsilon_{1}>0$ ($\varepsilon_{1}$ depends on
$\varepsilon$ and is specified later) we have the existence of  NAD
strategies $\alpha_{i}\in \mathcal {A}_{t,T}, i=1,\cdots,m$,  such
that,  for all
 $v\in \mathcal {V}_{t,T}$,
\begin{eqnarray}\label{Jeq8}
\alpha_{i}(v)&=&u^{\varepsilon_{0}}, \text{on}\ [t,t_{i}],\nonumber\\
^{2}Y^{t,x;\alpha_{i}(v),v}_{t_{i}}&\leq&
W_{n_{2}(N^{t,2}_{t_{i}})}(t_{i},X^{t,x;\alpha_{i}(v),v}_{t_{i}})+\varepsilon_{1},
\mathbb{P}-a.s.
\end{eqnarray}
For all $v\in \mathcal {V}_{t,T}$, let us define
\begin{eqnarray*}\label{}
S^{v}&=&\inf\Big\{s\geq t \ |\ \lambda(\{r\in[t,s]: v_{r}\neq
v_{r}^{\varepsilon_{0}}\})>0\Big\},\\
t^{v}&=&\inf\Big\{t_{i}\geq S^{v} \ |\ i=1,\cdots, m\Big\}\wedge T.
\end{eqnarray*}
Here we denote by $\lambda$  the Lebesgue measure on the real line
$\mathbb{R}$. Then, $S^{v}$ and $t^{v}$ are stopping times, and
$S^{v}\leq t^{v}\leq S^{v}+\delta$.  Let us set
\begin{equation*}\label{}
\alpha_{\varepsilon}(v)=\left\{
\begin{array}{ll}
 u^{\varepsilon_{0}}, & \ \text{on}\ [ t, t^{v}],\\
\alpha_{i}(v), & \ \text{on}\ (t_{i},T]\times \{t^{v}=t_{i}\}, 1\leq
i\leq m.
\end{array}
\right.
\end{equation*}
Then  $\alpha_{\varepsilon}$ is an NAD strategy, and by virtue of
(\ref{Jeq8}) we obtain
\begin{eqnarray}\label{Jeq9}
^{2}Y^{t,x;\alpha_{\varepsilon}(v),v}_{t^{v}}&=&
\sum\limits_{i=1}^{m} \
^{2}Y^{t,x;\alpha_{\varepsilon}(v),v}_{t_{i}}1_{\{t^{v}=t_{i}\}}\nonumber\\
&\leq&\sum\limits_{i=1}^{m}W_{n_{2}(N^{t,2}_{t_{i}})}(t_{i},X^{t,x;\alpha_{\varepsilon}(v),v}_{t_{i}})1_{\{t^{v}=t_{i}\}}
+\varepsilon_{1}\nonumber\\
&=&
W_{n_{2}(N^{t,2}_{t_{v}})}(t^{v},X^{t,x;\alpha_{\varepsilon}(v),v}_{t^{v}})+\varepsilon_{1},
\ \mathbb{P}-a.s.
\end{eqnarray}
Let us admit the following lemma for the moment; we shall prove it
after.
\begin{lemma}\label{Jl7}
For all $\varepsilon>0$ and $v\in \mathcal {V}_{t,T}$,
\begin{eqnarray*}\label{}
 J_{2}(t,x;\alpha_{\varepsilon}(v),v)
\leq J_{2}(t,x;u^{\varepsilon_{0}},v^{\varepsilon_{0}})
+\varepsilon, \quad
\alpha_{\varepsilon}(v^{\varepsilon_{0}})=u^{\varepsilon_{0}},
\end{eqnarray*}
for $(u^{\varepsilon_{0}},v^{\varepsilon_{0}})$ from (\ref{Je38}).
\end{lemma}
By a similar argument as that for Lemma \ref{Jl7} we can construct
$\beta_{\varepsilon}\in \mathcal {B}_{t,T}$ such that, for all $u\in
\mathcal {U}_{t,T}$,
\begin{eqnarray*}\label{}
 J_{1}(t,x;u,\beta_{\varepsilon}(u))
\leq J_{1}(t,x;u^{\varepsilon_{0}},v^{\varepsilon_{0}})
+\varepsilon,\quad
\beta_{\varepsilon}(u^{\varepsilon_{0}})=v^{\varepsilon_{0}}.
\end{eqnarray*}
From the latter both inequalities, (\ref{Je30})  and Lemma
\ref{Jle3} it follows that
$(\alpha_{\varepsilon},\beta_{\varepsilon})$ satisfies Definition
\ref{Jd2}. Consequently, $(e_{1},e_{2})$ is a Nash equilibrium
payoff.
\end{proof}

\noindent{\bf Proof of Lemma \ref{Jl7}:}\hskip 3mm From (\ref{Jeq9})
and the Lemmas \ref{l8} and \ref{l1} we get the existence of a
positive constant $C$ such that
\begin{eqnarray}\label{Jeqn2}
 J_{2}(t,x,\alpha_{\varepsilon}(v),v)&=&
 ^{2}G^{t,x;\alpha_{\varepsilon}(v),v}_{t,t^{v}}[^{2}Y^{t,x,\alpha_{\varepsilon}(v),v}_{t^{v}}]
\leq\
^{2}G^{t,x;\alpha_{\varepsilon}(v),v}_{t,t^{v}}[W_{n_{2}(N^{t,2}_{t^{v}})}(t^{v},X^{t,x;\alpha_{\varepsilon}(v),v}_{t^{v}})
+\varepsilon_{1}]\nonumber\\
&\leq &
^{2}G^{t,x;\alpha_{\varepsilon}(v),v}_{t,t^{v}}[W_{n_{2}(N^{t,2}_{t^{v}})}(t^{v},X^{t,x;\alpha_{\varepsilon}(v),v}_{t^{v}})]
+C\varepsilon_{1}.
\end{eqnarray}
From Proposition \ref{p4}, Lemma \ref{Jle2} as well as the
definitions of $t^{v}$ and $\alpha_{\varepsilon}$ it follows that
\begin{eqnarray*}\label{}
\mathbb{E}[|W_{n_{2}(N^{t,2}_{t^{v}})}(t^{v},X_{t^{v}}^{t,x;u^{\varepsilon_{0}},v^{\varepsilon_{0}}})
-W_{n_{2}(N^{t,2}_{t^{v}})}(t^{v},X^{t,x;\alpha_{\varepsilon}(v),v}_{t^{v}})|^{2}
\Big|\mathcal {F}_{t}]\leq  C\delta,\  \mathbb{P}-a.s.
\end{eqnarray*}
Consequently, by virtue of  Lemma \ref{l1} we deduce that
\begin{eqnarray*}\label{}
\begin{aligned}
&|\
^{2}G^{t,x;\alpha_{\varepsilon}(v),v}_{t,t^{v}}[W_{n_{2}(N^{t,2}_{t^{v}})}
(t^{v},X_{t^{v}}^{t,x;u^{\varepsilon_{0}},v^{\varepsilon_{0}}})] -\
^{2}G^{t,x;\alpha_{\varepsilon}(v),v}_{t,t^{v}}[W_{n_{2}(N^{t,2}_{t^{v}})}
(t^{v},X^{t,x;\alpha_{\varepsilon}(v),v}_{t^{v}})]|\\
&\leq
C\mathbb{E}[|W_{n_{2}(N^{t,2}_{t^{v}})}(t^{v},X_{t^{v}}^{t,x;u^{\varepsilon_{0}},v^{\varepsilon_{0}}})
-W_{n_{2}(N^{t,2}_{t^{v}})}(t^{v},X^{t,x;\alpha_{\varepsilon}(v),v}_{t^{v}})|^{2}
\Big|\mathcal {F}_{t}]^{\frac{1}{2}}\leq  C\delta^{\frac{1}{2}},
\end{aligned}
\end{eqnarray*}
from which, combined with (\ref{Jeqn2}), we get
\begin{eqnarray*}\label{}
&&J_{2}(t,x,\alpha_{\varepsilon}(v),v) \\ &\leq &
^{2}G^{t,x;\alpha_{\varepsilon}(v),v}_{t,t^{v}}[W_{n_{2}(N^{t,2}_{t^{v}})}(t^{v},
X_{t^{v}}^{t,x;u^{\varepsilon_{0}},v^{\varepsilon_{0}}})]+C\varepsilon_{1}\\
&&+|\
^{2}G^{t,x;\alpha_{\varepsilon}(v),v}_{t,t^{v}}[W_{n_{2}(N^{t,2}_{t^{v}})}(t^{v},X_{t^{v}}^{t,x;u^{\varepsilon_{0}},v^{\varepsilon_{0}}})]
- \ ^{2}G^{t,x;\alpha_{\varepsilon}(v),v}_{t,t^{v}}[W_{n_{2}(N^{t,2}_{t^{v}})}(t^{v},X^{t,x;\alpha_{\varepsilon}(v),v}_{t^{v}})]|\\
&\leq &
^{2}G^{t,x;\alpha_{\varepsilon}(v),v}_{t,t^{v}}[W_{n_{2}(N^{t,2}_{t^{v}})}(t^{v},X_{t^{v}}^{t,x;u^{\varepsilon_{0}},v^{\varepsilon_{0}}})]
+C\varepsilon_{1}+C\delta^{\frac{1}{2}}.
\end{eqnarray*}
For $s\in[t,T]$, we put
\begin{eqnarray*}\label{}
\Omega_{s}=\Big\{\
^{2}Y^{t,x;u^{\varepsilon_{0}},v^{\varepsilon_{0}}}_{s}\geq
W_{n_{2}(N^{t,2}_{s})}(s,X^{t,x;u^{\varepsilon_{0}},v^{\varepsilon_{0}}}_{s})-\varepsilon_{0}\Big\}.
\end{eqnarray*}
Then we have
\begin{eqnarray}\label{Jeqn4}
&& J_{2}(t,x;\alpha_{\varepsilon}(v),v)
 \leq \ ^{2}G^{t,x;\alpha_{\varepsilon}(v),v}_{t,t^{v}}
 [W_{n_{2}(N^{t,2}_{t^{v}})}(t^{v},X_{t^{v}}^{t,x;u^{\varepsilon_{0}},v^{\varepsilon_{0}}})]
+C\varepsilon_{1}+C\delta^{\frac{1}{2}}\nonumber\\
&&\leq\ ^{2}G^{t,x;\alpha_{\varepsilon}(v),v}_{t,t^{v}}
[\sum\limits_{i=1}^{m}W_{n_{2}(N^{t,2}_{t_{i}})}(t_{i},X_{t_{i}}^{t,x;u^{\varepsilon_{0}},v^{\varepsilon_{0}}})
1_{\{t^{v}=t_{i}\}}1_{\Omega_{t_{i}}}]\nonumber\\&&\hskip75mm+C\varepsilon_{1}+C\delta^{\frac{1}{2}}
+I,
\end{eqnarray}
where
\begin{eqnarray*}
&&I=|\ ^{2}G^{t,x;\alpha_{\varepsilon}(v),v}_{t,t^{v}}
[\sum\limits_{i=1}^{m}W_{n_{2}(N^{t,2}_{t_{i}})}
(t_{i},X_{t_{i}}^{t,x;u^{\varepsilon_{0}},v^{\varepsilon_{0}}})1_{\{t^{v}=t_{i}\}}]
\\ && \qquad  \qquad \qquad
-\ ^{2}G^{t,x;\alpha_{\varepsilon}(v),v}_{t,t^{v}}
[\sum\limits_{i=1}^{m}W_{n_{2}(N^{t,2}_{t_{i}})}(t_{i},X_{t_{i}}^{t,x;u^{\varepsilon_{0}},v^{\varepsilon_{0}}})
1_{\{t^{v}=t_{i}\}}1_{\Omega_{t_{i}}}]|.
\end{eqnarray*}
 Noting Lemma \ref{l1},
(\ref{Je38}) as well as the boundedness of
$W_{n_{2}(N^{t,2}_{t_{i}})}$  we conclude
\begin{eqnarray}\label{Jeqn5}
I&\leq&\mathbb{E}[\sum\limits_{i=1}^{m}|W_{n_{2}(N^{t,2}_{t_{i}})}(t_{i},
X_{t_{i}}^{t,x;u^{\varepsilon_{0}},v^{\varepsilon_{0}}})|^{2}
1_{\{t^{v}=t_{i}\}}1_{\Omega^{c}_{t_{i}}} \Big|\mathcal
{F}_{t}]^{\frac{1}{2}} \nonumber \\ &\leq & C\sum\limits_{i=1}^{m}
\mathbb{P}(\Omega^{c}_{t_{i}} |\mathcal{F}_{t})^{\frac{1}{2}}\leq
Cm\varepsilon_{0}^{\frac{1}{2}}.
\end{eqnarray}
Since $\
^{2}Y^{t,x;u^{\varepsilon_{0}},v^{\varepsilon_{0}}}_{t_{i}}\geq
W_{n_{2}(N^{t,2}_{t_{i}})}(t_{i},X^{t,x;u^{\varepsilon_{0}},v^{\varepsilon_{0}}}_{t_{i}})-\varepsilon_{0}
\ \text{on}\ \Omega_{t_{i}},$  it follows from Lemma \ref{l1} that
\begin{eqnarray*}\label{}
&&^{2}G^{t,x;\alpha_{\varepsilon}(v),v}_{t,t^{v}}
[\sum\limits_{i=1}^{m}W_{n_{2}(N^{t,2}_{t_{i}})}(t_{i},X_{t_{i}}^{t,x;u^{\varepsilon_{0}},v^{\varepsilon_{0}}})
1_{\{t^{v}=t_{i}\}}1_{\Omega_{t_{i}}}] \nonumber \\ &&\leq\
^{2}G^{t,x;\alpha_{\varepsilon}(v),v}_{t,t^{v}}
[\sum\limits_{i=1}^{m}\
^{2}Y^{t,x;u^{\varepsilon_{0}},v^{\varepsilon_{0}}}_{t_{i}}
1_{\{t^{v}=t_{i}\}}1_{\Omega_{t_{i}}}+\varepsilon_{0}]\\
&&\leq\ ^{2}G^{t,x;\alpha_{\varepsilon}(v),v}_{t,t^{v}}
[\sum\limits_{i=1}^{m}\
^{2}Y^{t,x;u^{\varepsilon_{0}},v^{\varepsilon_{0}}}_{t_{i}}
1_{\{t^{v}=t_{i}\}}1_{\Omega_{t_{i}}}]+C\varepsilon_{0}.
\end{eqnarray*}
Similarly  to (\ref{Jeqn5}) we have
\begin{eqnarray*}\label{}
|\ ^{2}G^{t,x;\alpha_{\varepsilon}(v),v}_{t,t^{v}}
[\sum\limits_{i=1}^{m}\
^{2}Y^{t,x;u^{\varepsilon_{0}},v^{\varepsilon_{0}}}_{t_{i}}
1_{\{t^{v}=t_{i}\}}1_{\Omega_{t_{i}}}] -\
^{2}G^{t,x;\alpha_{\varepsilon}(v),v}_{t,t^{v}}
[\sum\limits_{i=1}^{m} \
^{2}Y^{t,x;u^{\varepsilon_{0}},v^{\varepsilon_{0}}}_{t_{i}}1_{\{t^{v}=t_{i}\}}]|\leq
Cm\varepsilon_{0}^{\frac{1}{2}},
\end{eqnarray*}
from which together with
$^{2}Y^{t,x;u^{\varepsilon_{0}},v^{\varepsilon_{0}}}_{t^{v}}=
\sum\limits_{i=1}^{m} \
^{2}Y^{t,x;u^{\varepsilon_{0}},v^{\varepsilon_{0}}}_{t_{i}}1_{\{t^{v}=t_{i}\}}$
it follows that
\begin{eqnarray*}\label{}
&&^{2}G^{t,x;\alpha_{\varepsilon}(v),v}_{t,t^{v}}
[\sum\limits_{i=1}^{m}W_{n_{2}(N^{t,2}_{t_{i}})}(t_{i},X_{t_{i}}^{t,x;u^{\varepsilon_{0}},v^{\varepsilon_{0}}})
1_{\{t^{v}=t_{i}\}}1_{\Omega_{t_{i}}}] \\&\leq& \
^{2}G^{t,x;\alpha_{\varepsilon}(v),v}_{t,t^{v}} [
^{2}Y^{t,x;u^{\varepsilon_{0}},v^{\varepsilon_{0}}}_{t^{v}}]+C\varepsilon_{0}+ Cm\varepsilon_{0}^{\frac{1}{2}}\\
&\leq&|\ ^{2}G^{t,x;\alpha_{\varepsilon}(v),v}_{t,t^{v}} [ \
^{2}Y^{t,x;u^{\varepsilon_{0}},v^{\varepsilon_{0}}}_{t^{v}}] -\
^{2}G^{t,x;u^{\varepsilon_{0}},v^{\varepsilon_{0}}}_{t,t^{v}} [ \
^{2}Y^{t,x;u^{\varepsilon_{0}},v^{\varepsilon_{0}}}_{t^{v}}]|\\&& +
\ ^{2}G^{t,x;u^{\varepsilon_{0}},v^{\varepsilon_{0}}}_{t,t^{v}} [ \
^{2}Y^{t,x;u^{\varepsilon_{0}},v^{\varepsilon_{0}}}_{t^{v}}]
+C\varepsilon_{0}+ Cm\varepsilon_{0}^{\frac{1}{2}}\\
&=&|\ ^{2}G^{t,x;\alpha_{\varepsilon}(v),v}_{t,t^{v}} [ \
^{2}Y^{t,x;u^{\varepsilon_{0}},v^{\varepsilon_{0}}}_{t^{v}}] -\
^{2}G^{t,x;u^{\varepsilon_{0}},v^{\varepsilon_{0}}}_{t,t^{v}} [ \
^{2}Y^{t,x;u^{\varepsilon_{0}},v^{\varepsilon_{0}}}_{t^{v}}]|\\
&&+J_{2}(t,x;u^{\varepsilon_{0}},v^{\varepsilon_{0}})
+C\varepsilon_{0}+ Cm\varepsilon_{0}^{\frac{1}{2}}.
\end{eqnarray*}
Using  arguments similar to those  in \cite{L2011} we can show that
\begin{eqnarray*}\label{}
|\ ^{2}G^{t,x;\alpha_{\varepsilon}(v),v}_{t,t^{v}} [ \
^{2}Y^{t,x;u^{\varepsilon_{0}},v^{\varepsilon_{0}}}_{t^{v}}] -\
^{2}G^{t,x;u^{\varepsilon_{0}},v^{\varepsilon_{0}}}_{t,t^{v}} [ \
^{2}Y^{t,x;u^{\varepsilon_{0}},v^{\varepsilon_{0}}}_{t^{v}}]|\leq
C\delta^{\frac{1}{2}}.
\end{eqnarray*}
Consequently,
\begin{eqnarray*}\label{}
&&^{2}G^{t,x;\alpha_{\varepsilon}(v),v}_{t,t^{v}}
[\sum\limits_{i=1}^{m}W_{n_{2}(N^{t,2}_{t_{i}})}(t_{i},X_{t_{i}}^{t,x;u^{\varepsilon_{0}},v^{\varepsilon_{0}}})
1_{\{t^{v}=t_{i}\}}1_{\Omega_{t_{i}}}]\\ && \leq
C\delta^{\frac{1}{2}}
+J_{2}(t,x;u^{\varepsilon_{0}},v^{\varepsilon_{0}})
+C\varepsilon_{0}+Cm\varepsilon_{0}^{\frac{1}{2}}.
\end{eqnarray*}
Thus, from  (\ref{Jeqn4}) and (\ref{Jeqn5}) we have
\begin{eqnarray*}\label{}
 J_{2}(t,x;\alpha_{\varepsilon}(v),v)
\leq J_{2}(t,x;u^{\varepsilon_{0}},v^{\varepsilon_{0}})
+C\varepsilon_{0}+Cm\varepsilon_{0}^{\frac{1}{2}}
+C\varepsilon_{1}+C\delta^{\frac{1}{2}}.
\end{eqnarray*}
We can choose $\delta>0, \varepsilon_{0}>0,$ and $
\varepsilon_{1}>0$ such that
$C\varepsilon_{0}+Cm\varepsilon_{0}^{\frac{1}{2}}
+C\varepsilon_{1}+C\delta^{\frac{1}{2}}\leq \varepsilon$ and
$\varepsilon_{0}<\varepsilon.$ Thus,
\begin{eqnarray*}\label{}
 J_{2}(t,x;\alpha_{\varepsilon}(v),v)
\leq J_{2}(t,x;u^{\varepsilon_{0}},v^{\varepsilon_{0}})
+\varepsilon,\ v\in \mathcal {V}_{t,T}.
\end{eqnarray*}
This allows us to complete the proof. $\Box$\vskip2mm

One of our main results of this section is the following  existence
theorem of a Nash equilibrium payoff.
\begin{theorem}\label{Jt2}
Under the Isaacs condition,  there exists a Nash equilibrium payoff
at $(t,x)$, for all $(t,x)\in[0,T]\times\mathbb{R}^{n}$.
\end{theorem}
From Theorem \ref{Jt6} we only have to prove that, for all
$\varepsilon>0,$ there exists
$(u^{\varepsilon},v^{\varepsilon})\in\mathcal{U}_{t,T}\times\mathcal{V}_{t,T}$
such that (\ref{Jeq6}) and (\ref{Jeq7}) hold, for $\delta\in[0,T-t],
j=1,2$. The following proposition is crucial for  this proof  and it
will be proven after.
\begin{proposition}\label{Jp11}
For all $\varepsilon>0,$ there exists
$(u^{\varepsilon},v^{\varepsilon})\in\mathcal{U}_{t,T}\times\mathcal{V}_{t,T}$
independent of $\mathcal {F}_{t}$ such that, for all stopping time
$\tau$ ($t\leq \tau\leq T$) independent of $\mathcal {F}_{t}$,
$j=1,2$, (\ref{Jeq6}) holds:
\begin{eqnarray*}
 \mathbb{P}\Big(\
^{N^{t,j}_{\tau}}\widetilde{Y}^{t,x;u^{\varepsilon},v^{\varepsilon}}_{\tau}\geq
W_{n_{j}(N_{\tau}^{t,j})}(\tau,X^{t,x;u^{\varepsilon},v^{\varepsilon}}_{\tau})-\varepsilon\
|\ \mathcal {F}_{t}\Big)\geq 1-\varepsilon,\ \mathbb{P}-a.s.
\end{eqnarray*}
\end{proposition}

From the above proposition we immediately have the following
corollary.
\begin{corollary}\label{Jc11}
For all $\varepsilon>0,$ there exists
$(u^{\varepsilon},v^{\varepsilon})\in\mathcal{U}_{t,T}\times\mathcal{V}_{t,T}$
independent of $\mathcal {F}_{t}$ such that, for all   $s\in [t,T],$
$j=1,2$, (\ref{Jeq6}) holds:
\begin{eqnarray*}
 \mathbb{P}\Big(\
^{N^{t,j}_{s}}\widetilde{Y}^{t,x;u^{\varepsilon},v^{\varepsilon}}_{s}\geq
W_{n_{j}(N_{s}^{t,j})}(s,X^{t,x;u^{\varepsilon},v^{\varepsilon}}_{s})-\varepsilon\
|\ \mathcal {F}_{t}\Big)\geq 1-\varepsilon,\ \mathbb{P}-a.s.
\end{eqnarray*}
\end{corollary}
\noindent We now give the proof of Theorem \ref{Jt2}.
\begin{proof}\label{}
For  $\varepsilon>0$, let
$(u^{\varepsilon},v^{\varepsilon})\in\mathcal{U}_{t,T}\times\mathcal{V}_{t,T}$
be that of Corollary \ref{Jc11}. Then (\ref{Jeq6}) holds. By
noticing that $(u^{\varepsilon},v^{\varepsilon})$ is independent of
$\mathcal {F}_{t}$ we see that
$J_{j}(t,x;u^{\varepsilon},v^{\varepsilon}), j=1,2,$ are
deterministic and $\Big\{
(J_{1}(t,x;u^{\varepsilon},v^{\varepsilon}),$
$J_{2}(t,x;u^{\varepsilon},v^{\varepsilon})), \varepsilon>0\Big\}$
is a bounded sequence. Therefore, we can choose an accumulation
point of this sequence, as $\varepsilon\rightarrow 0$, and we denote
this point by $(e_{1},e_{2})$. Consequently,  from Theorem \ref{Jt6}
it follows that $(e_{1},e_{2})$ is a Nash equilibrium payoff at
$(t,x)$.
 The proof is complete.
\end{proof}

Before  proving   Proposition \ref{Jp11}, let us first make some
preliminaries for its proof.

\begin{lemma}
For all $\varepsilon>0,$  $\delta\in[0,T-t]$ and
$x\in\mathbb{R}^{n}$, we have the existence of
$(u^{\varepsilon},v^{\varepsilon})\in\mathcal{U}_{t,T}\times\mathcal{V}_{t,T}$
independent of $\mathcal {F}_{t}$, such that, $j=1,2,$
\begin{eqnarray*}\label{}
 W_{j}(t,x)-\varepsilon\leq
 \ ^{j}G^{t,x;u^{\varepsilon},v^{\varepsilon}}_{t,t+\delta}
 [W_{n_{j}(N_{t+\delta}^{t,j})}(t+\delta,X^{t,x;u^{\varepsilon},v^{\varepsilon}}_{t+\delta})], \ \mathbb{P}- a.s.
\end{eqnarray*}
\end{lemma}
For its proof, we  adapt the ideas developed in \cite{L2011} from
SDGs without jumps to SDGs with jumps.
\begin{proof}
We  denote by $\mathbb{F}^{t}=(\mathcal {F}^{t}_{s})_{s\in [t,T]}$
the following  filtration:
\begin{eqnarray*}\label{}
\mathcal {F}^{t}_{s}:=\sigma\Big\{B_{r}-B_{t}, N_{r}-N_{t}:\   t\leq
r \leq s\Big\}\vee \mathcal {N}_{\mathbb{P}},\ s\in [t,T],
\end{eqnarray*}
where $\mathcal {N}_{\mathbb{P}}$ is  the collection of all
$\mathbb{P}$-null sets. For $s\in [t,T]$, we denote by $\mathcal
{U}^{t}_{s,T}$ (resp., $\mathcal {V}^{t}_{s,T}$)  the set of
$\mathbb{F}^{t}$-adapted processes $\{u_{r}\}_{r\in[s,T]}$ (resp.,
$\{v_{r}\}_{r\in[s,T]}$) taking their values in $U$ (resp., $V$),
and we let $\mathcal {A}^{t}_{s,T}$ (resp., $\mathcal
{B}^{t}_{s,T}$) be the NAD strategies  from $\mathcal {V}^{t}_{s,T}$
into $\mathcal {U}^{t}_{s,T}$ (resp., $\mathcal {U}^{t}_{s,T}$ into
$\mathcal {V}^{t}_{s,T}$).

Let us replace the  framework of SDEs driven by a Brownian motion
$B=(B_{s})_{s\in [0,T]}$ by that of SDEs driven by a Brownian motion
$(B_{s}-B_{t})_{s\in [t,T]}$, and let us also replace  the framework
of BSDEs driven by a Brownian motion $B=(B_{s})_{s\in [0,T]}$ and
the Poisson process  $(N_{s})_{s\in [t,T]}$ by that of
$B^{t}=(B_{s}-B_{t})_{s\in [t,T]}$ and $(N_{s}-N_{t})_{s\in [t,T]}$.
Using the arguments of the Sections \ref{NS2} and \ref{NS3}  and
Isaacs conditions,  we conclude, for
$(s,x)\in[t,T]\times\mathbb{R}^{n}$,
\begin{eqnarray*}\label{}
\widetilde{W}_{1}(s,x)&=& \esssup_{\alpha\in\mathcal {A}^{t}_{s,T}}
\essinf_{\beta\in\mathcal {B}^{t}_{s,T}}
J_{1}(s,x;\alpha,\beta)=\essinf_{\beta\in\mathcal
{B}^{t}_{s,T}}\esssup_{\alpha\in\mathcal {A}^{t}_{s,T}}
J_{1}(s,x;\alpha,\beta),\\
\widetilde{W}_{2}(s,x)&=&\essinf_{\alpha\in\mathcal
{A}^{t}_{s,T}}\esssup_{\beta\in\mathcal {B}^{t}_{s,T}}
J_{2}(s,x;\alpha,\beta)=\esssup_{\beta\in\mathcal
{B}^{t}_{s,T}}\essinf_{\alpha\in\mathcal {A}^{t}_{s,T}}
J_{2}(s,x;\alpha,\beta).
\end{eqnarray*}

For $j=1,2$, it follows from the Sections \ref{NS2} and \ref{NS3}
that $W_{j}$ restricted to $[t,T]\times\mathbb{R}^{n}$ and
$\widetilde{W}_{j}$ are inside the class of continuous functions
with linear growth and the unique viscosity solutions of the same
system of  Isaacs equations. Therefore,
\begin{eqnarray*}\label{}
\widetilde{W}_{j}(s,x)=W_{j}(s,x), \
(s,x)\in[t,T]\times\mathbb{R}^{n}, j=1,2.
\end{eqnarray*}
By virtue of the dynamic programming principle for
$\widetilde{W}_{j}$ and  $\mathcal {V}^{t}_{t,t+\delta}\subset
\mathcal {B}^{t}_{t,t+\delta}$ we deduce that
\begin{eqnarray*}\label{}
 W_{1}(t,x)&=&\widetilde{W}_{1}(t,x)=\esssup_{\alpha\in\mathcal {A}^{t}_{t,t+\delta}}
\essinf_{\beta\in\mathcal {B}^{t}_{t,t+\delta}} \
^{1}G^{t,x;\alpha,\beta}_{t,t+\delta}
 [W_{n_{1}(N_{t+\delta}^{t,1})}(t+\delta,X^{t,x;\alpha,\beta}_{t+\delta})]\\
 &\leq&\esssup_{\alpha\in\mathcal {A}^{t}_{t,t+\delta}}
\essinf_{v\in\mathcal {V}^{t}_{t,t+\delta}} \
^{1}G^{t,x;\alpha(v),v}_{t,t+\delta}
 [W_{n_{1}(N_{t+\delta}^{t,1})}(t+\delta,X^{t,x;\alpha(v),v}_{t+\delta})].
\end{eqnarray*}
Therefore,  for $\varepsilon>0$ and $\delta>0$, we have  the
existence of $\alpha_{\varepsilon}\in\mathcal{A}^{t}_{t,t+\delta}$
such that, for all $v\in\mathcal{V}^{t}_{t,t+\delta}$,
\begin{eqnarray*}\label{}
 W_{1}(t,x)-\varepsilon\leq
\ ^{1}G^{t,x;\alpha_{\varepsilon}(v),v}_{t,t+\delta}
 [W_{n_{1}(N_{t+\delta}^{t,1})}(t+\delta,X^{t,x;\alpha_{\varepsilon}(v),v}_{t+\delta})],\
 \mathbb{P}-a.s.
\end{eqnarray*}
By a symmetric argument  there exists
$\beta_{\varepsilon}\in\mathcal{B}^{t}_{t,t+\delta}$ such that, for
all $u\in\mathcal{U}^{t}_{t,t+\delta}$,
\begin{eqnarray*}\label{}
 W_{2}(t,x)-\varepsilon\leq
\ ^{2}G^{t,x;u,\beta_{\varepsilon}(u)}_{t,t+\delta}
 [W_{n_{2}(N_{t+\delta}^{t,2})}(t+\delta,X^{t,x;u,\beta_{\varepsilon}(u)}_{t+\delta})],  \ \mathbb{P}-a.s.
\end{eqnarray*}
In analogy to  Lemma \ref{l5}, it can be shown that  there exists a
unique couple
$(u^{\varepsilon},v^{\varepsilon})\in\mathcal{U}^{t}_{t,t+\delta}\times\mathcal{V}^{t}_{t,t+\delta}$
such that $\alpha_{\varepsilon}(v^{\varepsilon})=u^{\varepsilon},
\beta_{\varepsilon}(u^{\varepsilon})=v^{\varepsilon}.$ Consequently,
\begin{eqnarray*}\label{}
 W_{j}(t,x)-\varepsilon\leq
\ ^{j}G^{t,x;u^{\varepsilon},v^{\varepsilon}}_{t,t+\delta}
 [W_{n_{j}(N_{t+\delta}^{t,j})}(t+\delta,X^{t,x;u^{\varepsilon},v^{\varepsilon}}_{t+\delta})],\
 j=1,2.
\end{eqnarray*}
This completes the proof.
\end{proof}
Also  the following Lemma is crucial for the proof of  Proposition
\ref{Jp11}.
\begin{lemma}\label{Jle6}
For $n\geq 1$, we fix some partition $t=t_{0}<t_{1}<\cdots<t_{n}=T$
of the interval $[t,T]$. Then, for all $\varepsilon>0,$  there
exists
$(u^{\varepsilon},v^{\varepsilon})\in\mathcal{U}_{t,T}\times\mathcal{V}_{t,T}$
independent of $\mathcal {F}_{t}$, such that, for all
$i=0,\cdots,n-1$ and $j=1,2,$
\begin{eqnarray*}\label{}
 W_{n_{j}(N_{t_{i}}^{t,j})}(t_{i},X^{t,x;u^{\varepsilon},v^{\varepsilon}}_{t_{i}})-\varepsilon\leq
\ ^{j}G^{t,x;u^{\varepsilon},v^{\varepsilon}}_{t_{i},t_{i+1}}
 [W_{n_{j}(N^{t,j}_{t_{i+1}})}(t_{i+1},X^{t,x;u^{\varepsilon},v^{\varepsilon}}_{t_{i+1}})],\
 \mathbb{P}- a.s.
\end{eqnarray*}
\end{lemma}
\begin{proof}
Let us prove it by induction. Obviously, due to the above lemma, it
holds for $i=0$. Let a couple $(u^{\varepsilon},v^{\varepsilon})$
independent of $\mathcal {F}_{t}$, be constructed on the interval
$[t,t_{i})$, and we shall give its definition  on $[t_{i},t_{i+1})$.
By virtue of the above lemma we have, for all  $y\in
\mathbb{R}^{n}$, the existence of
$(u^{y},v^{y})\in\mathcal{U}_{t_{i},T}\times\mathcal{V}_{t_{i},T}$
independent of $\mathcal {F}_{t}$, such that,
\begin{eqnarray}\label{eq10}
 W_{j}(t_{i},y)-\frac{\varepsilon}{2}\leq
\ ^{j}G^{t_{i},y;u^{y},v^{y}}_{t_{i},t_{i+1}}
 [W_{n_{j}(N^{t_{i},j}_{t_{i+1}})}(t_{i+1},X^{t_{i},y;u^{y},v^{y}}_{t_{i+1}})], \ \mathbb{P}- a.s,
 j=1,2.
\end{eqnarray}
For  $j=1,2,$   $y, z\in \mathbb{R}^{n}$ and $s\in[t_{i},t_{i+1}]$,
let us set
$$y^{1}_{s}=\ ^{j}G^{t_{i},y;u^{y},v^{y}}_{s,t_{i+1}}
 [W_{n_{j}(N^{t_{i},j}_{t_{i+1}})}(t_{i+1},X^{t_{i},y;u^{y},v^{y}}_{t_{i+1}})],$$ and
 $$y^{2}_{s}=\ ^{j}G^{t_{i},z;u^{y},v^{y}}_{s,t_{i+1}}
 [W_{n_{j}(N^{t_{i},j}_{t_{i+1}})}(t_{i+1},X^{t_{i},z;u^{y},v^{y}}_{t_{i+1}})].$$
Then let us consider the following associated BSDEs:
\begin{eqnarray*}\label{}
y^{1}_{s}&=&
W_{n_{j}(N^{t_{i},j}_{t_{i+1}})}(t_{i+1},X^{t_{i},y;u^{y},v^{y}}_{t_{i+1}})
+\displaystyle \int_{s}^{t_{i+1}}
f_{N^{t_{i},j}_{r}}(r,X^{t_{i},y;u^{y},v^{y}}_r,
y_r^{1},h_r^{1},z^{1}_r,u^{y}_r, v^{y}_r)dr\\&&\qquad-\displaystyle
\lambda\int_{s}^{t_{i+1}} h^{1}_rdr-\displaystyle \int_{s}^{t_{i+1}}
z^{1}_rdB_r-\displaystyle \int_{s}^{t_{i+1}}
h^{1}_rd\widetilde{N}_r, \quad s\in[t_{i},t_{i+1}],
\end{eqnarray*}
as well as
\begin{eqnarray*}\label{}
y^{2}_{s}&=&
W_{n_{j}(N^{t_{i},j}_{t_{i+1}})}(t_{i+1},X^{t_{i},z;u^{y},v^{y}}_{t_{i+1}})
+\displaystyle \int_{s}^{t_{i+1}}
f_{N^{t_{i},j}_{r}}(r,X^{t_{i},z;u^{y},v^{y}}_r,
y_r^{2},h_r^{2},z^{2}_r,u^{y}_r, v^{y}_r)dr\\&&\qquad-\displaystyle
\lambda\int_{s}^{t_{i+1}} h^{2}_rdr-\displaystyle \int_{s}^{t_{i+1}}
z^{2}_rdB_r-\displaystyle \int_{s}^{t_{i+1}}
h^{2}_rd\widetilde{N}_r, \quad s\in[t_{i},t_{i+1}].
\end{eqnarray*}
From Lemma  \ref{l1} it follows that
\begin{eqnarray*}\label{}
&&|^{j}G^{t_{i},y;u^{y},v^{y}}_{t_{i},t_{i+1}}
 [W_{n_{j}(N^{t_{i},j}_{t_{i+1}})}(t_{i+1},X^{t_{i},y;u^{y},v^{y}}_{t_{i+1}})]
-\ ^{j}G^{t_{i},z;u^{y},v^{y}}_{t_{i},t_{i+1}}
 [W_{n_{j}(N^{t_{i},j}_{t_{i+1}})}(t_{i+1},X^{t_{i},z;u^{y},v^{y}}_{t_{i+1}})]|^{2}\\
&\leq&C\mathbb{E}[|W_{n_{j}(N^{t_{i},j}_{t_{i+1}})}(t_{i+1},X^{t_{i},y;u^{y},v^{y}}_{t_{i+1}})
-W_{n_{j}(N^{t_{i},j}_{t_{i+1}})}(t_{i+1},X^{t_{i},z;u^{y},v^{y}}_{t_{i+1}})|^{2}
\Big|\mathcal {F}_{t_{i}}]\\
&&+C\mathbb{E}[\displaystyle \int_{t_{i}}^{t_{i+1}} |
f_{N^{t_{i},j}_{r}}(r,X^{t_{i},y;u^{y},v^{y}}_r,
y_r^{1},h_r^{1},z^{1}_r,u^{y}_r, v^{y}_r) \\ && \qquad \qquad \qquad
- f_{N^{t_{i},j}_{r}}(r,X^{t_{i},z;u^{y},v^{y}}_r,
y_r^{1},h_r^{1},z^{1}_r,u^{y}_r, v^{y}_r)|^{2}dr
\Big|\mathcal {F}_{t_{i}}]\\
&\leq&C\mathbb{E}[|X^{t_{i},y;u^{y},v^{y}}_{t_{i+1}}
-X^{t_{i},z;u^{y},v^{y}}_{t_{i+1}}|^{2} \Big|\mathcal
{F}_{t_{i}}]+C\mathbb{E}[\displaystyle \int_{t_{i}}^{t_{i+1}}|
X^{t_{i},y;u^{y},v^{y}}_r-X^{t_{i},z;u^{y},v^{y}}_r |^{2}dr
\Big|\mathcal {F}_{t_{i}}]\\
 &\leq & C|y-z|^{2}.
\end{eqnarray*}
Thus, due to  Proposition \ref{p4} and (\ref{eq10}) we obtain, for
$C|y-z|\leq \dfrac{\varepsilon}{2}$,
\begin{eqnarray*}\label{}
 W_{j}(t_{i},z)-\varepsilon&\leq&
 W_{j}(t_{i},y)-\varepsilon+C|y-z|\\
 &\leq& \ ^{j}G^{t_{i},y;u^{y},v^{y}}_{t_{i},t_{i+1}}
 [W_{n_{j}(N^{t_{i},j}_{t_{i+1}})}(t_{i+1},X^{t_{i},y;u^{y},v^{y}}_{t_{i+1}})]-\frac{\varepsilon}{2}+C|y-z|\\
&\leq& \ ^{j}G^{t_{i},z;u^{y},v^{y}}_{t_{i},t_{i+1}}
 [W_{n_{j}(N^{t_{i},j}_{t_{i+1}})}(t_{i+1},X^{t_{i},z;u^{y},v^{y}}_{t_{i+1}})]-\frac{\varepsilon}{2}+C|y-z|\\
 &\leq& \ ^{j}G^{t_{i},z;u^{y},v^{y}}_{t_{i},t_{i+1}}
 [W_{n_{j}(N^{t_{i},j}_{t_{i+1}})}(t_{i+1},X^{t_{i},z;u^{y},v^{y}}_{t_{i+1}})], \ \mathbb{P}- a.s.
\end{eqnarray*}
We let $\{O_{i}\}_{i\geq1}\subset \mathcal {B}(\mathbb{R}^{n})$ be a
partition of $\mathbb{R}^{n}$ with $
diam(O_{i})<\dfrac{\varepsilon}{2C}$ and let $y_{l}\in O_{l},
l\geq1$. Then, for $z\in O_{l}$, it follows
\begin{eqnarray}\label{Je27}
 W_{n_{j}(j)}(t_{i},z)-\varepsilon\leq
\ ^{j}G^{t_{i},z;u^{y_{l}},v^{y_{l}}}_{t_{i},t_{i+1}}
 [W_{n_{j}(N^{t_{i},j}_{t_{i+1}})}(t_{i+1},X^{t_{i},z;u^{y_{l}},v^{y_{l}}}_{t_{i+1}})], \ \mathbb{P}- a.s.
\end{eqnarray}
Let us define
\begin{eqnarray*}\label{}
u^{\varepsilon}=\sum\limits_{l\geq1}1_{O_{l}}(X^{t,x;u^{\varepsilon},v^{\varepsilon}}_{t_{i}})u^{y_{l}},\
v^{\varepsilon}=\sum\limits_{l\geq1}1_{O_{l}}(X^{t,x;u^{\varepsilon},v^{\varepsilon}}_{t_{i}})v^{y_{l}}.
\end{eqnarray*}
Then, since $\{X^{t,x;u^{\varepsilon},v^{\varepsilon}}_{t_{i}}\in
O_{l}\}\in \mathcal {F}_{t_{i}}, l\geq 1,$ we conclude
\begin{eqnarray*}\label{}
&& ^{j}G^{t,x;u^{\varepsilon},v^{\varepsilon}}_{t_{i},t_{i+1}}
 [W_{n_{j}(N_{t_{i+1}}^{t,j})}(t_{i+1},X^{t,x;u^{\varepsilon},v^{\varepsilon}}_{t_{i+1}})]\\
  &=& ^{j}G^{t_{i},X^{t,x;u^{\varepsilon},v^{\varepsilon}}_{t_{i}};u^{\varepsilon},v^{\varepsilon}}_{t_{i},t_{i+1}}
 [\sum\limits_{l\geq1}W_{n_{j}(N_{t_{i+1}}^{t_{i},N_{t_{i}}^{t,j}})}(t_{i+1},X^{t_{i},X^{t,x;u^{\varepsilon},v^{\varepsilon}}_{t_{i}};
 u^{\varepsilon},v^{\varepsilon}}_{t_{i+1}})1_{O_{l}}(X^{t,x;u^{\varepsilon},v^{\varepsilon}}_{t_{i}})]\\
   &=& ^{j}G^{t_{i},X^{t,x;u^{\varepsilon},v^{\varepsilon}}_{t_{i}};u^{\varepsilon},v^{\varepsilon}}_{t_{i},t_{i+1}}
 [\sum\limits_{l\geq1}W_{n_{j}(N_{t_{i+1}}^{t_{i},N_{t_{i}}^{t,j}})}(t_{i+1},X^{t_{i},X^{t,x;u^{\varepsilon},v^{\varepsilon}}_{t_{i}};
 u^{y_{l}},v^{y_{l}}}_{t_{i+1}})1_{O_{l}}(X^{t,x;u^{\varepsilon},v^{\varepsilon}}_{t_{i}})]\\
    &=& \sum\limits_{l\geq1}\ ^{j}G^{t_{i},X^{t,x;u^{\varepsilon},v^{\varepsilon}}_{t_{i}};
    u^{y_{l}},v^{y_{l}}}_{t_{i},t_{i+1}}
 [W_{n_{j}(N_{t_{i+1}}^{t_{i},N_{t_{i}}^{t,j}})}(t_{i+1},X^{t_{i},X^{t,x;u^{\varepsilon},v^{\varepsilon}}_{t_{i}};
 u^{y_{l}},v^{y_{l}}}_{t_{i+1}})]1_{O_{l}}(X^{t,x;u^{\varepsilon},v^{\varepsilon}}_{t_{i}}).
\end{eqnarray*}
From  (\ref{Je27}) it follows that
\begin{eqnarray*}\label{}
 ^{j}G^{t,x;u^{\varepsilon},v^{\varepsilon}}_{t_{i},t_{i+1}}
 [W_{n_{j}(N_{t_{i+1}}^{t,j})}(t_{i+1},X^{t,x;u^{\varepsilon},v^{\varepsilon}}_{t_{i+1}})]
&\geq&\sum\limits_{l\geq1}
[W_{n_{j}(N_{t_{i}}^{t,j})}(t_{i},X^{t,x;u^{\varepsilon},v^{\varepsilon}}_{t_{i}})-\varepsilon]
1_{O_{l}}(X^{t,x;u^{\varepsilon},v^{\varepsilon}}_{t_{i}})\\
  &= & W_{n_{j}(N_{t_{i}}^{t,j})}(t_{i},X^{t,x;u^{\varepsilon},v^{\varepsilon}}_{t_{i}})-\varepsilon,
\end{eqnarray*}
from which we conclude the wished result.
\end{proof} \vskip 2mm

\noindent Let us come, finally, to the proof of Proposition
\ref{Jp11}. \vskip 2mm

\begin{proof}
Let $j=1,2$ be arbitrarily fixed, and let  $\tau$ be a stopping time
independent of $\mathcal {F}_{t}$, such that $t\leq \tau\leq T$.  We
put $t_{i}=\frac{i(T-t)}{2^{n}}+t$, $A_{i}=\Big\{t_{i-1}\leq \tau <
t_{i}\Big\}, i=1, \cdots,2^{n},$  and define $\tau_{n}
=\sum\limits_{i=1}^{2^{n}}t_{i} 1_{A_{i}}.$  It is obvious that
$0\leq\tau_{n}-\tau\leq 2^{-n}$.   From Lemma \ref{Jle6} we know
that, for all $\varepsilon>0,$ there exists
$(u^{\varepsilon},v^{\varepsilon})\in\mathcal{U}_{t,T}\times\mathcal{V}_{t,T}$
independent of $\mathcal {F}_{t}$ such that, for all stopping time
$\tau$ independent of $\mathcal {F}_{t}$, and for all $i\  (0\leq i
\leq 2^{n}-1)$,
\begin{eqnarray*}\label{}
 W_{n_{j}(N_{t_{i}}^{t,j})}(t_{i},X^{t,x;u^{\varepsilon},v^{\varepsilon}}_{t_{i}})-\varepsilon_{0}\leq
\ ^{j}G^{t,x;u^{\varepsilon},v^{\varepsilon}}_{t_{i},t_{i+1}}
 [W_{n_{j}(N^{t,j}_{t_{i+1}})}(t_{i+1},X^{t,x;u^{\varepsilon},v^{\varepsilon}}_{t_{i+1}})],\
 \mathbb{P}- a.s.
\end{eqnarray*}
where   $\varepsilon_{0}>0$ depends on $\varepsilon$ and $n$,  and
will be specified after.

From the Lemmas \ref{l8} and \ref{l1} it follows that, for $ 0\leq i
\leq 2^{n}-1$,
\begin{eqnarray*}\label{}
&& ^{j}G^{t,x;u^{\varepsilon},v^{\varepsilon}}_{t_{i},T}
 [\Phi_{N_{T}^{t,j}}(X^{t,x;u^{\varepsilon},v^{\varepsilon}}_{T})]
 =\  ^{j}G^{t,x;u^{\varepsilon},v^{\varepsilon}}_{t_{i},t_{2^{n}-1}}
 [^{j}G^{t,x;u^{\varepsilon},v^{\varepsilon}}_{t_{2^{n}-1},T}
 [\Phi_{N_{T}^{t,j}}(X^{t,x;u^{\varepsilon},v^{\varepsilon}}_{T})]]\\
  &\geq&\  ^{j}G^{t,x;u^{\varepsilon},v^{\varepsilon}}_{t_{i},t_{2^{n}-1}}
  [W_{n_{j}(N_{t_{2^{n}-1}}^{t,j})}(t_{2^{n}-1},X^{t,x;u^{\varepsilon},v^{\varepsilon}}_{t_{2^{n}-1}})-\varepsilon_{0}]\\
   &\geq&\  ^{j}G^{t,x;u^{\varepsilon},v^{\varepsilon}}_{t_{i},t_{2^{n}-1}}
  [W_{n_{j}(N_{t_{2^{n}-1}}^{t,j})}(t_{2^{n}-1},X^{t,x;u^{\varepsilon},v^{\varepsilon}}_{t_{2^{n}-1}})]-C\varepsilon_{0}\\
   &\geq&\cdots\geq \ ^{j}G^{t,x;u^{\varepsilon},v^{\varepsilon}}_{t_{i},t_{i+1}}
  [W_{n_{j}(N_{t_{i+1}}^{t,j})}(t_{i+1},X^{t,x;u^{\varepsilon},v^{\varepsilon}}_{t_{i+1}})]-C(2^{n}-i)\varepsilon_{0}\\
 &\geq&
 W_{n_{j}(N_{t_{i}}^{t,j})}(t_{i},X^{t,x;u^{\varepsilon},v^{\varepsilon}}_{t_{i}})-C(2^{n}-i+1)\varepsilon_{0},
\end{eqnarray*}
from where
\begin{eqnarray*}\label{}
&& ^{j}G^{t,x;u^{\varepsilon},v^{\varepsilon}}_{\tau_{n},T}
 [\Phi_{N_{T}^{t,j}}(X^{t,x;u^{\varepsilon},v^{\varepsilon}}_{T})]
= \sum\limits_{i=1}^{2^{n}} 1_{A_{i}}\
^{j}G^{t,x;u^{\varepsilon},v^{\varepsilon}}_{t_{i},T}
 [\Phi_{N_{T}^{t,j}}(X^{t,x;u^{\varepsilon},v^{\varepsilon}}_{T})]\nonumber\\
 &\geq&   \sum\limits_{i=1}^{2^{n}}
 1_{A_{i}}W_{n_{j}(N_{t_{i}}^{t,j})}(t_{i},X^{t,x;u^{\varepsilon},v^{\varepsilon}}_{t_{i}})
 -C \sum\limits_{i=1}^{2^{n}} 1_{A_{i}}(2^{n}-i+1)\varepsilon_{0}\nonumber\\
  &\geq&   \sum\limits_{i=1}^{2^{n}}
 1_{A_{i}}W_{n_{j}(N_{t_{i}}^{t,j})}(t_{i},X^{t,x;u^{\varepsilon},v^{\varepsilon}}_{t_{i}})
 -C 2^{n}\varepsilon_{0}=  W_{n_{j}(N_{\tau_{n}}^{t,j})}(\tau_{n},X^{t,x;u^{\varepsilon},v^{\varepsilon}}_{\tau_{n}})
 -C  2^{n}\varepsilon_{0}.
\end{eqnarray*}
Therefore,
\begin{eqnarray}\label{Je29}
 ^{j}G^{t,x;u^{\varepsilon},v^{\varepsilon}}_{\tau,T}
 [\Phi_{N_{T}^{t,j}}(X^{t,x;u^{\varepsilon},v^{\varepsilon}}_{T})]
&=&\ ^{j}G^{t,x;u^{\varepsilon},v^{\varepsilon}}_{\tau,\tau_{n}}
 [\ ^{j}G^{t,x;u^{\varepsilon},v^{\varepsilon}}_{\tau_{n},T}
 [\Phi_{N_{T}^{t,j}}(X^{t,x;u^{\varepsilon},v^{\varepsilon}}_{T})]]\nonumber\\
  &\geq&  ^{j}G^{t,x;u^{\varepsilon},v^{\varepsilon}}_{\tau,\tau_{n}}
 [W_{n_{j}(N_{\tau_{n}}^{t,j})}(\tau_{n},X^{t,x;u^{\varepsilon},v^{\varepsilon}}_{\tau_{n}})
 -C  2^{n}\varepsilon_{0}]\nonumber\\
  &\geq&\  ^{j}G^{t,x;u^{\varepsilon},v^{\varepsilon}}_{\tau,\tau_{n}}
 [W_{n_{j}(N_{\tau_{n}}^{t,j})}(\tau_{n},X^{t,x;u^{\varepsilon},v^{\varepsilon}}_{\tau_{n}})]-C  2^{n}\varepsilon_{0}
 \nonumber\\
   &=&  ^{j}G^{t,x;u^{\varepsilon},v^{\varepsilon}}_{\tau,\tau_{n}}
 [W_{n_{j}(N_{\tau_{n}}^{t,j})}(\tau_{n},X^{t,x;u^{\varepsilon},v^{\varepsilon}}_{\tau_{n}})]-\frac{\varepsilon}{2} \nonumber\\
  &=& W_{n_{j}(N_{\tau}^{t,j})}(\tau,X^{t,x;u^{\varepsilon},v^{\varepsilon}}_{\tau})
    -\frac{\varepsilon}{2}+I_{2}-I_{1},
\end{eqnarray}
for  $\varepsilon_{0}=\dfrac{\varepsilon}{C2^{n+1}}$, where
\begin{eqnarray}\label{Jeq110}
 I_{1}&=& ^{j}G^{t,x;u^{\varepsilon},v^{\varepsilon}}_{\tau,T}
 [\Phi_{N_{T}^{t,j}}(X^{t,x;u^{\varepsilon},v^{\varepsilon}}_{T})]-\
    ^{j}G^{t,x;u^{\varepsilon},v^{\varepsilon}}_{\tau,\tau_{n}}
 [W_{n_{j}(N_{\tau_{n}}^{t,j})}(\tau_{n},X^{t,x;u^{\varepsilon},v^{\varepsilon}}_{\tau_{n}})]
    +\frac{\varepsilon}{2},\nonumber\\
 I_{2}&=&  ^{j}G^{t,x;u^{\varepsilon},v^{\varepsilon}}_{\tau,T}
 [\Phi_{N_{T}^{t,j}}(X^{t,x;u^{\varepsilon},v^{\varepsilon}}_{T})]-
    W_{n_{j}(N_{\tau}^{t,j})}(\tau,X^{t,x;u^{\varepsilon},v^{\varepsilon}}_{\tau})
    +\frac{\varepsilon}{2},
\end{eqnarray}
and we put, for $s\in[\tau,\tau_{n}]$,
$$y^{1}_{s}=\ ^{j}G^{t,x;u^{\varepsilon},v^{\varepsilon}}_{s,\tau_{n}}
 [W_{n_{j}(N_{\tau_{n}}^{t,j})}(\tau_{n},X^{t,x;u^{\varepsilon},v^{\varepsilon}}_{\tau_{n}})].$$
Obviously, $(y^{1}_{s})_{s\in[\tau,\tau_{n}]}$ is the solution of
the following BSDE:
\begin{eqnarray*}\label{J}
y^{1}_{s}&=&
W_{n_{j}(N_{\tau_{n}}^{t,j})}(\tau_{n},X^{t,x;u^{\varepsilon},v^{\varepsilon}}_{\tau_{n}})
+\displaystyle \int_{s}^{\tau_{n}}
f_{N^{t_{i},j}_{r}}(r,X^{t,x;u^{\varepsilon},v^{\varepsilon}}_r,
y_r^{1},h_r^{1},z^{1}_r,u^{\varepsilon}_r,
v^{\varepsilon}_r)dr\\&&\qquad-\displaystyle
\lambda\int_{s}^{\tau_{n}} h^{1}_rdr-\displaystyle
\int_{s}^{\tau_{n}} z^{1}_rdB_r-\displaystyle \int_{s}^{\tau_{n}}
h^{1}_rd\widetilde{N}_r, \quad s\in[\tau,\tau_{n}].
\end{eqnarray*}
We will compare $y^{1}$ with the process constant in time
$y^{2}_{s}=
W_{n_{j}(N_{\tau}^{t,j})}(\tau,X^{t,x;u^{\varepsilon},v^{\varepsilon}}_{\tau}),
\ s\in[\tau,\tau_{n}].$  We observe that $(y^{2},z^{2})$ is the
solution of a BSDE which driving coefficient equals to zero and
$z^{2}=0$. Hence, from Lemma \ref{l1} it follows that
\begin{eqnarray}\label{Je31}
&&| ^{j}G^{t,x;u^{\varepsilon},v^{\varepsilon}}_{\tau,\tau_{n}}
 [W_{n_{j}(N_{\tau_{n}}^{t,j})}(\tau_{n},X^{t,x;u^{\varepsilon},v^{\varepsilon}}_{\tau_{n}})]
-W_{n_{j}(N_{\tau}^{t,j})}(\tau,X^{t,x;u^{\varepsilon},v^{\varepsilon}}_{\tau})|^{2}\nonumber\\
&\leq&C\mathbb{E}[|W_{n_{j}(N_{\tau_{n}}^{t,j})}(\tau_{n},X^{t,x;u^{\varepsilon},v^{\varepsilon}}_{\tau_{n}})
-W_{n_{j}(N_{\tau}^{t,j})}(\tau,X^{t,x;u^{\varepsilon},v^{\varepsilon}}_{\tau})|^{2}
\Big|\mathcal {F}_{\tau}]\nonumber\\
&&+C\mathbb{E}[\displaystyle \int_{\tau}^{\tau_{n}}
|f_{N^{t_{i},j}_{r}}(r,X^{t,x;u^{\varepsilon},v^{\varepsilon}}_r,
y_r^{2},0,0,u^{\varepsilon}_r, v^{\varepsilon}_r)|^{2}dr
\Big|\mathcal {F}_{\tau}]\nonumber\\
&\leq&C\mathbb{E}[|W_{n_{j}(N_{\tau_{n}}^{t,j})}(\tau_{n},X^{t,x;u^{\varepsilon},v^{\varepsilon}}_{\tau_{n}})
-W_{n_{j}(N_{\tau}^{t,j})}(\tau,X^{t,x;u^{\varepsilon},v^{\varepsilon}}_{\tau})|^{2}
\Big|\mathcal {F}_{\tau}]+ C2^{-n},
\end{eqnarray}
where we have used the boundedness of $f_{i}, i=1,2$.

By virtue of the Propositions \ref{p4} and \ref{p21} we  deduce that
\begin{eqnarray*}\label{}
 &&\mathbb{E}[|W_{n_{j}(N_{\tau_{n}}^{t,j})}(\tau_{n},X^{t,x;u^{\varepsilon},v^{\varepsilon}}_{\tau_{n}})
 -
 W_{n_{j}(N_{\tau}^{t,j})}(\tau,X^{t,x;u^{\varepsilon},v^{\varepsilon}}_{\tau})|^{2}]\\
 &\leq& 3\mathbb{E}[|W_{n_{j}(N_{\tau_{n}}^{t,j})}(\tau_{n},X^{t,x;u^{\varepsilon},v^{\varepsilon}}_{\tau_{n}})
 - W_{n_{j}(N_{\tau_{n}}^{t,j})}(\tau_{n},X^{t,x;u^{\varepsilon},v^{\varepsilon}}_{\tau})|^{2}]\\
 &&  +3\mathbb{E}[|W_{n_{j}(N_{\tau_{n}}^{t,j})}(\tau_{n},X^{t,x;u^{\varepsilon},v^{\varepsilon}}_{\tau})
 -
 W_{n_{j}(N_{\tau_{n}}^{t,j})}(\tau,X^{t,x;u^{\varepsilon},v^{\varepsilon}}_{\tau})|^{2}]\\&&
 +3\mathbb{E}[|W_{n_{j}(N_{\tau_{n}}^{t,j})}(\tau,X^{t,x;u^{\varepsilon},v^{\varepsilon}}_{\tau})
 - W_{n_{j}(N_{\tau}^{t,j})}(\tau,X^{t,x;u^{\varepsilon},v^{\varepsilon}}_{\tau})|^{2}]\\
 &\leq&C\mathbb{E}[|X^{t,x;u^{\varepsilon},v^{\varepsilon}}_{\tau_{n}}
 - X^{t,x;u^{\varepsilon},v^{\varepsilon}}_{\tau}|^{2}]
 +C\mathbb{E}[(1+|X^{t,x;u^{\varepsilon},v^{\varepsilon}}_{\tau}|^{2})|\tau-\tau_{n}|]
  +C\mathbb{E}[1_{\{N_{\tau_{n}}^{t,j}\neq N_{\tau}^{t,j}\}}]
  \\ &\leq& C2^{-n},
\end{eqnarray*}
where $C>0$ is a constant which depends on $x$. Therefore, from
(\ref{Je31}) it follows that
\begin{eqnarray*}\label{}
\mathbb{E}[|\
^{j}G^{t,x;u^{\varepsilon},v^{\varepsilon}}_{\tau,\tau_{n}}
 [W_{n_{j}(N_{\tau_{n}}^{t,j})}(\tau_{n},X^{t,x;u^{\varepsilon},v^{\varepsilon}}_{\tau_{n}})]
-W_{n_{j}(N_{\tau}^{t,j})}(\tau,X^{t,x;u^{\varepsilon},v^{\varepsilon}}_{\tau})|^{2}]
\leq C2^{-n}.
\end{eqnarray*}
Consequently,  $ \mathbb{E}[| I_{1} - I_{2}|^{2}] \leq C2^{-n}.$
From (\ref{Je29}) we see that $I_{1}\geq 0.$ Therefore,  by the
above inequality we have
\begin{eqnarray*}\label{}
\mathbb{P}(I_{2}\leq - \frac{\varepsilon}{2}) \leq \mathbb{P}( |
I_{1} - I_{2}|\geq \frac{\varepsilon}{2}) \leq \frac{4\mathbb{E}[|
I_{1} - I_{2}|^{2}]}{\varepsilon^{2}}\leq
\frac{4C2^{-n}}{\varepsilon^{2}}\leq\varepsilon,
\end{eqnarray*}
where we choose $n$ such that $4C2^{-n}\leq \varepsilon^{3}$, i.e.,
$n\geq(2+\dfrac{\ln C-3\ln\varepsilon }{\ln 2})$, and from
(\ref{Jeq110}) we deduce
\begin{eqnarray*}\label{}
\mathbb{P}\Big(\
 W_{n_{j}(N_{\tau}^{t,j})}(\tau,X^{t,x;u^{\varepsilon},v^{\varepsilon}}_{\tau})-\varepsilon\leq
\ ^{j}G^{t,x;u^{\varepsilon},v^{\varepsilon}}_{\tau,T}
 [\Phi_{N_{T}^{t,j}}(X^{t,x;u^{\varepsilon},v^{\varepsilon}}_{T})]\Big)> 1-
 \varepsilon,
\end{eqnarray*}
i.e.,
\begin{eqnarray*}\label{}
\mathbb{P}\Big(\
 W_{n_{j}(N_{\tau}^{t,j})}(\tau,X^{t,x;u^{\varepsilon},v^{\varepsilon}}_{\tau})-\varepsilon\leq
\ ^{j} Y^{t,x;u^{\varepsilon},v^{\varepsilon}}_{\tau})]\Big)> 1-
\varepsilon.
\end{eqnarray*}
Since $(u^{\varepsilon},v^{\varepsilon})$ as well as $\tau$ are
independent of $\mathcal {F}_{t}$,  the event   $\Big\{
W_{n_{j}(N_{\tau}^{t,j})}(\tau,X^{t,x;u^{\varepsilon},v^{\varepsilon}}_{\tau})-\varepsilon\leq
\ ^{j} Y^{t,x;u^{\varepsilon},v^{\varepsilon}}_{\tau})\Big\}$ is
independent of $\mathcal {F}_{t}$, from which we see that the
conditional probability $\mathbb{P}(\cdot |\mathcal {F}_{t})$ of
$\Big\{
W_{n_{j}(N_{\tau}^{t,j})}(\tau,X^{t,x;u^{\varepsilon},v^{\varepsilon}}_{\tau})-\varepsilon\leq
\ ^{j} Y^{t,x;u^{\varepsilon},v^{\varepsilon}}_{\tau})\Big\}$
coincides with its probability. Consequently,
\begin{eqnarray*}\label{J}
\mathbb{P}\Big(\
 W_{n_{j}(N_{\tau}^{t,j})}(\tau,X^{t,x;u^{\varepsilon},v^{\varepsilon}}_{\tau})-\varepsilon\leq
\ ^{j} Y^{t,x;u^{\varepsilon},v^{\varepsilon}}_{\tau}) |\ \mathcal
{F}_{t}\Big)> 1- \varepsilon.
\end{eqnarray*}
 The proof is complete.
\end{proof}

\section{Proof of the Theorems \ref{t5} and \ref{t1} }\label{NS5}

In this section we still use the notations in Sections \ref{NS2} and
\ref{NS3}.

\subsection{Proof of Theorem \ref{t5}}\label{A1}
We have postponed to this section the proof of the DPP. We shall
establish it first for deterministic times, then we deduce the
 general version for stopping times.
\subsubsection{Dynamic programming principle for deterministic times}
\begin{theorem}\label{t3}
 Under the assumptions $(H4)$ and $(H5)$,  the following dynamic programming principle (DPP) holds: For any
$0\leq t<t+\delta \leq T,\ x\in {\mathbb{R}}^n, i=1,2,$
\begin{eqnarray}\label{et3}
W_{i}(t,x) &=&\esssup_{\alpha \in {\mathcal{A}}_{t,
t+\delta}}\essinf_{\beta \in {\mathcal{B}}_{t, t+\delta}}\
^{i}G^{t,x;\alpha,\beta}_{t,t+\delta}
[W_{N_{t+\delta}^{t,i}}(t+\delta,
X^{t,x;\alpha,\beta}_{t+\delta})].\\
 U_{i}(t,x)
&=&\essinf_{\beta \in {\mathcal{B}}_{t, t+\delta}}\esssup_{\alpha
\in {\mathcal{A}}_{t, t+\delta}}\
^{i}G^{t,x;\alpha,\beta}_{t,t+\delta}
[U_{N_{t+\delta}^{t,i}}(t+\delta,
X^{t,x;\alpha,\beta}_{t+\delta})].\nonumber
\end{eqnarray}
\end{theorem}

\begin{remark}
Recall that for $i=1,2, (t,x)\in [0,T]\times \mathbb{R}^{n}$,
$W(t,(x,i)):=W_{i}(t,x)$ is the value function of the stochastic
differential game which dynamics is given by the process $(X^{t,x;
u, v}, N^{t,i})$, and which cost functional is defined by our BSDE
(\ref{e2}). The DPP for games with jumps but of the type "strategy
against control" was proved in \cite{BHL2010}, and before, in
another framework by Biswas \cite{B2010}. However, unlike
\cite{BHL2010} we have to deal here with games of the type "NAD
strategy against NAD strategy".
\end{remark}

\begin{proof}
For  arbitrarily fixed $i=1,2,$ we only give  the proof for
$W_{i}(t,x)$, since for $U_{i}(t,x)$ we can use a symmetric
argument.  Let us denote by $W_\delta(t,x)$ the right hand side of
(\ref{et3}).  Using the arguments in Proposition \ref{p1} we can
show that $W_\delta(t,x)$ is deterministic. We split now the proof
into the following two lemmas.
\end{proof}

\begin{lemma}\label{l9}
$ W_{i}(t, x)\leq W_\delta(t,x).$
\end{lemma}

\begin{proof}
From the definition of $W_\delta(t,x)$\ it follows that
\begin{eqnarray*}
W_\delta(t,x)&=& \esssup_{\alpha_{1} \in {\mathcal{A}}_{t,
t+\delta}}\essinf_{\beta_{1} \in {\mathcal{B}}_{t, t+\delta}}\
^{i}G^{t,x;\alpha_{1},\beta_{1}}_{t,t+\delta}
[W_{N_{t+\delta}^{t,i}}(t+\delta, X^{t,x;\alpha_{1},\beta_{1}}_{t+\delta})]\\
&=&\esssup_{\alpha_1 \in {\mathcal{A}}_{t, t+\delta}}I_\delta(t, x,
\alpha_1),
\end{eqnarray*}
where $ I_\delta(t, x, \alpha_1)=\essinf_{\beta_{1} \in
{\mathcal{B}}_{t, t+\delta}}\
^{i}G^{t,x;\alpha_{1},\beta_{1}}_{t,t+\delta}
[W_{N_{t+\delta}^{t,i}}(t+\delta,
X^{t,x;\alpha_{1},\beta_{1}}_{t+\delta})].$   Then, there exists a
sequence  $\{\alpha_n^1,\ n\geq 1\}\subset {\mathcal{A}}_{t,
t+\delta},$ such that $W_\delta(t,x)=\sup_{n\geq 1}I_\delta(t, x,
\alpha_n^1),\ \mathbb{P}\mbox{-a.s.}$ For any $\varepsilon>0,$\ let
us put
$$\Lambda_n:=\Big\{I_\delta(t, x,
\alpha_n^1)+\varepsilon\geq W_\delta(t,x), I_\delta(t, x,
\alpha_j^1)+\varepsilon< W_\delta(t,x), 1\leq j \leq n-1\Big\}\in
{\mathcal{F}}_{t},\ n\geq 1.$$ Obviously,  $\{\Lambda_n,\ n\geq
1\}$\ is an $(\Omega, {\mathcal{F}}_{t})$-partition. We define
$\alpha^\varepsilon_1:=\sum_{n\geq 1}1_{\Lambda_n}\alpha_n^1$, then
by straight-forward proof it can be shown that
$\alpha^\varepsilon_1$ belongs to ${\mathcal{A}}_{t, t+\delta}.$  We
let et $(u^{n},v^{n})\in \mathcal {U}_{t,T}\times \mathcal
{V}_{t,T}$ be associated with $(\alpha_n^1,\beta_{1})$ by Lemma
\ref{l5}, and put
\begin{eqnarray*}
 u^\varepsilon_1:=\sum_{n\geq 1}1_{\Lambda_n}u^n,\quad
  v^\varepsilon_1:=\sum_{n\geq 1}1_{\Lambda_n}v^n.
\end{eqnarray*}
Then straight forward arguments allow to verify that
$(u_{1}^{\varepsilon},v_{1}^{\varepsilon})\in \mathcal
{U}_{t,T}\times \mathcal {V}_{t,T}$ is associated with
$(\alpha^\varepsilon,\beta_{1})$ by Lemma \ref{l5} such that $
 \alpha^\varepsilon(v^\varepsilon_1)=u^\varepsilon_1,\
  \beta_{1}(u^\varepsilon_1)=v^\varepsilon_1.$
Consequently, the uniqueness of the FBSDE and the definition of NAD
strategies allows to show that, for all $\beta_{1}\in
{\mathcal{B}}_{t, t+\delta},$ $$X^{t, x; \alpha_{1}^{\varepsilon},
\beta_{1}}=\sum_{n\geq 1} 1_{\Lambda_n}X^{t, x; \alpha_{n}^{1},
\beta_{1}},$$ and
\begin{eqnarray}\label{e32}
 \ ^{i}G^{t,x;\alpha_{1}^{\varepsilon},\beta_{1}}_{t,t+\delta}
[W_{N_{t+\delta}^{t,i}}(t+\delta,X^{t,x;\alpha_{1}^{\varepsilon},\beta_{1}}_{t+\delta})]
= \sum_{n\geq 1} \ ^{i}G^{t,x;\alpha_{n}^{1},\beta_{1}}_{t,t+\delta}
[W_{N_{t+\delta}^{t,i}}(t+\delta,X^{t,x;\alpha_{n}^{1},\beta_{1}}_{t+\delta})]1_{\Lambda_n}.
\end{eqnarray}
Therefore, for  all $\beta_1\in {\mathcal{B}}_{t, t+\delta}$,
 \begin{eqnarray}\label{e16}
&&W_\delta(t,x)\leq\sum_{n\geq 1} 1_{\Lambda_n}I_\delta(t,
x,\alpha_n^1)+\varepsilon \nonumber\\ &&\leq \sum_{n\geq 1}
1_{\Lambda_n}\ ^{i}G^{t,x;\alpha_{n}^{1},\beta_{1}}_{t,t+\delta}
[W_{N_{t+\delta}^{t,i}}(t+\delta,X^{t,x;\alpha_{n}^{1},\beta_{1}}_{t+\delta})]+\varepsilon \nonumber\\
&&= \ ^{i}G^{t,x;\alpha_{1}^{\varepsilon},\beta_{1}}_{t,t+\delta}
[W_{N_{t+\delta}^{t,i}}(t+\delta,X^{t,x;\alpha_{1}^{\varepsilon},\beta_{1}}_{t+\delta})]+\varepsilon,\
\mathbb{P}\mbox{-a.s.}
\end{eqnarray}
We let $\{O_j\}_{j\geq1}\subset {\mathcal{B}}({\mathbb{R}}^n)$\ be a
partition of ${\mathbb{R}}^n$\ such that
$\sum\limits_{j\geq1}O_j={\mathbb{R}}^n\ \mbox{and}\
\mbox{diam}(O_j)\leq \varepsilon,\ j\geq 1.$  Let us fix an element
$y_j$  of $O_j,\ j\geq1.$  Then, from the definition of
$W_{k}(t+\delta,y), k=1,2,$ it follows (through a procedure already
used above) that there exists $\alpha^{j,k}_{2}\in
{\mathcal{A}}_{t+\delta, T}$\ \ such that,  for  all $\beta_2\in
{\mathcal{B}}_{t+\delta,T}$,
\begin{eqnarray*}
 W_{k}(t+\delta,y_{j})\leq \essinf_{\beta_{2} \in
{\mathcal{B}}_{t+\delta,T}}J_{k}(t+\delta, y_{j}; \alpha^{j,k}_{2},
\beta_{2})+\varepsilon,\ \mathbb{P}\mbox{-a.s.}
\end{eqnarray*}
Therefore, for  all $\beta_2\in {\mathcal{B}}_{t+\delta,T}$,
\begin{eqnarray}\label{e17}
 &&W_{N_{t+\delta}^{t,i}}(t+\delta,y_{j})
 =\sum\limits_{k=1}^{2} 1_{\{N_{t+\delta}^{t,i}=k\}}
 W_{k}^{t,i}(t+\delta,y_{j})\nonumber\\
 &&\leq
 \sum\limits_{k=1}^{2} 1_{\{N_{t+\delta}^{t,i}=k\}} J_{k}(t+\delta, y_{j}; \alpha^{j,k}_{2},
\beta_{2})+\varepsilon\nonumber\\
&& \leq J_{N_{t+\delta}^{t,i}}(t+\delta, y_{j};
\alpha^{j,N_{t+\delta}^{t,i}}_{2}, \beta_{2})+\varepsilon,\
\mathbb{P}\mbox{-a.s.}
\end{eqnarray}
 For $\beta \in {\mathcal{B}}_{t,T}$ and $u_{2} \in
{\mathcal{U}}_{t+\delta,T}$, we let
$\beta_{1}(u_{1}):=\beta(u_1\oplus u_2)|_{[t,t+\delta]},\ u_1\in
{\mathcal{U}}_{t, t+\delta},\ \mbox{where}\ u_1\oplus u_2= u_1 1_
{[t, t+\delta]}+\ u_2 1_{(t+\delta, T]}$. Then $\beta_{1} \in
{\mathcal{B}}_{t,t+\delta}$. Notice that $\beta_{1}$ doesn't depend
on $u_2$ thanks  to the nonanticipativity  of $\beta$.  Since
$(\alpha_{1}^{\varepsilon}, \beta_{1}) \in
{\mathcal{A}}_{t,t+\delta} \times {\mathcal{B}}_{t,t+\delta}$, we
know from Lemma \ref{l5}  that there exists a  unique pair
$(u_{1}^{\varepsilon}, v_{1}^{\varepsilon})\in
{\mathcal{U}}_{t,t+\delta} \times {\mathcal{V}}_{t,t+\delta}$ such
that $\alpha_{1}^{\varepsilon}(v_{1}^{\varepsilon})=
u_{1}^{\varepsilon}$ and $\beta_{1}(u_{1}^{\varepsilon})=
v_{1}^{\varepsilon}$. Let us define
$$\beta^{\varepsilon}_{2}(u_{2}):=\beta(u_1^{\varepsilon}\oplus u_2
)|_{[t+\delta,T]}, \ \mbox{for}\  u_{2}\in {\mathcal{U}}_{t+\delta,
T}, \quad \alpha^\varepsilon_2:=\sum_{j\geq 1}
1_{O_j}(X^{t,x;\alpha_{1}^{\varepsilon},\beta_{1}}_{t+\delta})\alpha^{j,N_{t+\delta}^{t,i}}_2.$$
Moreover, we put
$\alpha^\varepsilon(v):=\alpha_{1}^\varepsilon(v_{1})\oplus
\alpha_{2}^\varepsilon(v_{2})$, for $v=v_{1}\oplus v_{2}, v_{1} \in
{\mathcal{V}}_{t,t+\delta}, v_{2} \in {\mathcal{V}}_{t+\delta,T}$.
  Since $\alpha_{1}^{\varepsilon} \in {\mathcal{A}}_{t,t+\delta}$
and $\alpha_{2}^{\varepsilon} \in {\mathcal{A}}_{t+\delta,T}$ are
NAD strategies,
  $\alpha^{\varepsilon}$\ is also an NAD strategy.

From Proposition \ref{p4} it follows that
\begin{eqnarray*}
W_{N_{t+\delta}^{t,i}}(t+\delta,X^{t,x;\alpha_{1}^{\varepsilon},\beta_{1}}_{t+\delta})]
\leq \sum_{j\geq 1}
1_{O_j}(X^{t,x;\alpha_{1}^{\varepsilon},\beta_{1}}_{t+\delta})W_{N_{t+\delta}^{t,i}}(t+\delta,
y_{j})+C\varepsilon,
\end{eqnarray*}
 which together with (\ref{e16}), (\ref{e17})  and Lemma \ref{l8}
yields
\begin{eqnarray*}
W_\delta(t,x) &\leq& \
^{i}G^{t,x;\alpha_{1}^{\varepsilon},\beta_{1}}_{t,t+\delta}
[W_{N_{t+\delta}^{t,i}}(t+\delta,X^{t,x;\alpha_{1}^{\varepsilon},\beta_{1}}_{t+\delta})]+\varepsilon,\\
&\leq&
^{i}G^{t,x;\alpha_{1}^{\varepsilon},\beta_{1}}_{t,t+\delta}[\sum_{j\geq
1}
1_{O_j}(X^{t,x;\alpha_{1}^{\varepsilon},\beta_{1}}_{t+\delta})W_{N_{t+\delta}^{t,i}}(t+\delta,
y_{j})+C\varepsilon]+\varepsilon, \ \mathbb{P}\mbox{-a.s.}
\end{eqnarray*}
It follows from (\ref{e17}) and  the Lemmas \ref{l8},  \ref{l1} and
\ref{l5} that
\begin{eqnarray*}
W_\delta(t,x) &\leq&
^{i}G^{t,x;\alpha_{1}^{\varepsilon},\beta_{1}}_{t,t+\delta}[\sum_{j\geq
1}
1_{O_j}(X^{t,x;\alpha_{1}^{\varepsilon},\beta_{1}}_{t+\delta})W_{N_{t+\delta}^{t,i}}(t+\delta,
y_{j})]+C\varepsilon\\
&\leq&
^{i}G^{t,x;\alpha_{1}^{\varepsilon},\beta_{1}}_{t,t+\delta}[\sum_{j\geq
1}
1_{O_j}(X^{t,x;\alpha_{1}^{\varepsilon},\beta_{1}}_{t+\delta})J_{N_{t+\delta}^{t,i}}(t+\delta,
y_{j}; \alpha^{j,N_{t+\delta}^{t,i}}_{2},\beta_{2}^{\varepsilon})+\varepsilon]+C\varepsilon\\
&\leq& \
^{i}G^{t,x;\alpha_{1}^{\varepsilon},\beta_{1}}_{t,t+\delta}[\sum_{j\geq
1}
1_{O_j}(X^{t,x;\alpha_{1}^{\varepsilon},\beta_{1}}_{t+\delta})J_{N_{t+\delta}^{t,i}}(t+\delta,
y_{j}; \alpha^{\varepsilon}_{2},\beta_{2}^{\varepsilon})]+C\varepsilon\\
&\leq&
^{i}G^{t,x;\alpha_{1}^{\varepsilon},\beta_{1}}_{t,t+\delta}[J_{N_{t+\delta}^{t,i}}(t+\delta,
X^{t,x;\alpha_{1}^{\varepsilon},\beta_{1}}_{t+\delta};
\alpha^{\varepsilon}_{2},\beta_{2}^{\varepsilon})
+C\varepsilon]+C\varepsilon\\
&\leq&
^{i}G^{t,x;\alpha_{1}^{\varepsilon},\beta_{1}}_{t,t+\delta}[J_{N_{t+\delta}^{t,i}}(t+\delta,
X^{t,x;\alpha_{1}^{\varepsilon},\beta_{1}}_{t+\delta}; \alpha^{\varepsilon}_{2},\beta_{2}^{\varepsilon})]+C\varepsilon\\
&=& ^{i}G^{t,x;\alpha_{1}^{\varepsilon},\beta_{1}}_{t,t+\delta}[ \
^{N_{t+\delta}^{t,i}}Y_{t+\delta}^{t+\delta,
X^{t,x;\alpha_{1}^{\varepsilon},\beta_{1}}_{t+\delta}; \alpha^{\varepsilon}_{2},\beta_{2}^{\varepsilon}}]+C\varepsilon\\
&=& ^{i}G^{t,x;\alpha_{1}^{\varepsilon},\beta_{1}}_{t,t+\delta}[ \
^{i}Y_{t+\delta}^{t,x; \alpha^{\varepsilon},\beta}]+C\varepsilon\\
&=&J_{i}(t,x;\alpha^{\varepsilon},\beta)+ C\varepsilon,\ \
\mathbb{P}\mbox{-a.s,} \ \ \mbox{for all} \ \beta \in
{\mathcal{B}}_{t,T}.
\end{eqnarray*}
Consequently, due to the choice of $\alpha^{\varepsilon}$ we
conclude $W_\delta(t,x)\leq W_{i}(t,x)+ C\varepsilon,\
\mathbb{P}\mbox{-a.s.}$  Letting $\varepsilon\downarrow0$\ we have
$W_\delta(t,x)\leq W_{i}(t,x).$
\end{proof}
In order to complete the proof of Theorem \ref{t3} we have to prove
the converse inequality.
\begin{lemma}\label{l6}
$W_\delta(t,x)\geq W_{i}(t,x).$
\end{lemma}
\begin{proof} For an arbitrarily given $\alpha\in {\mathcal{A}}_{t, T}$ and a given
$v_2(\cdot)\in {\mathcal{V}}_{t+\delta, T},$\ let us define
$$\alpha_1(v_1):=\alpha(v_1\oplus v_2 )|_{[t,
t+\delta]},\ \mbox{ }\ v_1(\cdot)\in {\mathcal{V}}_{t, t+\delta},
$$
where $v_1\oplus v_2:=v_1 1_{[t, t+\delta]}+v_2 1_{(t+\delta, T]}$.
Then  $\alpha_1\in {\mathcal{A}}_{t, t+\delta}$, and  $\alpha_1$\
does not depend on the choice of $v_2(\cdot)\in
{\mathcal{V}}_{t+\delta, T},$  since  $\alpha$ is  nonanticipative.
Therefore, by virtue of  the definition of $W_\delta(t,x) $ we know
that
\begin{eqnarray*}
  W_\delta(t,x)\geq \essinf_{\beta_{1}
\in {\mathcal{B}}_{t, t+\delta}}\
^{i}G^{t,x;\alpha_1,\beta_1}_{t,t+\delta}
[W_{N_{t+\delta}^{t,i}}(t+\delta,
X^{t,x;\alpha_1,\beta_1}_{t+\delta})],\ \mathbb{P}\mbox{-a.s.}
\end{eqnarray*}
Let us denote by  $I_\delta(t, x, \alpha_{1}, \beta_{1}):=\
^{i}G^{t,x;\alpha_1,\beta_1}_{t,t+\delta}
[W_{N_{t+\delta}^{t,i}}(t+\delta,
X^{t,x;\alpha_1,\beta_1}_{t+\delta})]$. Then there exists a sequence
$\{\beta_n^1,\ n\geq 1\}\subset {\mathcal{B}}_{t, t+\delta}$\ such
that
$$I_\delta(t, x, \alpha_1):=\essinf_{\beta_{1}
\in {\mathcal{B}}_{t, t+\delta}}I_\delta(t, x, \alpha_{1},
\beta_{1})=\mbox{inf}_{n\geq 1}I_\delta(t, x, \alpha_{1},
\beta_{n}^{1}),\ \ \mathbb{P}\mbox{-a.s.}$$

For any $\varepsilon>0,$  $n\geq 1,$  we let
$$\Lambda_n:=\{I_\delta(t, x,
\alpha_1)\geq I_\delta(t, x, \alpha_{1}, \beta_{n}^{1})-\varepsilon,
I_\delta(t, x, \alpha_1)<I_\delta(t, x, \alpha_{1},
\beta_{j}^{1})-\varepsilon, 1\leq j\leq n-1\}\in {\mathcal{F}}_{t}.
$$ Then, $\{\Lambda_n\}$ is a partition  of $(\Omega,
{\mathcal{F}}_{t})$. We define $\beta^\varepsilon_1:=\sum_{n\geq 1}
1_{\Lambda_n}\beta_n^1$, and it can be proven that
$\beta^\varepsilon_1\in {\mathcal{B}}_{t, t+\delta}.$  Thanks to the
uniqueness of the solution of the FBSDE we have $$I_\delta(t, x,
\alpha_1, \beta_1^{\varepsilon})=\sum_{n\geq 1}
1_{\Lambda_n}I_\delta(t, x, \alpha_1, \beta_n^{1}),\
\mathbb{P}\mbox{-a.s.}$$ Indeed, this relation can be proved with an
argument analogous to that for (\ref{e32}).  Therefore,
\begin{eqnarray*}
W_\delta(t,x)\geq I_\delta(t, x, \alpha_1)&\geq &\sum_{n\geq
1}1_{\Lambda_n}I_\delta(t, x, \alpha_{1},
\beta_{n}^{1})-\varepsilon=I_\delta(t, x, \alpha_1,
\beta_1^{\varepsilon})-\varepsilon\\
&=& \ ^{i}G^{t,x;\alpha_1,\beta_1^{\varepsilon}}_{t,t+\delta}
[W_{N_{t+\delta}^{t,i}}(t+\delta,
X^{t,x;\alpha_1,\beta_1^{\varepsilon}}_{t+\delta})]-\varepsilon,\
\mathbb{P}\mbox{-a.s.}
\end{eqnarray*}
 Since $(\alpha_{1},
\beta_{1}^{\varepsilon}) \in {\mathcal{A}}_{t,t+\delta} \times
{\mathcal{B}}_{t,t+\delta}$,  by Lemma \ref{l5}  there exists an
unique pair $(u_{1}^{\varepsilon}, v_{1}^{\varepsilon})\in
{\mathcal{U}}_{t,t+\delta} \times {\mathcal{V}}_{t,t+\delta}$ such
that $\alpha_{1}(v_{1}^{\varepsilon})= u_{1}^{\varepsilon}$ and
$\beta_{1}^{\varepsilon}(u_{1}^{\varepsilon})= v_{1}^{\varepsilon}$.
 We also define
$\alpha_{2}^{\varepsilon}(v_{2}):=\alpha(v_1^{\varepsilon}\oplus v_2
)|_{[t+\delta,T]}, \ v_{2}\in {\mathcal{V}}_{t+\delta, T}$.

For any $y\in {\mathbb{R}}^n, k=1,2,$ from the definition of
$W_{k}(t+\delta,y)$  it follows that
$$W_{k}(t+\delta,y)\geq \essinf_{\beta_2 \in {\mathcal{B}}_{t+\delta, T}}
J_{k}(t+\delta, y; \alpha_2^{\varepsilon}, \beta_2),\
\mathbb{P}\mbox{-a.s,}\  k=1,2.$$ Using the Lipschitz continuity of
$W_{k}(t+\delta,\cdot)$ and $J_{k}(t+\delta,\cdot)$  we can prove by
approximating $X^{t,x;\alpha_1,\beta_1^\varepsilon}_{t+\delta}$ by a
finite-valued random variable  that
\begin{eqnarray*}
W_{k}(t+\delta, X^{t,x;\alpha_1,\beta_1^\varepsilon}_{t+\delta}
)\geq \essinf_{\beta_2 \in {\mathcal{B}}_{t+\delta,
T}}J_{k}(t+\delta, X^{t,x;\alpha_1,\beta_1^\varepsilon}_{t+\delta};
\alpha_2^{\varepsilon}, \beta_2),\ \mathbb{P}\mbox{-a.s,} \ k=1,2.
\end{eqnarray*}
Therefore,
\begin{eqnarray}\label{e18}
W_{N_{t+\delta}^{t,i}}(t+\delta,
X^{t,x;\alpha_1,\beta_1^\varepsilon}_{t+\delta} )\geq
\essinf_{\beta_2 \in {\mathcal{B}}_{t+\delta,
T}}J_{N_{t+\delta}^{t,i}}(t+\delta,
X^{t,x;\alpha_1,\beta_1^\varepsilon}_{t+\delta};
\alpha_2^{\varepsilon}, \beta_2),\ \mathbb{P}\mbox{-a.s.}
\end{eqnarray}
Moreover, there exists some sequence $\{\beta_n^2,\ n\geq 1\}\subset
{\mathcal{B}}_{t+\delta, T}$\ such that
$$\essinf_{\beta_2 \in {\mathcal{B}}_{t+\delta, T}}J_{N_{t+\delta}^{t,i}}(t+\delta,
X^{t,x;\alpha_1,\beta_1^\varepsilon}_{t+\delta};
\alpha_2^{\varepsilon}, \beta_2)=\mbox{inf}_{n\geq
1}J_{N_{t+\delta}^{t,i}}(t+\delta,
X^{t,x;\alpha_1,\beta_1^\varepsilon}_{t+\delta};
\alpha_2^{\varepsilon}, \beta_n^{2}),\ \mathbb{P}\mbox{-a.s,}$$ and
we set, for $n\geq 1,$
\begin{eqnarray*}
\Delta_n:&=&\Big\{\essinf_{\beta_2 \in {\mathcal{B}}_{t+\delta,
T}}J_{N_{t+\delta}^{t,i}}(t+\delta,
X^{t,x;\alpha_1,\beta_1^\varepsilon}_{t+\delta};
\alpha_2^{\varepsilon}, \beta_2)\geq
J_{N_{t+\delta}^{t,i}}(t+\delta,
X^{t,x;\alpha_1,\beta_1^\varepsilon}_{t+\delta};
\alpha_2^{\varepsilon}, \beta_n^{2})-\varepsilon, \\&& \quad
\essinf_{\beta_2 \in {\mathcal{B}}_{t+\delta,
T}}J_{N_{t+\delta}^{t,i}}(t+\delta,
X^{t,x;\alpha_1,\beta_1^\varepsilon}_{t+\delta};
\alpha_2^{\varepsilon}, \beta_2)< J_{N_{t+\delta}^{t,i}}(t+\delta,
X^{t,x;\alpha_1,\beta_1^\varepsilon}_{t+\delta};
\alpha_2^{\varepsilon}, \beta_j^{2})-\varepsilon,\\ && \quad 1\leq j
\leq n-1\Big\}\in {\mathcal{F}}_{t+\delta}.
\end{eqnarray*}
Obviously,  $\{\Delta_n,  n\geq 1\}$ is a partition  of $(\Omega,
{\mathcal{F}}_{t+\delta})$. Let us define
$$\beta_2^\varepsilon:=\sum_{n\geq
1} 1_{\Delta_n}\beta_n^{2},\quad  \beta^{\varepsilon}(u_1\oplus
u_2):=\beta_1^\varepsilon(u_1)\oplus \beta_2^\varepsilon(u_2),$$ for
$u_{1}\in {\mathcal{U}}_{t,t+\delta}$ and $u_{2}\in
{\mathcal{U}}_{t+\delta,T}$.  Then, from the uniqueness for the
equations (\ref{e3}) and (\ref{e2}), combined with Lemma  \ref{l5},
it follows that
\begin{eqnarray*}
J_{N_{t+\delta}^{t,i}}(t+\delta,X^{t,x;\alpha_1,\beta_1^\varepsilon}_{t+\delta};
\alpha_2^{\varepsilon},\beta_2^\varepsilon) &=&\
^{N_{t+\delta}^{t,i}}Y_{t+\delta}^{t+\delta,X^{t,x;\alpha_1,\beta_1^\varepsilon}_{t+\delta};
\alpha_2^{\varepsilon},\beta_2^\varepsilon}\\
&=&\sum_{n\geq 1}
1_{\Delta_n}J_{N_{t+\delta}^{t,i}}(t+\delta,X^{t,x;\alpha_1,\beta_1^\varepsilon}_{t+\delta};
\alpha_2^{\varepsilon},\beta_n^2),\ \mathbb{P}\mbox{-a.s,}
\end{eqnarray*}
which together with (\ref{e18}) yields
\begin{eqnarray*}
W_{N_{t+\delta}^{t,i}}(t+\delta,
X^{t,x;\alpha_1,\beta_1^\varepsilon}_{t+\delta} )&\geq&
\essinf_{\beta_2 \in {\mathcal{B}}_{t+\delta,
T}}J_{N_{t+\delta}^{t,i}}(t+\delta,
X^{t,x;\alpha_1,\beta_1^\varepsilon}_{t+\delta};
\alpha_2^{\varepsilon}, \beta_2)\\ &\geq&\sum_{n\geq 1}
1_{\Delta_n}J_{N_{t+\delta}^{t,i}}(t+\delta,X^{t,x;\alpha_1,\beta_1^\varepsilon}_{t+\delta};
\alpha_2^\varepsilon,\beta_n^2)-\varepsilon\\
&=&J_{N_{t+\delta}^{t,i}}(t+\delta,X^{t,x;\alpha_1,\beta_1^\varepsilon}_{t+\delta};
\alpha_2^\varepsilon,\beta_2^\varepsilon)-\varepsilon\\
&=&\
^{N_{t+\delta}^{t,i}}Y_{t+\delta}^{t+\delta,X^{t,x;\alpha_1,\beta_1^\varepsilon}_{t+\delta};
\alpha_2^\varepsilon,\beta_2^\varepsilon}-\varepsilon.
\end{eqnarray*}
Let $\alpha^\varepsilon(v_{1}\oplus v_{2}):=\alpha_1(v_{1})\oplus
\alpha_2^\varepsilon(v_{2}), \ v_{1}\in\mathcal {V}_{t,t+\delta},
v_{2}\in\mathcal {V}_{t+\delta,T}.$   Then
$\alpha^\varepsilon(v)|_{[t,t+\delta]}=\alpha(v)|_{[t,t+\delta]}, \
v\in\mathcal {V}_{t,T}.$ Thus,
\begin{eqnarray*}
^{N_{t+\delta}^{t,i}}Y_{t+\delta}^{t+\delta,X^{t,x;\alpha_{1},\beta_{1}^\varepsilon}_{t+\delta};
\alpha_{2}^{\varepsilon},\beta_{2}^\varepsilon}=\
^{i}Y_{t+\delta}^{t,x;\alpha^{\varepsilon},\beta^\varepsilon}.
\end{eqnarray*}
Consequently, by virtue of the Lemmas \ref{l8} and \ref{l1} we
conclude
\begin{eqnarray*}
W_\delta(t,x)\geq I_\delta(t, x, \alpha_1)&\geq &\
^{i}G^{t,x;\alpha_1,\beta_1^{\varepsilon}}_{t,t+\delta}
[W_{N_{t+\delta}^{t,i}}(t+\delta,
X^{t,x;\alpha_1,\beta_1^{\varepsilon}}_{t+\delta})]-\varepsilon
\\&\geq & \ ^{i}G^{t,x;\alpha_1,\beta_1^{\varepsilon}}_{t,t+\delta} [\
^{i}Y_{t+\delta}^{t,x;\alpha^{\varepsilon},\beta^\varepsilon}-\varepsilon]-\varepsilon\\
&\geq &\
^{i}G^{t,x;\alpha^{\varepsilon},\beta^{\varepsilon}}_{t,t+\delta} [\
^{i}Y_{t+\delta}^{t,x;\alpha^{\varepsilon},\beta^\varepsilon}]-C\varepsilon\\&
=&\
^{i}Y_{t}^{t,x;\alpha^\varepsilon,\beta^\varepsilon}-C\varepsilon =\
^{i}Y_{t}^{t,x;\alpha,\beta^\varepsilon}-C\varepsilon,
\end{eqnarray*}
where, for the latter  equality, we have used Lemma \ref{l5}.
Indeed, letting $(u_{2}^{\varepsilon},v_{2}^{\varepsilon})\in
\mathcal {U}_{t+\delta,T}\times \mathcal {V}_{t+\delta,T}$ be
associated with $(\alpha_{2}^{\varepsilon},\beta_{2}^{\varepsilon})$
by Lemma \ref{l5} we have
\begin{eqnarray*}
&&\beta^\varepsilon(u_{1}^{\varepsilon}\oplus
u_{2}^{\varepsilon})=\beta_{1}^\varepsilon(u_{1}^{\varepsilon})\oplus\beta_{2}^\varepsilon(u_{2}^{\varepsilon})
=v_{1}^{\varepsilon}\oplus v_{2}^{\varepsilon},\\
&&\alpha^\varepsilon(v_{1}^{\varepsilon}\oplus
v_{2}^{\varepsilon})=\alpha_{1}^\varepsilon(v_{1}^{\varepsilon})\oplus\alpha^\varepsilon(v_{1}^{\varepsilon}\oplus
v_{2}^{\varepsilon})|_{[t+\delta,T]} =u_{1}^{\varepsilon}\oplus
u_{2}^{\varepsilon},
\end{eqnarray*}
but also $\alpha^\varepsilon(v_{1}^{\varepsilon}\oplus
v_{2}^{\varepsilon})=\alpha(v_{1}^{\varepsilon}\oplus
v_{2}^{\varepsilon}).$  Recalling now the arbitrariness  of
$\alpha\in \mathcal {A}_{t,T}$ we conclude that $ W_\delta(t,x)\geq
 W_{i}(t,x)- C\varepsilon.$
Finally, letting $\varepsilon\downarrow0$ we have $W_\delta(t,x)\geq
W_{i}(t, x).$
\end{proof}
From the dynamic programming principle (Theorem  \ref{t3}) and
standard arguments for BSDEs (see Peng \cite{P1997} or Buckdahn and
Li \cite{BL2006}) we obtain  the following proposition.
\begin{proposition}\label{p6}
 Under the assumptions $(H4)$ and $(H5)$, $W(t,x)$ is
 $\frac{1}{2}-$H\"{o}lder continuous in $t$, i.e., there exists a positive constant $C$ such that,
  for all  $x\in {\mathbb{R}}^n,\ t, t'\in [0, T]$,
  $$
|W_{i}(t, x)-W_{i}(t', x)|\leq C(1+|x|)|t-t'|^{\frac{1}{2}}.
  $$
 \end{proposition}
\subsubsection{Dynamic programming principle for stopping times}
The objective of this subsection is to obtain the dynamic
programming principle for stopping times. For this end, we need the
following proposition, which turns out to be crucial in our
approach. \vskip2mm
\begin{proposition}
  For $0\leq t \leq T$, let $\tau$ be a stopping time such that
 $t\leq \tau \leq T$. Then
 we have
\begin{eqnarray*}
\esssup_{\alpha_{1}\in\mathcal {A}_{\tau,T}}
\essinf_{\beta_{1}\in\mathcal {B}_{\tau,T}}\
^{i}Y^{\tau,x,\alpha_{1},\beta_{1}}_{\tau}=\esssup_{\alpha\in\mathcal
{A}_{t,T}} \essinf_{\beta\in\mathcal {B}_{t,T}}\
^{i}Y^{\tau,x,\alpha,\beta}_{\tau}.
\end{eqnarray*}
 \end{proposition}

\begin{proof} We give the proof in two steps.\vskip2mm

\noindent {\bf Step 1:} For all $\alpha\in\mathcal {A}_{t,T}$ and an
arbitrary fixed $v^{0}\in\mathcal {V}_{t,\tau}$,  let us  define
 $$\alpha_{1}(v^{1}):=\alpha(v^{0}\oplus v^{1})|_{[\tau,T]},
  v^{1}\in\mathcal {V}_{\tau,T},\ \text{and} \
\alpha_{0}(v^{0}):=\alpha(v^{0}\oplus v^{1})|_{[t,\tau]},
  v^{1}\in\mathcal {V}_{\tau,T}.$$
It is straight-forward to check that  $\alpha_{1}\in\mathcal
{A}_{\tau,T}$. Since $\alpha$ is nonanticipative,
$\alpha_{0}(v^{0})$ only depends on $v^{0}$, but not on $v^{1}$. We
put $u^{0}:=\alpha_{0}(v^{0})$ and define, for $ u\in\mathcal
{U}_{t,T}$ and $\beta_{1}\in\mathcal {B}_{\tau,T}$
 $\beta(u):=v^{0}\oplus\beta_{1}(u^{1}), u^{1}=u|_{[\tau,T]}.$
Then $\beta\in\mathcal {B}_{t,T}$.

 Since $\alpha_{1}\in\mathcal {A}_{\tau,T}$ and
$\beta_{1}\in\mathcal {B}_{\tau,T}$, by Lemma \ref{l5}, we have the
existence of a unique couple
$(\widetilde{u}^{1},\widetilde{v}^{1})\in \mathcal {U}_{\tau,T}
\times \mathcal {V}_{\tau,T}$ such that
$\alpha_{1}(\widetilde{v}^{1})=\widetilde{u}^{1},
\beta_{1}(\widetilde{u}^{1})=\widetilde{v}^{1}.$  On the other hand,
for $\alpha\in\mathcal {A}_{t,T}$ and $\beta\in\mathcal {B}_{t,T}$,
there exists a unique couple $(u^{*},v^{*})\in \mathcal {U}_{t,T}
\times \mathcal {V}_{t,T}$ such that $ \alpha(v^{*})=u^{*}, \quad
\beta(u^{*})=v^{*}.$  Consequently,
 $$\alpha(v^{0}\oplus \widetilde{v}^{1})=\alpha_{0}(v^{0})
 \oplus \alpha_{1}(\widetilde{v}^{1})=u_{0}\oplus\widetilde{u}^{1},\
 \beta(u^{0}\oplus \widetilde{u}^{1})=v^{0}
 \oplus \beta_{1}(\widetilde{u}^{1})=v_{0}\oplus\widetilde{v}^{1}.$$
From the uniqueness of $(u^{*},v^{*})$ it follows that $
\widetilde{u}^{1}=u^{*}|_{[\tau,T]},\quad
\widetilde{v}^{1}=v^{*}|_{[\tau,T]}.$  Therefore,
\begin{eqnarray*}
^{i}Y^{\tau,x,\alpha_{1},\beta_{1}}_{\tau} =\
^{i}Y^{\tau,x,\widetilde{u}^{1},\widetilde{u}^{1}}_{\tau} =\
^{i}Y^{\tau,x,u^{*},v^{*}}_{\tau} =\
^{i}Y^{\tau,x,\alpha,\beta}_{\tau}.
\end{eqnarray*}
 For $\alpha\in\mathcal
{A}_{t,T}$,
\begin{eqnarray*}
&&\mathrm{esssup}_{\overline{\alpha}_{1}\in\mathcal {A}_{\tau,T}}
\mathrm{essinf}_{\beta_{1}\in\mathcal {B}_{\tau,T}}\
^{i}Y^{\tau,x,\overline{\alpha}_{1},\beta_{1}}_{\tau}\geq
\mathrm{essinf}_{\beta_{1}\in\mathcal {B}_{\tau,T}}\
^{i}Y^{\tau,x,\alpha_{1},\beta_{1}}_{\tau}\geq
\mathrm{essinf}_{\beta\in\mathcal {B}_{t,T}}\
^{i}Y^{\tau,x,\alpha,\beta}_{\tau}.
\end{eqnarray*}
The latter estimate takes into account that we associated with
$\beta_{1}$ a particular $\beta$,
$\beta(u)=v^{0}\oplus\beta_{1}(u^{1}),\ \ u^{1}=u|_{[\tau,T]},$
independently of $\alpha$.  Consequently,
\begin{eqnarray*}
\mathrm{esssup}_{\alpha_{1}\in\mathcal {A}_{\tau,T}}
\mathrm{essinf}_{\beta_{1}\in\mathcal {B}_{\tau,T}}\
^{i}Y^{\tau,x,\alpha_{1},\beta_{1}}_{\tau}\geq\mathrm{esssup}_{\alpha\in\mathcal
{A}_{t,T}} \mathrm{essinf}_{\beta\in\mathcal {B}_{t,T}}\
^{i}Y^{\tau,x,\alpha,\beta}_{\tau}.
\end{eqnarray*}

\noindent {\bf Step 2:} \ For all $\varepsilon>0$, there exists
$\alpha^{\varepsilon}_{1}\in\mathcal {A}_{\tau,T}$ such that
\begin{eqnarray}\label{e15}
&&\mathrm{esssup}_{\alpha_{1}\in\mathcal {A}_{\tau,T}}
\mathrm{essinf}_{\overline{\beta}_{1}\in\mathcal {B}_{\tau,T}}\
^{i}Y^{\tau,x,\alpha_{1},\overline{\beta}_{1}}\leq
\mathrm{essinf}_{\overline{\beta}_{1}\in\mathcal {B}_{\tau,T}}\
^{i}Y^{\tau,x,\alpha_{1}^{\varepsilon},\overline{\beta}_{1}}+\varepsilon.
\end{eqnarray}

For all $\beta\in\mathcal {B}_{t,T}$ and an arbitrary fixed
$u^{0}\in\mathcal {U}_{t,\tau}$, we define, for $u^{1}\in\mathcal
{U}_{\tau,T}$ and $v\in\mathcal {V}_{t,T}$,
\begin{eqnarray}\label{e35}
\beta_{1}(u^{1}):=\beta(u^{0}\oplus u^{1})|_{[\tau,T]},\ \
\alpha^{\varepsilon}(v):=u^{0}\oplus\alpha^{\varepsilon}_{1}(v^{1}),\
v^{1}=v|_{[\tau,T]}.
\end{eqnarray}
 Then $\alpha^{\varepsilon}\in\mathcal
{A}_{t,T}$ and $\beta_{1}\in\mathcal {B}_{\tau,T}$.

Since  $\alpha^{\varepsilon}\in\mathcal {A}_{t,T}$ and
$\beta\in\mathcal {B}_{t,T}$, by Lemma \ref{l5} we have the
existence of  a unique couple $(u^{*},v^{*})\in \mathcal {U}_{t,T}
\times \mathcal {V}_{t,T}$ such that
$\alpha^{\varepsilon}(v^{*})=u^{*}, \beta(u^{*})=v^{*}.$  On the
other hand, from $\alpha^{\varepsilon}_{1}\in\mathcal {A}_{\tau,T}$
and $\beta_{1}\in\mathcal {B}_{\tau,T}$ it follows that  there
exists a unique couple $(\widetilde{u}^{1},\widetilde{v}^{1})\in
\mathcal {U}_{\tau,T} \times \mathcal {V}_{\tau,T}$ such that$
\alpha^{\varepsilon}_{1}(\widetilde{v}^{1})=\widetilde{u}^{1},
\beta_{1}(\widetilde{u}^{1})=\widetilde{v}^{1}.$   Since $\beta$ is
nonanticipative,  $\beta(u^{0}\oplus \widetilde{u}^{1})|_{[t,\tau]}$
dependents only on $u^{0}$ and not on $\widetilde{u}^{1}$. We  put
$\beta_{0}(\overline{u}^{0})=\beta(\overline{u}^{0}\oplus
\widetilde{u}^{1})|_{[t,\tau]}, \overline{u}^{0}\in \mathcal
{U}_{t,\tau}.$ Consequently, $\beta_{0}\in \mathcal {B}_{t,\tau},$
and
 $$\beta(u^{0}\oplus \widetilde{u}^{1})
 =\beta(u^{0}\oplus \widetilde{u}^{1})|_{[t,\tau]}\oplus\beta_{1}(\widetilde{u}^{1})
 =\beta_{0}(u^{0})\oplus\beta_{1}(\widetilde{u}^{1})=\beta_{0}(u^{0})\oplus\widetilde{v}^{1},$$
 $$\alpha^{\varepsilon}(\beta_{0}(u^{0})\oplus \widetilde{v}^{1})
  =u^{0}\oplus\alpha^{\varepsilon}_{1}(\widetilde{v}^{1})=u^{0}\oplus\widetilde{u}^{1}.$$
Due to the uniqueness of $(u^{*}, v^{*})$ we obtain
 $v^{*}=\beta_{0}(u^{0})\oplus\widetilde{v}^{1},\
u^{*}=u^{0}\oplus\alpha^{\varepsilon}_{1}(\widetilde{v}^{1})=u^{0}\oplus\widetilde{u}^{1}.$
Therefore,
\begin{eqnarray*}
^{i}Y^{\tau,x,\alpha_{1}^{\varepsilon},\beta_{1}}_{\tau} =\
^{i}Y^{\tau,x,\widetilde{u}^{1},\widetilde{v}^{1}}_{\tau} =\
^{i}Y^{\tau,x,u^{*},v^{*}}_{\tau} =\
^{i}Y^{\tau,x,\alpha^{\varepsilon},\beta}_{\tau},
\end{eqnarray*}
from which combined with (\ref{e15}) we get
\begin{eqnarray*}
&&\mathrm{esssup}_{\alpha_{1}\in\mathcal {A}_{\tau,T}}
\mathrm{essinf}_{\overline{\beta}_{1}\in\mathcal {B}_{\tau,T}}\
^{i}Y^{\tau,x,\alpha_{1},\overline{\beta}_{1}}_{\tau}\leq\
\mathrm{essinf}_{\overline{\beta}_{1}\in\mathcal {B}_{\tau,T}}\
^{i}Y^{\tau,x,\alpha_{1}^{\varepsilon},\overline{\beta}_{1}}_{\tau}+\varepsilon\\
&&\leq\mathrm{essinf}_{\beta_\in\mathcal {B}_{t,T}}\
^{i}Y^{\tau,x,\alpha^{\varepsilon},\beta}_{\tau}+\varepsilon \leq
\mathrm{esssup}_{\alpha\in\mathcal {A}_{t,T}}
\mathrm{essinf}_{\beta_\in\mathcal {B}_{t,T}}\
^{i}Y^{\tau,x,\alpha,\beta}_{\tau}+\varepsilon.
\end{eqnarray*}
For the second estimate we used that $\beta_{1}$ in (\ref{e35}) is
defined with the help of $\beta$ and, thus, runs only a subclass of
$\mathcal {B}_{\tau,T}$. The above both steps allow  to conclude the
proof.
\end{proof}

For a stopping time $\tau$ with values in $[t,T]$ we define the
value functions for a game over the stochastic interval
$[[\tau,T]]:$
\begin{eqnarray*}
\overline{W}_{i}(\tau,x):&=&\mathrm{esssup}_{\alpha\in\mathcal
{A}_{\tau,T}} \mathrm{essinf}_{\beta\in\mathcal {B}_{\tau,T}}\
^{i}Y^{\tau,x,\alpha,\beta}_{\tau},\\
 \overline{U}_{i}(\tau,x):&=&
\mathrm{essinf}_{\beta\in\mathcal
{B}_{\tau,T}}\mathrm{esssup}_{\alpha\in\mathcal {A}_{\tau,T}}\
^{i}Y^{\tau,x,\alpha,\beta}_{\tau}.
\end{eqnarray*}

\begin{remark}
Obviously, $\overline{W}_{i}(t,x)=W_{i}(t,x),
\overline{U}_{i}(t,x)=U_{i}(t,x),$ for all $(t,x)\in[0,T]\times
\mathbb{R}^{n},$ and it can be checked in a straight-forward manner
that  for all discrete valued stopping times $\tau\ (0\leq \tau\leq
T)$, $ i=1,2,$
\begin{eqnarray*}
\overline{W}_{i}(\tau,x)&=&W_{i}(\tau,x)(:=W_{i}(t,x)|_{t=\tau}), \\
\overline{U}_{i}(\tau,x)&=&U_{i}(\tau,x)(:=U_{i}(t,x)|_{t=\tau}),
\quad \mathbb{P}-a.s.
\end{eqnarray*}
Our objective is to extend this result to general stopping times.
For this end we have to extend Proposition \ref{p4} and Theorem
\ref{t3}.
\end{remark}
Similar to Proposition \ref{p4} we have from   standard BSDEs
estimates  the following proposition.
\begin{proposition}\label{p5}
There exists a constant $C>0$\ such that, for all stopping time
$\tau$ $ (0 \leq \tau \leq T),\ x, x'\in {\mathbb{R}}^n$, $i=1,2,$
\begin{eqnarray*}
 |\overline{W}_{i}(\tau,x)-\overline{W}_{i}(\tau,x')| \leq
 C|x-x'|, \ \ |\overline{W}_{i}(\tau,x)| \leq C(1+|x|).
\end{eqnarray*}
 The same property holds for $\overline{U}_{i}$.
\end{proposition}

\begin{theorem}\label{t4}
Let the assumptions $(H4)$ and $(H5)$ hold. Then the following
dynamic programming principles hold: For any stopping times
$\tau,\eta$ with $0\leq t<\tau \leq \eta\leq T,\ x\in
{\mathbb{R}}^n, i=1,2,$
\begin{eqnarray}\label{et4}
 \overline{W}_{i}(\tau,x) & =&\esssup_{\alpha
\in {\mathcal{A}}_{\tau, \eta}}\essinf_{\beta \in
{\mathcal{B}}_{\tau, \eta}}\ ^{i}G^{t,x;\alpha,\beta}_{\tau, \eta}
[\overline{W}_{N_{\eta}^{\tau,i}}(\eta,
X^{\tau,x;\alpha,\beta}_{\eta})],\\
 \overline{U}_{i}(\tau,x) & =&\essinf_{\beta \in
{\mathcal{B}}_{\tau, \eta}}\esssup_{\alpha \in {\mathcal{A}}_{\tau,
\eta}}\ ^{i}G^{t,x;\alpha,\beta}_{\tau, \eta}
[\overline{U}_{N_{\eta}^{\tau,i}}(\eta,
X^{\tau,x;\alpha,\beta}_{\eta})].\nonumber
\end{eqnarray}
\end{theorem}
\begin{proof}
We only comment  the proof of the first relation. Let us denote by
$\widehat{W}_{i}(\tau,x)$ the right hand side of $\ref{et4}$.   We
prove in a straight-forward way that $
\overline{W}_{i}(\tau,x)=\widehat{W}_{i}(\tau,x)$. For this we adapt
the proof of the Lemmas \ref{l9} and \ref{l6} in an obvious manner.
\end{proof}

From Theorem \ref{t4} and Proposition \ref{p5} we can deduce the
following proposition.
\begin{proposition}\label{p7}
Let the assumptions $(H4)$ and  $(H5)$ be satisfied. Then, for any
stopping time $\tau,\eta$ with $0\leq t<\tau \leq \eta\leq T,\ x\in
{\mathbb{R}}^n$, where  $\eta$ is supposed to be
$\sigma(\tau)-$measurable, we have the following:
\begin{eqnarray*}
|\overline{W}_{i}(\tau,x)-\overline{W}_{i}(\eta,x)|\leq
C(1+|x|)|\tau-\eta|^{\frac{1}{2}}.
\end{eqnarray*}
\end{proposition}

\begin{proof}
We only prove that
$\overline{W}_{i}(\tau,x)-\overline{W}_{i}(\eta,x)\leq
C(1+|x|)\tau-\eta|^{\frac{1}{2}},$  since  the other inequality can
be proved in a similar manner. Let $\varepsilon>0$. In analogy to
(\ref{e16}), but now for $(\tau,\eta)$ instead of $(t,t+\delta)$, we
can show that there exists some $\alpha^{\varepsilon} \in
{\mathcal{A}}_{\tau, \eta}$ such that,  for any $\beta \in
{\mathcal{B}}_{\tau, \eta}$,
\begin{eqnarray*}
\overline{W}_{i}(\tau,x)\leq\
^{i}G^{\tau,x;\alpha^{\varepsilon},\beta}_{\tau, \eta}
[\overline{W}_{N_{\eta}^{\tau,i}}(\eta,
X^{t,x;\alpha^{\varepsilon},\beta}_{\eta})]+\varepsilon.
\end{eqnarray*}
Consequently, $\overline{W}_{i}(\tau,x)-\overline{W}_{i}(\eta,x)\leq
I_{1}+I_{2}+I_{3} +\varepsilon,$  where
\begin{eqnarray*}
I_{1} & := & ^{i}G^{\tau,x;\alpha^{\varepsilon},\beta}_{\tau, \eta}
[\overline{W}_{N_{\eta}^{\tau,i}}(\eta,
X^{t,x;\alpha^{\varepsilon},\beta}_{\eta})] -\
^{i}G^{\tau,x;\alpha^{\varepsilon},\beta}_{\tau, \eta}
[\overline{W}_{N_{\eta}^{\tau,i}}(\eta,x)], \\
I_{2}& := & ^{i}G^{\tau,x;\alpha^{\varepsilon},\beta}_{\tau, \eta}
[\overline{W}_{N_{\eta}^{\tau,i}}(\eta,x)]-\
^{i}G^{\tau,x;\alpha^{\varepsilon},\beta}_{\tau, \eta}
[\overline{W}_{i}(\eta,x)],\\
I_{3} & := &\ ^{i}G^{\tau,x;\alpha^{\varepsilon},\beta}_{\tau, \eta}
[\overline{W}_{i}(\eta,x)]-\overline{W}_{i}(\eta,x).
\end{eqnarray*}
Thanks to Lemma \ref{l5}, we let  $(u,v)\in\mathcal
{U}_{\tau,\eta}\times\mathcal {V}_{\tau,\eta}$ be such that
$\alpha^{\varepsilon}(v)=u, \beta(u)=v,$ on $[[\tau,\eta]].$

 By virtue of Proposition
\ref{p5} there exists some positive constant $C$ which does not
depend on  $\alpha^{\varepsilon}$ and $ \beta$ such that
\begin{eqnarray*}
|I_{1} | &\leq& C(\mathbb{E}[|\overline{W}_{N_{\eta}^{\tau,i}}(\eta,
X^{t,x;u,v}_{\eta}) -\ \overline{W}_{N_{\eta}^{\tau,i}}(\eta,x)|^2|{{\mathcal{F}}_\tau})])^{\frac{1}{2}}\\
&\leq& C(\mathbb{E}[| X^{\tau,x; u,v}_{\eta}
-x|^2|{{\mathcal{F}}_\tau})])^{\frac{1}{2}} \leq
C(1+|x|^2)|\tau-\eta|^{\frac{1}{2}},
\end{eqnarray*}
where the latter relation follows from  standard SDE estimates.  Let
us denote  by $(^{i}Y,\ ^{i}Z,\ ^{i}H)$ the solution of the BSDE:
 \begin{eqnarray*}
   \left \{\begin{array}{rcl}
   -d\ ^{i}Y_s & = & f_{N_{s}^{t,i}}(s,X^{t,x; u, v}_s, \ ^{i}Y_s,
   \ ^{i}H_s,\  ^{i}Z_s,u_s, v_s) ds\\
   &&-\lambda\ ^{i}H_s ds-\ ^{i}Z_s dB_s-\ ^{i}H_s d\widetilde{N}_{s},\ s\in[\tau,\eta],\\
        ^{i}Y_\eta  & = & \overline{W}_{i}(\eta,x).
   \end{array}\right.
  \end{eqnarray*}
From the definition of our backward stochastic semigroup and the
$\sigma(\tau)-$measurability  of $\eta$  it follows that
\begin{eqnarray*}
I_{3}& = & ^{i}G^{t,x;\alpha^{\varepsilon},\beta}_{\tau, \eta}
[\overline{W}_{i}(\eta,x)]-\overline{W}_{i}(\eta,x)\\
&=&\mathbb{E}[\overline{W}_{i}(\eta,x)+\int^{\eta}_\tau
f_{N_{s}^{\tau,i}}(s,X^{\tau,x; u, v}_s,
   \ ^{i}Y_s,   \ ^{i}H_s,\ ^{i}Z_s,u_s, v_s) ds   \\
 & &  -\lambda \int^{\eta}_\tau\ ^{i}H_s ds
 -\int^{\eta}_\tau\ ^{i}Z_s dB_s -\int^{\eta}_\tau\ ^{i}H_s d \widetilde{N}_s\ \Big|{{\mathcal{F}}_\tau}]
- \overline{W}_{i}(\eta,x) \\
 &=&\mathbb{E}[\int^{\eta}_\tau f_{N_{s}^{\tau,i}}(s,X^{\tau,x; u,
v}_s,   \ ^{i}Y_s,   \ ^{i}H_s,\ ^{i}Z_s,u_s, v_s) ds -\lambda
\int^{\eta}_\tau\ ^{i}H_s
 ds\ \Big|{{\mathcal{F}}_\tau}].
\end{eqnarray*}
Therefore, the Schwartz inequality and  the Appendix of
\cite{BHL2010} yield
\begin{eqnarray*}
|I_{3} | &\leq&\mathbb{E}[\int^{\eta}_\tau
\Big(|f_{N_{s}^{\tau,i}}(s,X^{\tau,x; u, v}_s,
   \ ^{i}Y_s,   \ ^{i}H_s,\   ^{i}Z_s,u_s,  v_s)|
   +\lambda|\ ^{i}H_s|\Big) ds\Big|{{\mathcal{F}}_\tau}]\\
 &\leq& |\tau-\eta|^{\frac{1}{2}}\mathbb{E}[\int^{\eta}_{\tau}
\Big(|f_{N_{s}^{t,i}}(s,X^{t,x; u, v}_s,
   \ ^{i}Y_s,   \ ^{i}H_s,\   ^{i}Z_s,u_s,  v_s)|
   +\lambda|\ ^{i}H_s|\Big)^{2} ds\Big|{{\mathcal{F}}_\tau}]^{\frac{1}{2}}\\
&\leq& C |\tau-\eta|^{\frac{1}{2}}\mathbb{E}[\int^{\eta}_\tau
\Big(|f_{N_{s}^{\tau,i}}(s,X^{\tau,x; u, v}_s, 0, 0, 0,u_s,
   v_s)|^{2} +|\ ^{i}Y_s|^{2}+
 |\ ^{i}H_s|^{2} +|\ ^{i}Z_s|^{2}\Big)
 ds\Big|{{\mathcal{F}}_\tau}]^{\frac{1}{2}}\\
&\leq& C |\tau-\eta|^{\frac{1}{2}}\mathbb{E}[\int^{\eta}_\tau
\Big(1+|X^{\tau,x; u, v}_s|^{2} +|\ ^{i}Y_s|^{2}+|\ ^{i}H_s|^{2} +|\
^{i}Z_s|^{2}\Big) ds\Big|{{\mathcal{F}}_\tau}]^{\frac{1}{2}}\\
 & \leq & C (1+|x|) |\tau-\eta|^{\frac{1}{2}}.
\end{eqnarray*}
Finally, let us give the estimate of $I_{2}$. From Lemma \ref{l1} we
have
\begin{eqnarray*}
|I_{2}| &\leq &
C\mathbb{E}[|\overline{W}_{N_{\eta}^{\tau,i}}(\eta,x)-\overline{W}_{i}(\eta,x)||{{\mathcal{F}}_\tau}]
=C\mathbb{E}[|\overline{W}_{N_{\eta}^{\tau,i}}(\eta,x)-\overline{W}_{i}(\eta,x)|1_{\{N_{\eta}^{\tau,i}\neq
i\}}|{{\mathcal{F}}_\tau}]\\
& \leq & C(1+|x|)\mathbb{E}[1_{\{N_{\eta}^{\tau,i}\neq
i\}}|{{\mathcal{F}}_\tau}] =
C(1+|x|)\mathbb{P}[\{N_{\eta}^{\tau,i}\neq
i\}|{{\mathcal{F}}_\tau}]\\
& \leq & C(1+|x|)(1-\exp(-\lambda(\eta-\tau)))
 \leq  C(1+|x|)  |\tau-\eta|.
\end{eqnarray*}
Consequently, from the above inequalities we deduce that
\begin{eqnarray*}
\overline{W}_{i}(\tau,x)-\overline{W}_{i}(\eta,x)\leq
C(1+|x|)|\tau-\eta|^{\frac{1}{2}}+\varepsilon,\ \varepsilon>0,
\end{eqnarray*}
  from where we conclude
\begin{eqnarray*}
\overline{W}_{i}(\tau,x)-\overline{W}_{i}(\eta,x)\leq
C(1+|x|)|\tau-\eta|^{\frac{1}{2}}.
\end{eqnarray*}
The desired result then follows.
\end{proof}

\begin{proposition}\label{p8}
Under the assumptions $(H4)$ and  $(H5)$, the following holds:
$\overline{W}_{i}(\tau,x)=W_{i}(\tau,x),\ i.e.,$
\begin{eqnarray*}
W_{i}(\tau,x)=\mathrm{esssup}_{\alpha\in\mathcal {A}_{\tau,T}}
\mathrm{essinf}_{\beta\in\mathcal {B}_{\tau,T}}\
^{i}Y^{\tau,x,\alpha,\beta}_{\tau},
\end{eqnarray*}
for all stopping time $\tau$ with values in $[t,T]$.
\end{proposition}

\begin{proof}
Let $t\in [0,T]$ be such that $t\leq \tau\leq T.$  Let us define,
for $i=1, \cdots,2^{n}$,
\begin{eqnarray*}
t_{i}=\frac{i(T-t)}{2^{n}}+t,\quad  A_{i}=\Big\{t_{i-1} < \tau \leq
t_{i}\Big\}, \quad \text{and}\  A_{0}=\{ \tau = t\},\quad  \tau_{n}
=\sum\limits_{i=0}^{2^{n}}t_{i} 1_{A_{i}}.
\end{eqnarray*}
Obviously, $0\leq\tau_{n}-\tau\leq \dfrac{1}{2^{n}}$, and from the
definition of $W_{i}$ as well as that of the essential infimun and
supremun of a family of random variables we deduce
\begin{eqnarray*}
W_{i}(\tau_{n},x)(:=W_{i}(s,x)|_{s=\tau_{n}})=\mathrm{esssup}_{\alpha\in\mathcal
{A}_{t,T}} \mathrm{essinf}_{\beta\in\mathcal {B}_{t,T}}\
^{i}Y^{\tau_{n},x,\alpha,\beta}_{\tau_{n}}.
\end{eqnarray*}
Therefore, $W_{i}(\tau_{n},x)=\overline{W}_{i}(\tau_{n},x), \
\mathbb{P}-a.s.$  Since $\tau\leq\tau_{n}\leq T$,  and $\tau_{n}$ is
$\sigma(\tau)-$measurable, it follows from the Propositions \ref{p6}
and \ref{p7}  that
\begin{eqnarray*}
|\overline{W}_{i}(\tau,x)-\overline{W}_{i}(\tau_{n},x)|&\leq&
C(1+|x|)|\tau-\tau_{n}|^{\frac{1}{2}}\rightarrow 0,\ \mbox{as}\
n\rightarrow\infty,
\end{eqnarray*}
and
\begin{eqnarray*}
 |W_{i}(\tau,x)-W_{i}(\tau_{n},x)|&\leq&
C(1+|x|)|\tau-\tau_{n}|^{\frac{1}{2}}\rightarrow 0, \ \mbox{as}\
n\rightarrow\infty,
\end{eqnarray*}
from where we conclude that also
$W_{i}(\tau,x)=\overline{W}_{i}(\tau,x).$  Therefore,
\begin{eqnarray*}
W_{i}(\tau,x)=\mathrm{esssup}_{\alpha\in\mathcal {A}_{\tau,T}}
\mathrm{essinf}_{\beta\in\mathcal {B}_{\tau,T}}\
^{i}Y^{\tau,x,\alpha,\beta}_{\tau}.
\end{eqnarray*}
The proof is complete.
\end{proof}\vskip2mm

From the above proposition and Proposition \ref{p7} we immediately
have:
\begin{proposition}\label{p21}
Under  the assumptions $(H4)$ and  $(H5)$, there is some positive
constant $C$ such that, for any stopping times $\tau,\eta$ with
$0\leq t<\tau \leq \eta\leq T,\ x\in {\mathbb{R}}^n$, where  $\eta$
is supposed to be $\sigma(\tau)-$measurable, we have the following:
\begin{eqnarray*}
|W_{i}(\tau,x)-W_{i}(\eta,x)|\leq C(1+|x|)|\tau-\eta|^{\frac{1}{2}}.
\end{eqnarray*}
\end{proposition}
Theorem \ref{t5} follows from  Proposition \ref{p8} and Theorem
\ref{t4}.  We also have the following statement, which is a direct
consequence of Proposition \ref{p8}.
\begin{proposition}\label{p9}
Under the assumptions $(H4)$ and  $(H5)$, the following holds:
\begin{eqnarray*}
W_{N^{t,i}_{\tau}}(\tau,x)=\mathrm{esssup}_{\alpha\in\mathcal
{A}_{\tau,T}} \mathrm{essinf}_{\beta\in\mathcal {B}_{\tau,T}}\
^{N^{t,i}_{\tau}}Y^{\tau,x,\alpha,\beta}_{\tau}.
\end{eqnarray*}
\end{proposition}

\subsection{Proof of Theorem \ref{t1}}\label{A2}

\begin{proof}
We only give the proof for $U=(U_{1},U_{2})$, that for
$W=(W_{1},W_{2})$ uses a similar argument.  Let $i=1,2$ be
arbitrarily fixed, $(t,x) \in [0,T) \times {\mathbb{R}}^n,$ and
$\delta>0$. We put
   $\tau^{\delta}=\mbox{inf}\{s\geq t, N_{s}^{t,i}\neq
   i\}\wedge(t+\delta).$

 For  $\varphi \in C^3_{l, b} ([0,T] \times
{\mathbb{R}}^n)$, we define
\begin{eqnarray*}
    \  F(s,x,y,h, z, u, v) &=& \dfrac{\partial }{\partial s}\varphi (s,x) +
     \dfrac{1}{2}tr(\sigma\sigma^{T}(s,
x, u, v) D^{2}\varphi (s,x))+ D\varphi (s,x)b(s, x, u, v)\\
        &&+ f_{i}(s, x, y+\varphi (s,x), h+U_{m(i+1)}(s,x)-\varphi (s,x),
         z+ D\varphi (s,x).\sigma(s,x,u, v), u, v),
\end{eqnarray*}
where $(s,x,y,h,z,u, v)\in [0,T] \times {\mathbb{R}}^n \times
{\mathbb{R}}\times {\mathbb{R}} \times {\mathbb{R}}^d \times U
\times V.$

Let us consider the following  BSDE on the interval
$[t,\tau^{\delta}]:$
\begin{eqnarray}\label{e36}
    \left \{\begin{array}{rl}
      -dY^{1,u,v}_s =&\!\!\!\! F(s,X^{t,x;u,v}_s, Y^{1,u,v}_s, H^{1,u,v}_s, Z^{1,u,v}_s, u_s,v_s)ds
                   -\lambda H^{1,u,v}_s ds\\&\hskip2cm -Z^{1,u,v}_s dB_s-H^{1,u,v}_s d\widetilde{N}_{s}, \\
     Y^{1,u,v}_{\tau^{\delta}}=&\!\!\!\!0,\qquad \qquad \qquad  s\in [t,\tau^{\delta}],
     \end{array}\right.
 \end{eqnarray}
     where $X^{t,x;u,v}$\ is the solution of (\ref{e3}) with $\zeta=x\in \mathbb{R}^{n}$.

We notice that  $F(s,x, y,h,z,u,v)$ is Lipschitz in   $(y,h,z)$,
uniformly in   $(s,x, u,v)$, and there exists a positive constant
$C$ such that
\begin{eqnarray*}
    |F(s,x, 0,0,0,u,v)|\leq C(1+|x|^{2}),\ (s,x,u, v)\in [0,T] \times {\mathbb{R}}^n
 \times U \times V.
 \end{eqnarray*}
Consequently, BSDE (\ref{e36}) has a unique solution $(Y^{1,u,v},
H^{1,u,v}, Z^{1,u,v})$. This solution obviously depends on $\delta$,
but for the sake of simplifying the notations we don't add $\delta$
as superscript to $(Y^{1,u,v}, H^{1,u,v}, Z^{1,u,v})$.

 For the proof of Theorem \ref{t1}, which extends the approaches in
 \cite{BCQ2011} and \cite{BH2010} to SDGs with jumps,
  we admit the following  lemmas for the moment, they will be proven after.
\begin{lemma}\label{l2}
     For every $s\in [t,\tau^{\delta}]$, the following holds:
\begin{eqnarray*}
Y^{1,u,v}_{s\wedge \tau^{\delta}} =\
^{i}G^{t,x;u,v}_{s\wedge\tau^{\delta},\tau^{\delta}}\Big [\varphi
(\tau^{\delta},X^{t,x;u,v}_{\tau^{\delta}})1_{N_{\tau^{\delta}}^{t,i}=i}+
U_{m(i+1)}(\tau^{\delta},X^{t,x;u,v}_{\tau^{\delta}})1_{N_{\tau^{\delta}}^{t,i}=m(i+1)}\Big]
-\overline{Y}_{s\wedge \tau^{\delta}},
\end{eqnarray*}
where
\begin{eqnarray*}
\overline{Y}_{s}=\varphi
(s,X^{t,x;u,v}_s)+\int_{t}^{s}\Big(U_{m(i+1)}(r,X^{t,x;u,v}_{r})-\varphi
(r,X^{t,x;u,v}_r)\Big)dN_{r}.
\end{eqnarray*}
In particular,
\begin{eqnarray*}
Y^{1,u,v}_{t} =\ ^{i}G^{t,x;u,v}_{t,\tau^{\delta}} \Big[\varphi
(\tau^{\delta},X^{t,x;u,v}_{\tau^{\delta}})1_{N_{\tau^{\delta}}^{t,i}=i}+
U_{m(i+1)}(\tau^{\delta},X^{t,x;u,v}_{\tau^{\delta}})1_{N_{\tau^{\delta}}^{t,i}=m(i+1)}\Big]-\varphi(t,x).
\end{eqnarray*}
(Recall that $m(j)=1,$ if $j$ is odd, and $m(j)=2,$ if $j$ is even.)
\end{lemma}

We also consider the following BSDE where,  in equation (\ref{e36}),
$X^{t,x;u,v}$ is substituted by its deterministic initial value $x$:
 \begin{eqnarray}\label{e9}
    \left \{\begin{array}{rl}
      -dY^{2,u,v}_s =&\!\!\!\! F(s,x, Y^{2,u,v}_s, H^{2,u,v}_s, Z^{2,u,v}_s, u_s,v_s)ds
                   -\lambda H^{2,u,v}_s ds\\& \hskip2cm -Z^{2,u,v}_s dB_s-H^{2,u,v}_s d\widetilde{N}_{s}, \\
     Y^{2,u,v}_{\tau^{\delta}}=&\!\!\!\!0,\qquad \qquad  s\in
     [t,\tau^{\delta}].
     \end{array}\right.
 \end{eqnarray}
Then we have the following comparison between $Y^{1,u,v}$ and
$Y^{2,u,v}$.
\begin{lemma}\label{l3}
For every $\delta \in (0,1)$ and $u \in {\mathcal{U}}_{t, t+\delta},
v \in {\mathcal{V}}_{t, t+\delta},$
 \begin{eqnarray*}
 |Y^{1,u,v}_t-Y^{2,u,v}_t| \leq C\delta^{\frac{3}{2}},\ \  \mathbb{P}-a.s.,
 \end{eqnarray*}
 where the constant $C$ does not depend on the control processes  $u$\ and $v$, neither on $\delta>0$.
\end{lemma}
\noindent Moreover, we have
\begin{lemma}\label{l4}
For every $u \in {\mathcal{U}}_{t, t+\delta}, v \in
{\mathcal{V}}_{t, t+\delta},$\ we have
 \begin{eqnarray*}
\mathbb{E}[\int_t^{t+\delta}|Y^{2,u,v}_s|ds|{\mathcal{F}}_{t}]
+\mathbb{E}[\int_t^{t+\delta}|Z^{2,u,v}_s|ds|{\mathcal{F}}_{t}]
+\mathbb{E}[\int_t^{t+\delta}|H^{2,u,v}_s|ds|{\mathcal{F}}_{t}] \leq
C\delta^{\frac{3}{2}},\ \mathbb{P}-a.s.,
\end{eqnarray*}
  where the
constant $C$\ is independent of\ $t,\ \delta$\ as well as of  the
control processes $u,\ v$.
\end{lemma}
\noindent Let us prove that $U=(U_{1},U_{2})$ is a viscosity
solution of the system (\ref{e6}).  We begin with showing that\\
 a) {\it $U=(U_{1}, U_{2})$ is a viscosity
subsolution of the system (\ref{e6}).}\\
 Let $i=1,2$ be
arbitrarily fixed, and we suppose that $U_{i}\leq\varphi $ and
$U_{i}(t,x)=\varphi(t,x) $. We claim that
\begin{eqnarray*}
\inf_{v \in V}\sup_{u\in U}F(t,x,0,0,0,u, v)\geq 0.
\end{eqnarray*}
Observe that, this claim just means that $U=(U_{1},U_{2})$ satisfies
(\ref{v1}) of Definition \ref{d3}; we will come back to this point.
We make the proof by contradiction, and suppose that the above claim
is not true. Then, thanks to the continuity of $F$,   there exist
some $\theta>0, v^{*}\in V$ and $0< \delta'\leq T-t$ such that
\begin{eqnarray}\label{e10}
 \sup_{u\in U}F(s,x,0,0,0,u,v^{*})\leq-\theta<0,
\end{eqnarray}
for all $s\in [t,t+\delta']$.  Since  $U_{i}\leq\varphi $ and
$U_{i}(t,x)=\varphi(t,x) $,  we have
 \begin{eqnarray*}
&&\varphi(\tau^{\delta},X^{t,x;\alpha,\beta}_{\tau^{\delta}})1_{N_{\tau^{\delta}}^{t,i}=i}+
U_{m(i+1)}(\tau^{\delta},X^{t,x;\alpha,\beta}_{\tau^{\delta}})1_{N_{\tau^{\delta}}^{t,i}=m(i+1)}\\
&&\geq
U_{i}(\tau^{\delta},X^{t,x;\alpha,\beta}_{\tau^{\delta}})1_{N_{\tau^{\delta}}^{t,i}=i}+
U_{m(i+1)}(\tau^{\delta},X^{t,x;\alpha,\beta}_{\tau^{\delta}})1_{N_{\tau^{\delta}}^{t,i}=m(i+1)}\\
&&=U_{N_{\tau^{\delta}}^{t,i}}(\tau^{\delta},X^{t,x;\alpha,\beta}_{\tau^{\delta}}).
\end{eqnarray*}
From the Lemmas \ref{l8} and \ref{l2} it follows that
  \begin{eqnarray*}
Y^{1,\alpha,\beta}_{t} &=& \
^{i}G^{t,x;\alpha,\beta}_{t,\tau^{\delta}} \Big[\varphi
(\tau^{\delta},X^{t,x;\alpha,\beta}_{\tau^{\delta}})1_{N_{\tau^{\delta}}^{t,i}=i}+
U_{m(i+1)}(\tau^{\delta},X^{t,x;\alpha,\beta}_{\tau^{\delta}})1_{N_{\tau^{\delta}}^{t,i}=m(i+1)}\Big]-\varphi(t,x)\\
&\geq&\ ^{i}G^{t,x;\alpha,\beta}_{t,\tau^{\delta}}
[U_{N_{\tau^{\delta}}^{t,i}}(\tau^{\delta},X^{t,x;\alpha,\beta}_{\tau^{\delta}})]-\varphi(t,x).
\end{eqnarray*}
On the other hand, by virtue of Theorem \ref{t5}  we have
\begin{eqnarray*}
     \essinf_{\beta \in {\mathcal{B}}_{t, \tau^{\delta}}}\esssup_{u \in
{\mathcal{U}}_{t, \tau^{\delta}}} \
^{i}G^{t,x;\alpha,\beta}_{t,\tau^{\delta}}
[U_{N_{\tau^{\delta}}^{t,i}}(\tau^{\delta},X^{t,x;\alpha,\beta}_{\tau^{\delta}})]-\varphi(t,x)
=U_{i}(t,x)-\varphi(t,x)=0.
\end{eqnarray*}
Therefore,
    $ \essinf_{\beta \in {\mathcal{B}}_{t, \tau^{\delta}}}\esssup_{\alpha \in
{\mathcal{A}}_{t, \tau^{\delta}}}Y^{1,\alpha,\beta}_{t}\geq 0.$
Putting  $\beta^{*}(u)=v^{*}$ on $ [t,\tau^{\delta}]\times \mathcal
{U}_{t,\tau^{\delta}}$, we have  $\beta^{*}\in \mathcal
{B}_{t,\tau^{\delta}}$. From Lemma \ref{l3} it follows that
 \begin{eqnarray*}
 0&\leq&\essinf_{\beta \in {\mathcal{B}}_{t, \tau^{\delta}}}\esssup_{\alpha \in
{\mathcal{A}}_{t, \tau^{\delta}}}Y^{1,\alpha,\beta}_{t}
\leq\esssup_{\alpha \in
{\mathcal{A}}_{t, \tau^{\delta}}}Y^{1,\alpha,\beta^{*}}_{t}\\
&\leq&\esssup_{\alpha \in {\mathcal{A}}_{t,
\tau^{\delta}}}Y^{2,\alpha,\beta^{*}}_{t}+C\delta^{\frac{3}{2}}
=\esssup_{\alpha \in
{\mathcal{A}}_{t, \tau^{\delta}}}Y^{2,\alpha(v^{*}),v^{*}}_{t}+C\delta^{\frac{3}{2}}\\
&\leq&\esssup_{u \in {\mathcal{U}}_{t,
\tau^{\delta}}}Y^{2,u,v^{*}}_{t}+C\delta^{\frac{3}{2}}.
\end{eqnarray*}
Here we have used that
$Y^{2,\alpha,\beta^{*}}_{t}=Y^{2,\alpha(v^{*}),v^{*}}_{t}$, since
$\beta^{*}(\alpha(v^{*}))=v^{*}$. Furthermore, for all $\delta>0,
\varepsilon>0$, there exists $u^{\varepsilon} \in {\mathcal{U}}_{t,
\tau^{\delta}}$ such that
 \begin{eqnarray*}
Y^{2,u^{\varepsilon},v^{*}}_{t} \geq \esssup_{u \in
{\mathcal{U}}_{t,
\tau^{\delta}}}Y^{2,u,v^{*}}_{t}-\varepsilon\delta.
\end{eqnarray*}
The two above inequalities yield
\begin{eqnarray}\label{e39}
Y^{2,u^{\varepsilon},v^{*}}_{t} \geq
-C\delta^{\frac{3}{2}}-\varepsilon\delta.
\end{eqnarray}
Since
\begin{eqnarray*}
Y^{2,u^\varepsilon,v^*}_t
  =\mathbb{E}[\int_t^{\tau^{\delta}}F(s, x, Y^{2,u^\varepsilon,v^*}_s,H^{2,u^\varepsilon,v^*}_s,
  Z^{2,u^\varepsilon,v^*}_s,u^\varepsilon_s,v^*_s)ds\Big|{\mathcal{F}}_{t}],
\end{eqnarray*}
we can deduce  from the Lipschitz property of $F$ in $(y, h,z)$ that
\begin{eqnarray}\label{e11}
Y^{2,u^\varepsilon,v^*}_t
  &=&\mathbb{E}[\int_t^{\tau^{\delta}}\Big(F(s, x, Y^{2,u^\varepsilon,v^*}_s,H^{2,u^\varepsilon,v^*}_s,
  Z^{2,u^\varepsilon,v^*}_s,u^\varepsilon_s,v^*_s)-\lambda H^{2,u^\varepsilon,v^*}_s\Big)ds\Big|{\mathcal{F}}_{t}]\nonumber\\
 &\leq&\mathbb{E}[\int_t^{\tau^{\delta}}F(s, x, 0,0,0,u^\varepsilon_s,v^*_s)ds\Big|{\mathcal{F}}_{t}]\nonumber\\
  && +C\mathbb{E}[\int_t^{\tau^{\delta}}\Big(
  |Y^{2,u^\varepsilon,v^*}_s|+|H^{2,u^\varepsilon,v^*}_s|+
  |Z^{2,u^\varepsilon,v^*}_s|\Big)ds\Big|{\mathcal{F}}_{t}].
\end{eqnarray}
We notice that
\begin{eqnarray*}
&&\mathbb{E}[t+\delta-\tau^{\delta}|{\mathcal{F}}_{t}]\leq
\delta\mathbb{E}[1_{t+\delta>\tau^{\delta}}|{\mathcal{F}}_{t}]
=\delta\mathbb{P}[t+\delta>\tau^{\delta}|{\mathcal{F}}_{t}]\\
&&=\delta(1-\exp(-\lambda\delta)) \leq\lambda\delta^{2}.
\end{eqnarray*}
Choosing  $\delta<\delta'$ we obtain  from (\ref{e10})
\begin{eqnarray*}
&&\mathbb{E}[\int_t^{\tau^{\delta}}F(s, x,
0,0,0,u^\varepsilon_s,v^*_s)ds|{\mathcal{F}}_{t}]\leq
   -\theta\mathbb{E}[\tau^{\delta}-t|{\mathcal{F}}_{t}]\\
&&\leq
\theta(\mathbb{E}[t+\delta-\tau^{\delta}|{\mathcal{F}}_{t}]-\delta)
\leq \theta(\lambda\delta^{2}-\delta).
\end{eqnarray*}
This together with (\ref{e39}), (\ref{e11}) and Lemma \ref{l4}
yields $-C\delta^{\frac{3}{2}}- \varepsilon\delta\leq
\lambda\delta^{2}\theta-\theta\delta.$ Therefore,
$-C\delta^{\frac{1}{2}}- \varepsilon\leq
\lambda\delta\theta-\theta.$
 Letting $\delta\downarrow0$,\ and
$\varepsilon\downarrow0$  we deduce that $\theta\leq0$, which is in
contradiction with $\theta>0$. Consequently,
\begin{eqnarray*}
\inf_{v \in V}\sup_{u\in U}F(t,x,0,0,0,u, v)\geq 0.
\end{eqnarray*}
Finally, taking into account the definition of $F$, we see that
\begin{eqnarray*}
\inf_{v \in V}\sup_{u\in U}F(t,x,0,0,0,u, v)&=& \dfrac{\partial
}{\partial t} \varphi(t,x)+ \inf_{v \in V}\sup_{u\in
U}\Big\{\dfrac{1}{2}tr(\sigma\sigma^{T}(t,
x, u, v) D^{2}\varphi (t,x))+ D\varphi (t,x)b(t, x, u, v)\\
  &&\quad + f_{i}(t, x, \varphi (t,x), U_{m(i+1)}(t,x)-\varphi (t,x),  D\varphi (t,x)\sigma(t, x, u, v), u,
  v)\Big\}\\
  &=& \dfrac{\partial
}{\partial t} \varphi(t,x)+ \inf_{v \in V}\sup_{u\in
U}\Big\{\dfrac{1}{2}tr(\sigma\sigma^{T}(t,
x, u, v) D^{2}\varphi (t,x))+ D\varphi (t,x)b(t, x, u, v)\\
  &&\quad + \widetilde{f}_{i}(t, x, U_{1}(t,x), U_{2}(t,x),  D\varphi (t,x)\sigma(t, x, u, v), u,
  v)\Big\},
\end{eqnarray*}
from where it follows  that $ U=(U_{1}, U_{2})$ is a viscosity
subsolution of the system (\ref{e6}).\vskip 1mm

\noindent b)\ Let us show that  {\it $ U=(U_{1}, U_{2})$ is also a
viscosity supersolution of the system (\ref{e6}).}\vskip 1mm
\noindent Let $i=1,2$ be arbitrarily fixed, and we suppose that the
test function $\varphi$ is such that $U_{i}\geq\varphi $ and
$U_{i}(t,x)=\varphi(t,x).$ Then
 \begin{eqnarray*}
&&\varphi(\tau^{\delta},X^{t,x;\alpha,\beta}_{\tau^{\delta}})1_{N_{\tau^{\delta}}^{t,i}=i}+
U_{m(i+1)}(\tau^{\delta},X^{t,x;\alpha,\beta}_{\tau^{\delta}})1_{N_{\tau^{\delta}}^{t,i}=m(i+1)}\\
&&\leq
U_{i}(\tau^{\delta},X^{t,x;\alpha,\beta}_{\tau^{\delta}})1_{N_{\tau^{\delta}}^{t,i}=i}+
U_{m(i+1)}(\tau^{\delta},X^{t,x;\alpha,\beta}_{\tau^{\delta}})1_{N_{\tau^{\delta}}^{t,i}=m(i+1)}\\
&&=U_{N_{\tau^{\delta}}^{t,i}}(\tau^{\delta},X^{t,x;\alpha,\beta}_{\tau^{\delta}}),
\end{eqnarray*}
and thanks to Lemma \ref{l2}  we have
  \begin{eqnarray*}
Y^{1,\alpha,\beta}_{t} &=& \
^{i}G^{t,x;\alpha,\beta}_{t,\tau^{\delta}} [\varphi
(\tau^{\delta},X^{t,x;\alpha,\beta}_{\tau^{\delta}})1_{N_{\tau^{\delta}}^{t,i}=i}+
U_{m(i+1)}(\tau^{\delta},X^{t,x;\alpha,\beta}_{\tau^{\delta}})1_{N_{\tau^{\delta}}^{t,i}=m(i+1)}]-\varphi(t,x)\\
&\leq&\ ^{i}G^{t,x;\alpha,\beta}_{t,\tau^{\delta}}
[U_{N_{\tau^{\delta}}^{t,i}}(\tau^{\delta},X^{t,x;\alpha,\beta}_{\tau^{\delta}})]-\varphi(t,x).
\end{eqnarray*}
Consequently, from Theorem \ref{t5}
\begin{eqnarray*}
     \essinf_{\beta \in {\mathcal{B}}_{t, \tau^{\delta}}}\esssup_{u \in
{\mathcal{U}}_{t, \tau^{\delta}}} \
^{i}G^{t,x;\alpha,\beta}_{t,\tau^{\delta}}
[U_{N_{\tau^{\delta}}^{t,i}}(\tau^{\delta},X^{t,x;\alpha,\beta}_{\tau^{\delta}})]-\varphi(t,x)=0,
\end{eqnarray*}
we obtain that
     $\essinf_{\beta \in {\mathcal{B}}_{t, \tau^{\delta}}}\esssup_{\alpha \in
{\mathcal{A}}_{t, \tau^{\delta}}}Y^{1,\alpha,\beta}_{t}\leq 0.$
Then, from Lemma \ref{l3} it follows that
\begin{eqnarray*}
     \essinf_{\beta \in {\mathcal{B}}_{t, \tau^{\delta}}}\esssup_{\alpha \in
{\mathcal{A}}_{t, \tau^{\delta}}}Y^{2,\alpha,\beta}_{t}\leq
C\delta^{\frac{3}{2}}.
\end{eqnarray*}
Hence, there exists  $\beta^{*}\in \mathcal {B}_{t,\tau^{\delta}}$
(depending on $\delta$) such that
\begin{eqnarray}\label{e12}
\esssup_{\alpha \in {\mathcal{A}}_{t,
\tau^{\delta}}}Y^{2,\alpha,\beta^{*}}_{t}\leq
2C\delta^{\frac{3}{2}}.
\end{eqnarray}

From  $\beta^{*}\in \mathcal {B}_{t,\tau^{\delta}}$ it follows that
 there exists an increasing sequence of stopping times
$\{S_{n}(u)\}_{n\geq 1},$  for all $u\in\mathcal
{U}_{t,\tau^{\delta}}$, with  $t=S_{0}(u)\leq S_{1}(u) \leq\cdots
\leq S_{n}(u)\leq \cdots \leq \tau^{\delta}$ and  $\bigcup_{n\geq
1}\{S_{n}(u)=\tau^{\delta}\}=\Omega$, $\mathbb{P}$-a.s.,  such that,
for all $n\geq 1 $ and $u,u' \in \mathcal {U}_{t,\tau^{\delta}}$
with $u=u'$ on $[[t,S_{n-1}(u)]]$, it holds
 \begin{eqnarray*}
 S_{l}(u)=S_{l}(u'), 1\leq l \leq n,\ \mbox{and} \ \beta^{*}(u)=\beta^{*}(u'), \mbox{on}\ \ [[ t,S_{n}(u)]].
\end{eqnarray*}
Therefore, $S_{1}$ as well as $v^{*}:=\beta^{*}(u)$ on  $[[
t,S_{1}(u)]]$ do not dependent on the choice of $u\in\mathcal
{U}_{t,\tau^{\delta}}$. Let us define $u^{*}$ on $[[t,S_{1}(u)]]$ as
the process such that $u^{*}_{\cdot\wedge S_{1}}\in\mathcal
{U}_{t,\tau^{\delta}}$ and
\begin{eqnarray*}
F(s,x,0,0,0,u^{*}_{s},v^{*}_{s})=\sup_{u \in {\mathcal{U}}_{t,
\tau^{\delta}}}F(s,x,0,0,0,u,v^{*}_{s}), \ \ s\in[t,S_{1}(u^{*})].
\end{eqnarray*}

Putting $v^{*}:=\beta^{*}(u^{*}_{\cdot\wedge S_{1}})$ on
$]]S_{1}(u^{*}),S_{2}(u^{*})]]$ (Observe that $S_{2}(u^{*})$ only
dependents on $u^{*}_{\cdot\wedge S_{1}}$), let us define the
process $u^{*}$ on $[[t,S_{2}(u^{*})]]$ by $u^{*}_{(\cdot\wedge
S_{2}(u^{*}))\vee  S_{1}(u^{*})}\in\mathcal {U}_{t,\tau^{\delta}}$
such that
\begin{eqnarray*}
F(s,x,0,0,0,u^{*}_{s},v^{*}_{s})=\sup_{u \in
U}F(s,x,0,0,0,u,v^{*}_{s}), \ \ s\in[S_{1}(u^{*}),S_{2}(u^{*})].
\end{eqnarray*}
Therefore,
\begin{eqnarray*}
F(s,x,0,0,0,u^{*}_{s},\beta^{*}(u^{*}_{\cdot\wedge
S_{2}(u^{*})})_{s})=\sup_{u \in U} F(s,x,0,0,0,u,v^{*}_{s}), \ \
s\in[t,S_{2}(u^{*})].
\end{eqnarray*}
Iterating  the above argument we obtain $u^{*}$ on
$[[t,S_{\infty}[[$, $S_{\infty}:=\lim_{n\rightarrow\infty}\uparrow
S_{n}(u^{*})\leq \tau^{\delta}.$  Choosing an arbitrary $u_{0}\in
U$, we define
$u^{*}:=u^{*}1_{[t,S_{\infty}[}+u_{0}1_{[S_{\infty},\tau^{\delta}]}\in
\mathcal {U}_{t,\tau^{\delta}}.$  Since  $\bigcup_{n\geq
1}\{S_{n}(u^{*})=\tau^{\delta}\}=\Omega$  we can conclude
\begin{eqnarray*}
F(s,x,0,0,0,u^{*}_{s},\beta^{*}(u^{*}_{\cdot\wedge
S_{n}(u^{*})})_{s})&=&\sup_{u \in U}F(s,x,0,0,0,u,v^{*}_{s})\nonumber \\
&\geq&\inf_{v \in V}\sup_{u \in U} F(s,x,0,0,0,u,v), \ \
s\in[t,S_{n}(u^{*})].
\end{eqnarray*}
 for all $n\geq 1,$ and hence, that
\begin{eqnarray}\label{e37}
F(s,x,0,0,0,u^{*}_{s},\beta^{*}(u^{*})_{s})\geq\inf_{v \in U}\sup_{u
\in U}F(s,x,0,0,0,u,v), \ \ s\in[t,\tau^{\delta}].
\end{eqnarray}
On the other hand, defining  $\alpha^{*}(v)=u^{*}$ on
$[t,\tau^{\delta}]\times U_{t,\tau^{\delta}}$, we deduce from
(\ref{e12})
\begin{eqnarray}\label{e13}
2C\delta^{\frac{3}{2}}\geq Y^{2,\alpha^{*},\beta^{*}}_{t}=
Y^{2,u^{*},\beta^{*}(u^{*})}_{t}.
\end{eqnarray}
Let us consider the following BSDEs:
\begin{eqnarray*}\label{}
    \left \{\begin{array}{rl}
      -dY^{2,u^{*},\beta^{*}(u^{*})}_s =&\!\!\!\! F(s,x, Y^{2,u^{*},\beta^{*}(u^{*})}_s,
       H^{2,u^{*},\beta^{*}(u^{*})}_s, Z^{2,u^{*},\beta^{*}(u^{*})}_s,
       u^{*}_s,\beta^{*}(u^{*}_s))ds\\
               &-\lambda H^{2,u^{*},\beta^{*}(u^{*})}_sds-Z^{2,u^{*},\beta^{*}(u^{*})}_s dB_s
               -H^{2,u^{*},\beta^{*}(u^{*})}_s d\widetilde{N}_{s}, \\
     Y^{2,u,v}_{\tau^{\delta}}=&\!\!\!\!0,\qquad \qquad  s\in
     [t,\tau^{\delta}],
     \end{array}\right.
 \end{eqnarray*}
and
\begin{eqnarray}\label{e14}
    \left \{\begin{array}{rl}
      -dY_s^{\delta} =&\!\!\!\! 1_{[t,\tau^{\delta}]}(s)\Big(\inf_{v \in U}\sup_{u \in U}F(s,x,
      0,0,0,u,v)-L|Y_s^{\delta}|-L|Z_s^{\delta}|-(L+\lambda) H^{\delta}_s \Big)ds\\ &
      -Z^{\delta}_s dB_s -H^{\delta}_s d\widetilde{N}_{s},\\
     Y_{t+\delta}^{\delta}=&\!\!\!\!0,\qquad \qquad  s\in
     [t,t+\delta],
     \end{array}\right.
 \end{eqnarray}
where we denote by $L$ the Lipschitz constant of $F(s,x, y,0,0,u,v)$
with respect to $y$. Thanks to Lemma \ref{l8} and (\ref{e37}) we
conclude  $Y^{2,u^{*},\beta^{*}(u^{*})}_t\geq Y_t^{\delta}.$

We also consider the following equation:
\begin{eqnarray}\label{e110}
    \left \{\begin{array}{rl}
      -d\bar{Y}_s^{\delta} =&\!\!\!\! \Big(\inf_{v \in U}\sup_{u \in U}F(s,x,
      0,0,0,u,v)-L|Y_s^{\delta}| \Big)ds,\\
     Y_{t+\delta}^{\delta}=&\!\!\!\!0,\qquad \qquad  s\in
     [t,t+\delta],
     \end{array}\right.
 \end{eqnarray}
Then we get the following lemma.
\begin{lemma}\label{ll21}
For every $\delta \in (0,1)$,
 \begin{eqnarray*}
 |Y^{\delta}_t-\bar{Y}^{\delta}_t| \leq C\delta^{\frac{3}{2}},\ \  \mathbb{P}-a.s.,
 \end{eqnarray*}
 where the constant $C$ does not depend on  $\delta>0$.
\end{lemma}

From the above lemma, (\ref{e13}) and (\ref{e14}) it follows that
\begin{eqnarray*}\label{}
C\delta^{\frac{1}{2}}\geq
\frac{1}{\delta}\bar{Y}_t^{\delta}\rightarrow \inf_{v \in V}\sup_{u
\in U}F(t,x, 0,0,0,u,v)
\end{eqnarray*}
as $\delta\rightarrow 0.$ Therefore, $\inf_{v \in V}\sup_{u \in
U}F(t,x, 0,0,0,u,v)\leq 0.$  From the definition of $F$ it now
follows that $U=(U_{1}, U_{2})$ is a viscosity supersolution of the
system (\ref{e6}). We conclude the proof.
\end{proof}\vskip2mm

Let us give the proof of the Lemmas \ref{l2}, \ref{l3}, \ref{l4} and
\ref{ll21}, which is an adaption of those  in \cite{P1997} or
\cite{BHL2010} to our framework of games of the type "NAD strategy
against NAD strategy". We begin with the\vskip2mm

\noindent {\bf Proof of Lemma \ref{l2}}\hskip3mm  We notice that
$^{i}G^{t,x;u,v}_{s\wedge\tau^{\delta},\tau^{\delta}}
[\varphi(\tau^{\delta},X^{t,x;u,v}_{\tau^{\delta}})1_{N_{\tau^{\delta}}^{t,i}=i}+
U_{m(i+1)}(\tau^{\delta},X^{t,x;u,v}_{\tau^{\delta}})1_{N_{\tau^{\delta}}^{t,i}=m(i+1)}]$
is defined by the following BSDE:
   \begin{eqnarray}\label{e7}
   \left \{\begin{array}{rcl}
   -d\ ^{i}\widehat{Y}^{t,x; u, v}_s & = & f_{i}(s,X^{t,x; u, v}_s, \ ^{i}\widehat{Y}^{t,x; u, v}_s,
   \ ^{i}\widehat{H}^{t,x; u, v}_s,\  ^{i}\widehat{Z}^{t,x; u, v}_s,u_s, v_s) ds\\
   &&-\lambda\ ^{i}\widehat{H}^{t,x; u, v}_s ds-\ ^{i}\widehat{Z}^{t,x; u, v}_s dB_s
   -\ ^{i}\widehat{H}^{t,x; u, v}_s d\widetilde{N}_{s}
   , \quad s\in [t,\tau^{\delta}],\\
        ^{i}\widehat{Y}^{t,x; u, v}_{\tau^{\delta}} & = & \varphi(\tau^{\delta},
        X^{t,x;u,v}_{\tau^{\delta}})1_{N_{\tau^{\delta}}^{t,i}=i}+
    U_{m(i+1)}(\tau^{\delta},X^{t,x;u,v}_{\tau^{\delta}})1_{N_{\tau^{\delta}}^{t,i}=m(i+1)},
   \end{array}\right.
  \end{eqnarray}
through the following relation:
\begin{eqnarray*}
    ^{i}G^{t,x;u,v}_{s\wedge\tau^{\delta},\tau^{\delta}}
[\varphi(\tau^{\delta},X^{t,x;u,v}_{\tau^{\delta}})1_{N_{\tau^{\delta}}^{t,i}=i}+
U_{m(i+1)}(\tau^{\delta},X^{t,x;u,v}_{\tau^{\delta}})1_{N_{\tau^{\delta}}^{t,i}=m(i+1)}]
=\ ^{i}\widehat{Y}^{t,x; u, v}_s, \     s\in [t,\tau^{\delta}].
\end{eqnarray*}
On the other hand, by applying  It\^{o}'s formula to
$\overline{Y}_s$ (see the definition of $\overline{Y}$ in Lemma
\ref{l2}), we obtain
\begin{eqnarray}\label{e8}
   d \overline{Y}_s & = & \frac{\partial \varphi}{\partial s}(s,X^{t,x; u, v}_s)ds+
   (\nabla_{x}\varphi b)(s,X^{t,x; u, v}_s,u_s, v_s)ds+
   (\nabla_{x}\varphi \sigma)(s,X^{t,x; u, v}_s,u_s, v_s)dB_{s}\nonumber\\
   &&+\frac{1}{2}tr(\partial_{xx}\varphi \sigma\sigma^{*})(s,X^{t,x; u, v}_s,u_s, v_s)ds
   +(U_{m(i+1)}(s,X^{t,x;u,v}_{s})-\varphi
(s,X^{t,x;u,v}_s))dN_{s}.
\end{eqnarray}
From the definition of $\tau^{\delta}$ we have
\begin{eqnarray*}
\overline{Y}_{\tau^{\delta}}&=&\varphi
(\tau^{\delta},X^{t,x;u,v}_{\tau^{\delta}})+\int_{t}^{\tau^{\delta}}(U_{m(i+1)}(r,X^{t,x;u,v}_{r})-\varphi
(r,X^{t,x;u,v}_r))dN_{r}\\
&=&\varphi
(\tau^{\delta},X^{t,x;u,v}_{\tau^{\delta}})+(U_{m(i+1)}(\tau^{\delta},X^{t,x;u,v}_{\tau^{\delta}})-\varphi
(\tau^{\delta},X^{t,x;u,v}_{\tau^{\delta}}))\Delta N_{\tau^{\delta}}\\
&=&\varphi
(\tau^{\delta},X^{t,x;u,v}_{\tau^{\delta}})+(U_{m(i+1)}(\tau^{\delta},X^{t,x;u,v}_{\tau^{\delta}})-\varphi
(\tau^{\delta},X^{t,x;u,v}_{\tau^{\delta}}))1_{N_{\tau^{\delta}}^{t,i}=m(i+1)}\\
&=&\varphi(\tau^{\delta},X^{t,x;u,v}_{\tau^{\delta}})1_{N_{\tau^{\delta}}^{t,i}=i}+
    U_{m(i+1)}(\tau^{\delta},X^{t,x;u,v}_{\tau^{\delta}})1_{N_{\tau^{\delta}}^{t,i}=m(i+1)},
\end{eqnarray*}
and thus $\overline{Y}_{\tau^{\delta}}=\
^{i}\widehat{Y}_{\tau^{\delta}}^{t,x,u,v}.$
 Therefore, from the above equality, (\ref{e7}) and (\ref{e8}),  and the uniqueness of the solution of BSDE (\ref{e36}),
  it follows that
\begin{eqnarray*}
Y^{1,u,v}_{s\wedge \tau^{\delta}} =\
^{i}G^{t,x;u,v}_{s\wedge\tau^{\delta},\tau^{\delta}} [\varphi
(\tau^{\delta},X^{t,x;u,v}_{\tau^{\delta}})1_{N_{\tau^{\delta}}^{t,i}=i}+
U_{m(i+1)}(\tau^{\delta},X^{t,x;u,v}_{\tau^{\delta}})1_{N_{\tau^{\delta}}^{t,i}=m(i+1)}]
-\overline{Y}_{s\wedge\tau^{\delta}},
\end{eqnarray*}
and this allows to conclude the proof. \ $\Box$\vskip2mm

\noindent {\bf Proof of Lemma \ref{l3}}\hskip5mm  From (\ref{e22})
it follows that, for all $p\geq 2$, there exists
 some positive constant $C=C_p$\ such that
 \begin{eqnarray*}
 \mathbb{E} [\sup \limits_{t\leq s \leq t+\delta} |X^{t,x;u,v}_s -x|^p|{{\mathcal{F}}_t}] \leq
       C\delta(1+|x|^p),\ \mathbb{P}-a.s.,
 \end{eqnarray*}
\ uniformly in $u \in {\mathcal{U}}_{t, t+\delta}, v \in
{\mathcal{V}}_{t, t+\delta}$.  Let   $$ \varphi(s)=F(s,x,Y_s^{2, u,
v},H_s^{2, u, v},Z_s^{2, u, v},u_s,v_s) -F(s,X^{t,x,u,v}_s,Y_s^{2,
u, v},H_s^{2, u, v},Z_s^{2, u, v},u_s,v_s).$$
 Then, we have
 $$ |\varphi(s)|\leq C(1+|x|^2)(|X^{t,x;u,v}_s -x|+|X^{t,x;u,v}_s
 -x|^3),$$
  for $s\in [t, t+\delta], (t, x)\in [0, T)\times {\mathbb{R}}^n$,
  $u\in {\mathcal{U}}_{t, t+\delta}, v \in {\mathcal{V}}_{t, t+\delta}. $
    From  Lemma \ref{l1} it follows that
 \begin{eqnarray*}
      && \mathbb{E}[\int^{t+\delta}_t (|Y^{1,u,v}_s -Y^{2,u,v}_s|^2 +|Z^{1,u,v}_s -Z^{2,u,v}_s|^2)ds|{\mathcal{F}}_{t}]
        + \mathbb{E}[\int^{t+\delta}_t |H^{1,u,v}_{s} -H^{2,u,v}_{s}|^2ds|{\mathcal{F}}_{t}]\\
       && \leq C\mathbb{E}[\int^{t+\delta}_t \rho^2(|X^{t,x,u,v}_s -x|)
       ds|{\mathcal{F}}_{t}]\\
        && \leq  C \delta \mathbb{E}[\sup \limits_{t\leq s \leq t+\delta}\rho^2(|X^{t,x,u,v}_s
       -x|)|{\mathcal{F}}_{t}] \leq C\delta^2,
 \end{eqnarray*}
 where  $\rho(r) =(1+|x|^2)(r+r^3),\ r\geq 0. $\
Therefore, we have
 \begin{eqnarray*}
      && |Y^{1,u,v}_t -Y^{2,u,v}_t|  =|\mathbb{E}[(Y^{1,u,v}_t -Y^{2,u,v}_t )|{\mathcal{F}}_{t}]|  \\
       & = & |\mathbb{E}[\int^{t+\delta}_t \Big(F(s,X^{t,x,u,v}_s,Y^{1,u,v}_s,H^{1,u,v}_s,Z^{1,u,v}_s,u_s,
       v_s)\\
        &&\ \hskip2cm   -F(s,x,Y^{2,u,v}_s ,H^{2,u,v}_s,Z^{2,u,v}_s,u_s, v_s)
        -\lambda H^{1,u,v}_s +\lambda H^{2,u,v}_s\Big) ds\Big|{\mathcal{F}}_{t}]| \\
       & \leq & C\mathbb{E} [\int^{t+\delta}_t (\rho(|X^{t,x,u,v}_s -x|) +|Y^{1,u,v}_s -Y^{2,u,v}_s|
       +|Z^{1,u,v}_s -Z^{2,u,v}_s|)ds|{\mathcal{F}}_{t}]\\
      &&\ \hskip1cm +C\mathbb{E} [\int^{t+\delta}_t|H^{1,u,v}_s-H^{2,u,v}_s|ds|{\mathcal{F}}_{t}]  \\
       & \leq & C\mathbb{E}[\int^{t+\delta}_t\rho (|X^{t,x,u,v}_s -x|)ds|{\mathcal{F}}_{t}] +
       C\delta^{\frac{1}{2}} \{ \mathbb{E}[\int^{t+\delta}_t
            |Y^{1,u,v}_s-Y^{2,u,v}_s|^2ds|{\mathcal{F}}_{t}]^{\frac{1}{2}}\\
         &&   + \mathbb{E}[\int^{t+\delta}_t|Z^{1,u,v}_s -Z^{2,u,v}_s|^2ds|{\mathcal{F}}_{t}]^{\frac{1}{2}}
         + \mathbb{E}[\int^{t+\delta}_t|H^{1,u,v}_s -H^{2,u,v}_s|^2ds|{\mathcal{F}}_{t}]^{\frac{1}{2}}\}  \\
       & \leq & C\delta^{\frac{3}{2}}.
 \end{eqnarray*}
The desired result then follows. \ $\Box$\vskip2mm
 \noindent {\bf
Proof of Lemma \ref{l4}:} \hskip1mm Since $F(s, x, \cdot,
\cdot,\cdot, u, v)$\ has a linear
 growth in $(y, h,z)$, uniformly in $(s,x,u, v)$,    there exists a positive constant $C$\
 independent of $\delta,$\ $u$ and $ v,$ such that, for $s\in[t,t+\delta],$
\begin{eqnarray*}
 |Y^{2,u,v}_s|^2\leq C\delta,\ \mathbb{E}[\int_s^{t+\delta}|Z^{2,u,v}_r|^2dr|{\cal{F}}_s]\leq
C\delta,\ \mathbb{E}[\int_s^{t+\delta}|H^{2,u,v}_r|^2
dr|{\cal{F}}_s]\leq C\delta.
\end{eqnarray*}
 By virtue of  equation (\ref{e9}) we
have, for $s\in[t, t+\delta],$
\begin{eqnarray*}
 |Y^{2,u,v}_s|&\leq &\mathbb{E}[\int_s^{t+\delta}\Big(|F(r,x,{Y}^{2, u, v}_r,
 H^{2, u, v}_r, Z^{2, u, v}_r, u _r, v_r)|+\lambda|H^{2, u, v}_r|\Big)dr|{\cal{F}}_s]\\
 &\leq & C\mathbb{E}[\int_s^{t+\delta}\Big(1+ |x|^2+|{Y}^{2, u, v}_r|
 + |{H}^{2, u, v}_r|+|{Z}^{2, u, v}_r|\Big)dr|{\cal{F}}_s]\\
 &\leq & C\delta+C\sqrt{\delta}(\mathbb{E}[\int_s^{t+\delta}|Z^{2,u,v}_r|^2dr|{\cal{F}}_s])^{\frac{1}{2}}
 +C\sqrt{\delta}(\mathbb{E}[\int_s^{t+\delta}|H^{2,u,v}_r|^2dr|{\cal{F}}_s])^{\frac{1}{2}}\\
 &\leq & C\delta,\ \mathbb{P}\mbox{-a.s.}
\end{eqnarray*}
By applying  It\^o formula to $|Y^{2,u,v}_s|^2$ we conclude
\begin{eqnarray*}
\mathbb{E}[\int_t^{t+\delta}|Z^{2,u,v}_s|^2ds|{\cal{F}}_t]
+\mathbb{E}[\int_t^{t+\delta}|H^{2,u,v}_s|^2ds|{\cal{F}}_t]\leq
C\delta^2, \ \ \mathbb{P}\mbox{-a.s.}
\end{eqnarray*}
Therefore,
\begin{eqnarray*}
&&\mathbb{E}[\int_t^{t+\delta}|Y^{2,u,v}_s|ds|{\cal{F}}_t]+\mathbb{E}[\int_t^{t+\delta}|Z^{2,u,v}_s|ds|{\cal{F}}_t]
+\mathbb{E}[\int_t^{t+\delta}|H^{2,u,v}_s|ds|{\mathcal{F}}_{t}]\\
&&\leq
C\delta^2+\delta^{\frac{1}{2}}\{\mathbb{E}[\int_t^{t+\delta}|Z^{2,u,v}_s|^2ds|{\cal{F}}_t]\}^{\frac{1}{2}}
+C\delta^{\frac{1}{2}}\{\mathbb{E}[\int_t^{t+\delta}|H^{2,u,v}_s|^2ds|{\cal{F}}_t]\}^{\frac{1}{2}}\\
&&\leq C\delta^{\frac{3}{2}}, \ \ \mathbb{P}\mbox{-a.s.}
\end{eqnarray*}
 The proof is complete.  \ $\Box$\vskip2mm

\vskip2mm
 \noindent {\bf Proof of Lemma \ref{ll21}:} \hskip1mm
By using standard arguments of BSDEs we get the following estimate
\begin{eqnarray*}
\mathbb{E}[\sup_{s\in[t,t+\delta]}|Y^{\delta}_s-\bar{Y}^{\delta}_s|^{2}|{\mathcal{F}}_{t}]
+\mathbb{E}[\int_t^{t+\delta}|Z^{\delta}_s|^{2}ds|{\mathcal{F}}_{t}]
+\mathbb{E}[\int_t^{t+\delta}|H^{\delta}_s|^{2}ds|{\mathcal{F}}_{t}]
\leq C\delta^{2},\ \mathbb{P}-a.s.
\end{eqnarray*}

From equations (\ref{e14}) and (\ref{e110}) it follows that
\begin{eqnarray*}
 |Y^{\delta}_t-\bar{Y}^{\delta}_t|&\leq &
 C\mathbb{E}[\int_t^{\tau^{\delta}}\Big(|Y^{\delta}_s-\bar{Y}^{\delta}_s|
 +|Z^{\delta}_s|+|H^{\delta}_s|\Big)ds|{\cal{F}}_t]\\
 &&+ \mathbb{E}[\int_{\tau^{\delta}}^{t+\delta}\Big(\inf_{v \in U}\sup_{u \in U}F(s,x,
      0,0,0,u,v)+L|Y_s^{\delta}| \Big)ds|{\cal{F}}_t]\\
      &\leq &
 C\delta^{\frac{1}{2}}\mathbb{E}[\sup_{s\in[t,t+\delta]}|Y^{\delta}_s-\bar{Y}^{\delta}_s|^{2}|
 {\mathcal{F}}_{t}]^{\frac{1}{2}}
 +C\delta^{\frac{1}{2}}\mathbb{E}[\int_t^{t+\delta}
  |Z^{\delta}_s|^{2} ds|{\cal{F}}_t]^{\frac{1}{2}}\\
  && +C\delta^{\frac{1}{2}}\mathbb{E}[\int_t^{t+\delta}
  |H^{\delta}_s|^{2} ds|{\cal{F}}_t]^{\frac{1}{2}}+ \mathbb{E}[t+\delta-\tau^{\delta}|{\cal{F}}_t]\\
 &\leq & C\delta^{\frac{3}{2}},\ \mathbb{P}\mbox{-a.s.}
\end{eqnarray*}
 The proof is complete.  \ $\Box$

\end{document}